\documentclass{amsart}
\usepackage[utf8]{inputenc}
\usepackage{amssymb} 
\usepackage{amsmath}
\usepackage{amsthm}
\usepackage{mathrsfs}
\usepackage{mathtools}
\usepackage{stmaryrd}
\usepackage{graphicx}
\usepackage{tikz}
\usetikzlibrary{decorations.markings, patterns, calc}
\usepackage[outline]{contour}
\usepackage{enumitem}
\usepackage{wasysym}
\usepackage{subcaption}
\usepackage{setspace}

\makeatletter
\def\paragraph{\@startsection{paragraph}{4}	\z@\z@{-\fontdimen2\font} {\normalfont\itshape\bfseries}}
\makeatother

\theoremstyle{plain} 
\newtheorem{theorem}{Theorem}[section] % reset theorem numbering for each section
\newtheorem{proposition}[theorem]{Proposition}
\newtheorem{corollary}[theorem]{Corollary}

\newtheorem{lemma}[theorem]{Lemma}

\theoremstyle{definition}
\newtheorem{remark}[theorem]{Remark}

\newtheorem{definition}[theorem]{Definition}
\newtheorem{notation}[theorem]{Notation}

\def\mbb #1{\mathbb{#1}}

\def\mbf #1{\mathbf{#1}}
\def\mrm #1{\mathrm{#1}}
\def\Cal #1{\mathcal{#1}}
\def\mfr #1{\mathfrak{#1}}

\def\C{\mbb{C}}
\def\CP{\mbb{CP}}
\def\R{\mbb{R}}
\def\Z{\mbb{Z}}
\def\N{\mbb{N}}

\def\GL{\mrm{GL}}
\def\SL{\mrm{SL}}
\def\PSL{\mrm{PSL}}
\def\tr{\mrm{tr}}
\def\id{\mrm{id}}
\def\Aut{\mrm{Aut}}

\newcommand{\Sing}{\mathrm{Sing\,}}
\newcommand{\sminus}{\smallsetminus}

\newcommand{\uppie}{\scalebox{0.7}{\!$\mathrlap{\rotatebox[origin=c]{-45}{$\LEFTCIRCLE$}}\rotatebox[origin=c]{-135}{$\LEFTCIRCLE$}\!$}}
\newcommand{\downpie}{\scalebox{0.7}{\!$\mathrlap{\rotatebox[origin=c]{45}{$\LEFTCIRCLE$}}\rotatebox[origin=c]{135}{$\LEFTCIRCLE$}\!$}}

\newcommand*{\transp}[2][-3mu]{\ensuremath{\mskip1mu\prescript{\smash{\mathrm t\mkern#1}}{}{\mathstrut#2}}}%

\makeatletter
\newcommand*\bigcdot{\mathpalette\bigcdot@{.5}}
\newcommand*\bigcdot@[2]{\mathbin{\vcenter{\hbox{\scalebox{#2}{$\m@th#1\bullet$}}}}}
\makeatother

\def\sX{\mathsf{X}}
\def\sE{\mathsf{E}}

\def\sU{\mathsf{U}}

\def\sI{\mathsf{I}}
\def\sO{\mathsf{O}}

\def\pXzero{\tilde X_0}
\def\pXone{\tilde X_1}
\def\pXinf{\tilde X_\infty}
\def\pX{\tilde X}
\def\pF{\tilde F}
\def\pthetazero{\tilde\theta_0}
\def\pthetaone{\tilde\theta_1}
\def\pthetainf{\tilde\theta_\infty}
\def\pthetat{\tilde\theta_t}
\def\ptheta{\tilde\theta}
\begin{document}
%\doublespacing
%\onehalfspacing

\author{Martin Klime\v{s}}
%\address{Centre of Mathematics, Faculty of Science of University of Porto, Rua do Campo Alegre 687, 4169-007, Porto, Portugal.}
%\email{klmm@seznam.cz, martin.klimes@fc.up.pt}

%\author{Martin Klime\v s\thanks{		email : {\tt }}

\title{Wild monodromy of the Fifth Painlevé equation \\ and its action on wild character variety: \\ an approach of confluence}
%\thanks{}

\vskip1pt
\begin{abstract}
\noindent
The article studies the Fifth Painlevé equation and of the nonlinear Stokes phenomenon at its irregular singularity at infinity from the point of view of confluence from 
the Sixth Painlevé equation. This approach is developped separately on both sides of the Riemann--Hilbert correspondance.
On the side of the nonlinear Painlevé--Okamoto foliation the relation between the nonlinear monodromy group of Painlevé VI and the ``nonlinear wild monodromy pseudogroup'' of Painlevé V (that is the pseudogroup generated by nonlinear monodromy, nonlinear Stokes operators and nonlinear exponential torus) is explained in detail.
On the side of the corresponding linear isomonodromic problem, the ``wild'' character variety (the space of the linear monodromy and Stokes data) associated to Painlevé V is constructed through a birational transformation from the character variety (the space of the linear monodromy data) associated to Painlevé VI. 
This allows to transport the known description of the action of the nonlinear monodromy of Painlevé VI on its character variety to that of Painlevé V,
and to provide explicit formulas for the action of the ``nonlinear wild monodromy'' of Painlevé V on its character variety.

\end{abstract}

\maketitle

\pagestyle{myheadings}\markboth{}{Wild monodromy of Painlevé V}

\section{Introduction}

The six families of Painlevé equations $P_{I},\ldots,P_{VI}$ can be written in the form of non-autonomous Hamiltonian systems 
\begin{align*}%\label{eq:wm-hamiltonianvectorfield}
\frac{dq}{dt}=\, \frac{\partial H_{\bullet}(q,p,t)\!}{\partial p},\qquad
\frac{dp}{dt}=-\frac{\partial H_{\bullet}(q,p,t)\!}{\partial q},\qquad \bullet=I,\ldots,VI,
\end{align*}
the solutions of which define a foliation in the $(q,p,t)$-space (or better, in the Okamoto's semi-compactification of it)
which is transversely Hamiltonian with respect to the projection $(q,p,t)\mapsto t$.
By virtue of the Painlevé property, there are well-defined \emph{nonlinear monodromy} operators acting on the foliation by analytic continuation of solutions along loops in the $t$-variable. 
In the case of the sixth  Painlevé equation $P_{VI}$ the associated monodromy group carries a great deal of information about the foliation and hence about the equation.
The other Painlevé equations $P_{I},\ldots,P_{V}$ are obtained from $P_{VI}$ by a limiting process through confluences of singularities and other degenerations.
When different singularities merge, as is the case in the degeneration $P_{VI}\to P_V$, one loses some of the monodromy. 
It is known that the lost information should reappear in some way as a \emph{nonlinear Stokes phenomenon} at the confluent irregular singularity. 
Roughly speaking, this nonlinear phenomenon corresponds to the possibility to normalize the irregular singularity above certain sectors in the $t$-space (which together cover a full neighborhood of the singularity and above each of which the foliation is fairly rigid), and therefore to the 
existence of canonical 2-parameter families of solutions with well-behaved exponential asymptotics. This was proved originally in the works of Takano \cite{Ta, Ta90} and Yoshida \cite{Yo1, Yo1}, and recently also by Bittmann \cite{Bit}. 
As one passes from one sector to a neighboring one the family changes -- this change is then encoded by a \emph{nonlinear Stokes operator}.  
The ``good'' analogue of the nonlinear monodromy group for the equations $P_{I},\ldots,P_{V}$ is the, so called, ``\emph{wild monodromy pseudogroup}'', which is generated by both the nonlinear monodromy operators and a tuple of nonlinear Stokes operators  together with nonlinear exponential tori associated to each irregular singularity. 
This is a non-linear analogue of the wild monodromy group of meromorphic linear differential systems (or meromorphic connections over Riemann surfaces) of J.~Martinet and J.P.~Ramis \cite{MR2,Ra} whose closure is the differential Galois group \cite{Singer-Put}. 
The main goal of this paper is to describe in explicit terms the dynamics of this wild monodromy pseudogroup in the case of non-degenerated fifth Painlev\'e equation $P_V$.

In the paper \cite{Kl4}, the author has shown that in the case of confluence of two regular singularities to an irregular one in non-autonomous Hamiltonian systems, such as is the case of the degeneration $P_{VI}\to P_{V}$, both the nonlinear Stokes phenomenon and the wild monodromy pseudogroup can be reconstructed from a parametric family of limit \emph{wild monodromy operators} to which the usual nonlinear monodromy operators accumulate along certain discrete sequences of the parameter of confluence. We will use this confluence approach to describe the nonlinear Stokes phenomenon in the fifth Painlevé equation $P_V$
on the other side of the Riemann--Hilbert correspondence: as operators acting on a ``wild character variety''.

This provides one of the very first results in the general program of study of  ``wild monodromy" actions on the character varieties of isomonodromic deformations of linear differential systems that was sketched in \cite{PR}.

The Riemann--Hilbert correspondence in Painlevé equations is a well established and fruitful method based on the fact that these equations
govern the isomonodromic (and iso-Stokes) deformations of certain linear differential systems (or connections) with meromorphic coefficients on $\CP^1$. 
There are two kinds of such isomonodromic linear systems that we will consider here, both leading ultimately to the same description
(this is not surprising since the two systems are related one to the other by a middle convolution and a Laplace transform, or by a Harnad duality \cite{Har94, Maz, Bo05, HF}): 
\begin{itemize}[wide=0pt, leftmargin=\parindent]
	\item[i)] $2\!\times\!2$ traceless linear differential systems with 4 Fuchsian singular points at $0,1,t,\infty\in\CP^1$, governed by the equation $P_{VI}$ with independent variable $t$, and their confluent degeneration to systems with 1 nonresonant irregular singularity of Poincar\'e rank 1 and 2 Fuchsian ones, governed by $P_V$, \\[1pt]
	
	\item[ii)] $3\!\times\!3$ systems in Birkhoff normal form with an irregular singularity of Poincaré rank 1 at the origin (of eigenvalues $0,1,t$) and a Fuchsian singularity at the infinity governed by the equation $P_{VI}$ with independent variable $t$, and their degeneration through confluence of eigenvalues, governed by $P_V$.
\end{itemize}

The usual Riemann--Hilbert correspondence between linear systems and their generalized linear monodromy representations (consisting of monodromy and Stokes data),
can be interpreted as a map between the space of local solutions of the given Painlevé equation with fixed values of parameters, and the space of generalized monodromy representations with fixed local multipliers. This latter space is called the \emph{(wild) character variety} \cite{Bo14}.
In this setting, the Riemann--Hilbert correspondence conjugates the transcendental flow of the Painlevé equations to a locally constant flow on the corresponding  character variety \cite{IIS}, and the nonlinear monodromy of the sixth Painlevé equation $P_{VI}$ then corresponds to an action of the
a pure-braid group $\Cal B_3$ on the character variety \cite{Dub96,DM,I,I2}. 

Our goal is to use the confluence from $P_{VI}$ to $P_{V}$ in order to describe the action of the nonlinear wild monodromy pseudogroup of the foliation of $P_V$ on the 
associated wild character variety of $P_V$.
In order to do that, we will need to describe the confluence procedure on both sides of the Riemann--Hilbert correspondence: 
on the Painlevé--Okamoto foliation on one side, and on the linear isomonodromic problem and on the associated character variety on the other side. 
This obviously brings certain level technicality to this paper. 
The part dealing with the nonlinear foliation by solutions of a nonautonomous Hamiltonian system has been already treated in \cite{Kl4}. 
We will recall these results in the Sections~\ref{section:wm-0}, 
where we introduce the nonlinear monodromy group of $P_{VI}$ and the nonlinear wild monodromy pseudogroup of  $P_{V}$,
and~\ref{section:wm-1a}, where we explain the confluence and the relation between these nonlinear monodromy (pseudo)groups.
 
The main part of this paper is devoted to the study of the confluence in the associated linear isomonodromy problem, and of the dynamics on the wild character variety.
In Section~\ref{section:wm-1} we recall the usual approach to $P_{VI}$ as an equation governing the isomonodromic deformations of $2\!\times\! 2$ linear systems with four Fuchsian singularities, and some classical results concerning the geometry of the character variety of $P_{VI}$ and the braid group action on it.
In Section \ref{section:wm-2} we will study in detail the confluence $P_{VI}\to P_V$ through the viewpoint of isomonodromic deformations,
and show that the (wild) character variety of the equation Painlevé V is obtained from the character varieties of $P_{VI}$ 
by birational changes of coordinates (in fact, a blow-down) depending on the parameter of confluence.
This part is fundamentally based on the theory of confluence of singularities in linear systems of Lambert, Rousseau \& Hurtubise \cite{LR,HLR}. 
An alternative approach is provided in the Appendix using isomonodromic deformations of $3\!\times\! 3$ linear systems in Birkhoff normal form with an irregular singularity of Poincaré rank 1 at the origin and a Fuchsian singularity at infinity, where the degeneration from $P_{VI}$ to $P_V$ happens through a confluence of eigenvalues, the description of which is based on the author's work \cite{Kl1}. 

Using these birational transformations we transfer the explicitly known pure-braid actions from the character variety of $P_{VI}$ to that of $P_V$, and then push them to the limit using the aforementioned result on their accumulation along discrete sequences of the parameter of confluence. 
This leads to our main result Theorem~\ref{theorem:wm-main},
which gives explicit formulas for the action of the nonlinear monodromy, of the nonlinear Stokes operators and of the nonlinear exponential torus of $P_V$ on its wild character variety.

It is expected that the nonlinear wild monodromy pseudogroup described in this paper should have a natural interpretation in terms of a differential Galois theory (e.g. the differential Galois groupoid of Malgrange \cite{Mal, Cas}), perhaps in a similar way to \cite{CL}, and that our results could be eventually applied to construct and classify certain special type solutions of $P_V$ in an analogy with the construction and classification of algebraic solutions of $P_{VI}$ \cite{DM,Bo05,Bo10,LT}. However this is well beyond the scope of this paper.

While the main motivation of this paper is to describe in precise terms the confluence $P_{VI}\to P_V$, and to recover the nonlinear action of the wild monodromy pseudogroup, there are several smaller results that are obtained along the way which are worth of independent attention.
For example, it is known that the irregular singularity of $P_V$ has a special pole free solution, ``tronqu\'ee'' solution, on each of the two sectors of normalizations,
which correspond to the sectorial center manifold of the saddle-node singularity of the foliation. What is perhaps not known, is that this pair of sectorial center manifold solutions unfolds to a single solution in the confluent family of $P_{VI}$, characterized by its asymptotics at both of the two confluent singularities,
and which is pole free on certain ``unfolded sectorial'' domain attached to the two singularities. This solution, both in $P_{VI}$ and at the limit in $P_{V}$, corresponds through the Riemann--Hilbert correspondence to a point on the intersection of two lines on the character variety (which is a cubic surface containing up to 27 lines in case of $P_{VI}$, resp. up to 21 lines in case of $P_V$). These kind of solutions, which seem to be new, might be expected to play a special role in physics, similar to the one that the ``bi-tronqu\'ee'' solutions play. 
Another side result worth of mentioning are Propositions~\ref{prop:wm-lines} and ~\ref{prop:wm-linesV} which give explicit formulas for all the lines on the character varieties of $P_{VI}$ and $P_V$ and their interpretations in terms of the isomonodromic problem. 

Finally, let us remark that our confluence approach can be well extended to the other Painlevé equations (with increasing complexity of the description, the further the equation is from $P_{VI}$ in the degeneration process), and some of this shall be done in future works.
As of now, a general theory of confluence in linear systems has presently been developed only for non-resonant irregular singularities \cite{HLR}, therefore allowing to deal with only about half of the isomonodromic systems associated to Painlevé equations. 
However, in the case of traceless linear $2\times 2$-systems with resonant irregular singularities, this theory has a natural generalization along the lines of \cite{Kl1}: this is a subject of a paper in preparation by the author.

\subsection*{Acknowledgment.}
I am most grateful to Emmanuel Paul and Jean-Pierre Ramis for their interest in this work, for many inspiring discussions that helped to shape parts of this paper, and for their great hospitality during my stay in Toulouse.
I am also indebted to Frank Loray for introducing me to the problem, to Alexey Glutsyuk for his interest and encouragement,
and to Christiane Rousseau who taught me the techniques of confluence.

\goodbreak

\section{Nonlinear monodromy and Stokes phenomenon in $P_{VI}$ and $P_V$}\label{section:wm-0}

\subsection{The Painlevé equations.}
The Painlevé equations originated from the effort of Painlev\'e \cite{Painleve} and Gambier \cite{Gam} to classify
all second order ordinary differential equations of type
$q''=R(q',q,t)$, with $R$ rational,
possessing the so called  \emph{Painlevé property} which controls the ramification points of solutions:

\smallskip\noindent
\textit{%
\textbf{Painlevé property:}
{Each germ of a solution can be meromorphically continued along any path avoiding the singular points of the equation (fixed singularities)}.
In other words, solutions cannot have any other movable singularities other than poles.
}
\smallskip

Painlevé and Gambier \cite{Gam} produced a list of 50 canonical forms of equations to which any such equation can be reduced.
Aside of equations solvable in terms of classical special functions, 
the list contained six new families of equations, $P_{I},\ldots,P_{VI}$, whose general solutions provided a new kind of special functions.
In many aspects they may be regarded as nonlinear analogues of the hypergeometric equations \cite{IKSY}.  

In a modern approach to the Painlevé equations, following on Okamoto's works, the traditional families $P_{II}, P_{III}, P_{V}$ are further divided to subfamilies by specification of some redundant parameters,
and are classified according to the affine Weyl group of their B\"acklund symmetries \cite{OkaVI, NY, Sakai}, as well as according to the type of the isomonodromic problem they control \cite{OO06, PS, CMR15}.
The equation $P_{VI}$ is a mother equation for the other Painlevé equations, which can be obtained through degeneration and confluence following the diagram \cite{OO06}
\newlength{\arrowlength}\settowidth{\arrowlength}{$\searrow$}
\[\begin{matrix}
& & & & P_{III}^{D_6} &\hskip-9pt\to\hskip-9pt & P_{III}^{D_7} & \hskip-9pt\to\hskip-9pt & P_{III}^{D_8}\\[-2pt]
& & &\hskip-9pt\nearrow\hskip-9pt & &\hskip-9pt\nearrow\hskip-\arrowlength\searrow\hskip-9pt & &\hskip-9pt\searrow\hskip-9pt &   \\[-2pt]
P_{VI}& \hskip-9pt\to\hskip-9pt & P_V & \hskip-6pt\to\hskip-6pt & P_V^{deg} & & P_{II}^{JM} & \hskip-6pt\to\hskip-6pt & \hskip-6pt P_I\\[-2pt]
& & &\hskip-9pt\searrow\hskip-9pt & &\hskip-9pt\searrow\hskip-\arrowlength\nearrow\hskip-9pt & &\hskip-8pt\nearrow\hskip-9pt&\\[-2pt]
& & & & P_{IV}& \hskip-9pt\to\hskip-9pt & P_{II}^{FN} & &
\end{matrix}
\]
according to which also the associated isomonodromy problems degenerate.
A good understanding of the degeneration procedures should allow to transfer information along the diagram.
The main obstacle is that as the nature of the singularities changes at the limit, it causes the naive limit of most local objects to diverge -- this is a common rule in  confluence problems. One therefore needs to find for each of the arrows in the above diagram some unified description that allows to deal with this divergence.
This article studies some aspects of the confluence $P_{VI}\to P_V$ through the Riemann--Hilbert correspondence.

Each of the Painlevé equations is equivalent to a time dependent \emph{Hamiltonian system}
\begin{align}\label{eq:wm-hamiltonianvectorfield}
\frac{dq}{dt}=\, \frac{\partial H_{\bullet}(q,p,t)\!}{\partial p},\qquad
\frac{dp}{dt}=-\frac{\partial H_{\bullet}(q,p,t)\!}{\partial q},\qquad \bullet=I,\ldots,VI,
\end{align}
from which it is obtained by reduction to the $q$-variable  \cite{Oka1}.

The general form of the sixth Painlevé equation is \cite{JM}
\begin{equation*}%\label{eq:wm-PVI}
\begin{split}
P_{VI}(\vartheta):\ \ 
q''&=\frac{1}{2}\Big(\frac{1}{q}+\frac{1}{q-1}+\frac{1}{q-t}\Big)(q')^2-
\Big(\frac{1}{t}+\frac{1}{t-1}+\frac{1}{q-t}\Big)q'\\
&\ +\frac{q(q-1)(q-t)}{2\,t^2(t-1)^2}
\Big[(\vartheta_\infty-1)^2-\vartheta_0^2\frac{t}{q^2}+\vartheta_1^2\frac{(t-1)}{(q-1)^2}+(1-\vartheta_t^2)\frac{t(t-1)}{(q-t)^2}\Big],
\end{split}
\end{equation*}
where $()'=\frac{d}{dt}$, and where $\vartheta=(\vartheta_0,\vartheta_t,\vartheta_1,\vartheta_\infty)\in\C^4$ are complex constants related to the eigenvalues of the associated isomonodromic problem \eqref{eq:wm-A}.
The Hamiltonian function $H_{VI}(q,p,t)$ of its associated Hamiltonian system \eqref{eq:wm-hamiltonianvectorfield} is given by
\begin{equation}\label{eq:wm-H_VI}
H_{VI}=\frac{q(q\!-\!1)(q\!-\!t)}{t(t\!-\!1)}\Big[p^2-\Big(\frac{\vartheta_0}{q}+\frac{\vartheta_1}{q\!-\!1}+\frac{\vartheta_t\!-\!1}{q\!-\!t} \Big)p 
 + \frac{(\vartheta_0\!+\!\vartheta_1\!+\!\vartheta_t\!-\!1)^2-(\vartheta_\infty\!-\!1)^2}{4q(q\!-\!1)}\Big].
%\begin{split}
%H_{VI}=\frac{1}{t(t-1)}\Big[&
%q(q\!-\!1)(q\!-\!t)p^2-\Big(\vartheta_0(q\!-\!1)(q\!-\!t)+\vartheta_1 q(q\!-\!t)+(\vartheta_t\!-\!1)q(q\!-\!1) \Big)p \\
%& + \tfrac{(\vartheta_0+\vartheta_1+\vartheta_t-1)^2-(\vartheta_\infty-1)^2}{4}(q-t)\Big].
%\end{split}
\end{equation}
The Hamiltonian system of $P_{VI}$ has three simple (regular) singular points on the Riemann sphere $\CP^1$
at $t=0,1,\infty$.

The non-degenerate \emph{fifth Painlevé equation} $P_V$ (more precisely the fifth Painlevé equation with a parameter $\eta_1=-1$, the general form is obtained by scaling $t\mapsto -\frac{t}{\eta_1}$) is
\begin{equation*}%\label{eq:wm-PV}
P_{V}(\tilde{\vartheta}):\ \ 
q''=\Big(\frac{1}{2q}+\frac{1}{q\!-\!1}\Big)(q')^2-
\frac{1}{\tilde t}q'+
\frac{(q\!-\!1)^2}{2\tilde t^2}\Big((\vartheta_\infty\!-1)^2 q-\frac{\vartheta_0^2}{q}\Big)+\tilde\vartheta_1\frac{q}{\tilde t}-\frac{q(q\!+\!1)}{2(q\!-\!1)},
\end{equation*}
where $()'=\frac{d}{d\tilde t}$ and $\tilde{\vartheta}=(\vartheta_0,\tilde{\vartheta_1},\vartheta_\infty)$, is obtained from $P_{VI}$ as a limit $\epsilon\to 0$ after the change of the independent variable and a substitution of the parameters
\begin{equation}\label{eq:wm-t}
t= 1+\epsilon\tilde t,\qquad \vartheta_t= \frac{1}{\epsilon},\quad \vartheta_1= -\frac{1}{\epsilon}+\tilde\vartheta_1,
\end{equation}
which sends the three singularities to $\tilde t=-\frac{1}{\epsilon},0,\infty$.
At the limit, the two simple singular points $-\frac{1}{\epsilon}$ and $\infty$ merge into a double (irregular) singularity at the infinity.

The change of variables \eqref{eq:wm-t}, changes the function $\epsilon\cdot H_{VI}$ to
\begin{equation*}
\begin{split}
H_{VI}^{conf}=\frac{q(q\!-\!1)(q\!-\!1\!-\!\epsilon\tilde t)}{\tilde t(1\!+\!\epsilon\tilde t)}\Big[p^2-&\Big(\frac{\vartheta_0}{q}+\frac{\tilde\vartheta_1\!-\!1}{q\!-\!1}+ \frac{(1\!-\!\epsilon)\tilde t}{(q\!-\!1)(q\!-\!1\!-\!\epsilon\tilde t)} \Big)p\\ 
+ &\frac{(\vartheta_0\!+\!\tilde\vartheta_1\!-\!1)^2\!-\!(\vartheta_\infty\!-\!1)^2}{4q(q\!-\!1)}\Big],
\end{split}
\end{equation*}
and the Hamiltonian system to
\begin{align}\label{eq:wm-PVIepsilon}
\frac{dq}{d\tilde t}=\, \frac{\partial H_{VI}^{conf}(q,p,\tilde t)}{\partial p},\qquad
\frac{dp}{d\tilde t}=-\frac{\partial H_{VI}^{conf}(q,p,\tilde t)}{\partial q},
\end{align}
whose limit $\epsilon\to 0$ is a Hamiltonian system of $P_V$, $H_V=\lim_{\epsilon\to 0}H_{VI}^{conf}$.
%\begin{equation*}
%H_{V}=\frac{1}{t}\Big[q(q-1)^2p^2-\Big(\vartheta_0(q-1)^2+(\tilde\vartheta_1-1) q(q-1)+tq \Big)p + \frac{1}{4}\Big((\vartheta_0+\tilde\vartheta_1-1)^2-(\vartheta_\infty-1)^2\Big)(q-1)\Big].
%\end{equation*}

\subsection{Nonlinear monodromy of $P_{VI}$.}

Consider the foliation in the $(q,p,t)$-space given by the solutions of the Hamiltonian system of $P_{j}(\vartheta)$, $j=I,\ldots,VI$.
As general solutions may have many poles the flow of $P_{j}$ on $\C^2\times(\CP^1\sminus\Sing(P_j))$, where $\Sing(P_j)\subseteq \{0,1,\infty\}$ is the set of fixed singularities of $P_j(\vartheta)$, is not complete. 
Okamoto \cite{Oka2} has constructed a semi-compactification 
$\Cal M_{j}(\vartheta)$ of this space in form of a fibration over $\CP^1\sminus\Sing(P_j)$ (corresponding to the projection $(q,p,t)\mapsto t$) on which the foliation is analytic and transverse to the fibres.
We will skip the details of this construction here as we won't need them. We will denote 
\begin{equation}\label{eq:wm-Okamotospace}
	\Cal M_{j,t}(\vartheta):= \text{\emph{the Okamoto space of initial conditions of $P_j(\vartheta)$}},
\end{equation} 
the fiber of $\Cal M_{j}(\vartheta)$ above a point $t\in\CP^1\sminus\Sing(P_j)$. 
It is a complex surface sitting inside a compact rational surface as a complement of some anti-canonical divisor \cite{Oka2,Sakai},
and endowed with a symplectic structure given by the standard symplectic form
\begin{equation}\label{eq:wm-omega}
\omega=dq\wedge dp,
\end{equation}
in the local coordinate $(q,p)$.

The Painlevé property of $P_{j}$ means that for each path $t_0\xrightarrow{\gamma} t_1$ in the $t$-space $\CP^1\sminus\Sing(P_j)$ the flow induces a symplectomorphisms
$\Cal M_{j,t_0}(\vartheta)\to \Cal M_{j,t_1}(\vartheta)$ between the fibers corresponding  to analytic continuation of the solutions  of \eqref{eq:wm-hamiltonianvectorfield} along the path (see Figure \ref{figure:wm-1}).
In particular, for any given base-point $t_0\in\CP^1\sminus\Sing(P_j)\}$ the loops $\gamma\in \pi_1(\CP^1\sminus\Sing(P_j),t_0)$
induce a nonlinear \emph{monodromy} action which is a representation of the fundamental group of the base space into the group of symplectomorphisms of $\Cal M_{j,t_0}(\vartheta)$,
\begin{align*}
\pi_1(\CP^1\sminus\Sing(P_j),t_0)&\to\Aut_\omega(\Cal M_{j,t_0}(\vartheta)),\\
\gamma\mapsto \mfr M_\gamma(\cdot,t_0),
\end{align*}
analytically depending on $t_0$.

\begin{figure}[t]
\centering
\includegraphics[width=0.6\textwidth]{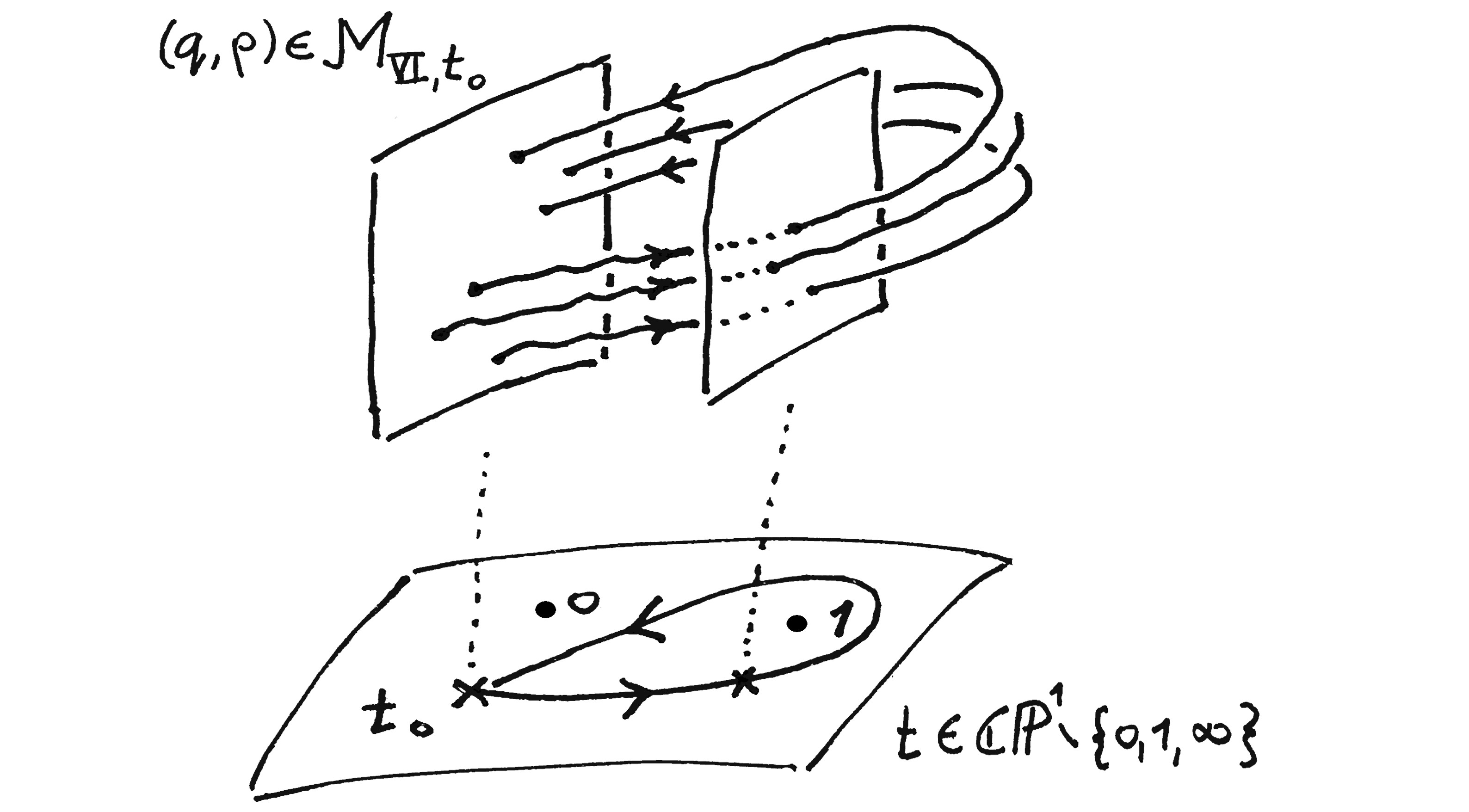}
\caption{The Painlevé flow and its monodromy in case of $P_{VI}$.}
\label{figure:wm-1}
\end{figure}

In case of $P_{VI}$ this nonlinear monodromy 
$\pi_1(\CP^1\sminus\{0,1,\infty\},t_0)\to\Aut_\omega(\Cal M_{VI,t_0}(\vartheta))\, $
carries (together with the analytic invariants of the singular fibers $t=0,1,\infty$ which are determined by the parameter $\vartheta$) a complete information about the foliation and thus about the equation. 
For example, algebraic solutions of $P_{VI}$
correspond to finite orbits of the monodromy action (see e.g. \cite{DM, Bo10, LT})).
The nonlinear monodromy also plays an important role in the differential Galois theory:
the Galois groupoid of Malgrange of the foliation is ``generated'' by the monodromy. 
Cantat and Loray \cite{CL} have showed the irreducibility of $P_{VI}$ in the sense of Malgrange (i.e. maximality of the Malgrange--Galois groupoid),
which then implies transcendentness of general solutions \cite{Cas}, by studying the action of this monodromy.
The important fact on which these studies are based is that  the nonlinear monodromy of $P_{VI}$ has a well known explicit representation of as a braid group action on a certain $\SL_2(C)$-character variety (more about this in Section~\ref{section:wm-1}).

On the other hand, for the other Painlevé equations, $P_I,\ldots,P_V$, the nonlinear monodromy operator doesn't carry enough information 
(in particular it is trivial for $P_I,P_{II}^{FN},P_{II}^{JM}$ and $P_{IV}$).
In case of $P_V$, the reason for this is that in the confluence $P_{VI}\to P_V$, the merging of the two singularities $-\frac{1}{\epsilon}, \infty$ 
in the confluent family \eqref{eq:wm-PVIepsilon} of $P_{VI}$ into a single singularity $\infty$ of $P_V$ means that an important part of the monodromy group is lost and therefore also the information carried by it.
We shall see that this lost information reappears in the nonlinear Stokes phenomenon at the irregular singularity.

\subsection{Nonlinear Stokes phenomenon of $P_{V}$.}
In the local coordinate $x=\tilde t^{-1}$ near $\tilde t=\infty$, the Hamiltonian system \eqref{eq:wm-hamiltonianvectorfield} of $P_V$ has the form
\begin{equation}\label{eq:wm-PVstokes} 
\mbox{\scriptsize$
\begin{aligned}
x^2\tfrac{d}{dx} q &= x\vartheta_0 + \big(1-x(2\vartheta_0\!+\!\tilde\vartheta_1\!-\!1)\big)q +x\Big((\vartheta_0\!+\!\tilde\vartheta_1\!-\!1)q^2-2pq+4pq^2-2pq^3\Big), \\[6pt]
x^2\tfrac{d}{dx}p &= x\tfrac{(\vartheta_0+\tilde\vartheta_1-1)^2-(\vartheta_\infty-1)^2}{4}-\big(1-x(2\vartheta_0\!+\!\tilde\vartheta_1\!-\!1)\big)p +x\Big(\!-\!2(\vartheta_0\!+\!\tilde\vartheta_1\!-\!1)pq+p^2-4p^2\!q+3p^2\!q^2\Big),\hskip-6pt
\end{aligned}$}
\end{equation}
%\begin{equation}\label{eq:wm-PVstokes}
% x^2\tfrac{d}{dx}\!\left(\begin{smallmatrix} q \\[6pt] p \end{smallmatrix}\right)=
% x\left(\begin{smallmatrix} \vartheta_0 \\[2pt] \!\!\!\frac{(\vartheta_0+\tilde\vartheta_1-1)^2-(\vartheta_\infty-1)^2}{4}\!\!\!
% \end{smallmatrix}\right)+
% \bigg(\begin{smallmatrix} \! 1-(2\vartheta_0+\tilde\vartheta_1-1) x \hskip-18pt & \hfill \ 0\\[6pt] 0\ \hfill  & \hskip-18pt -1+(2\vartheta_0+\tilde\vartheta_1-1) x\! \end{smallmatrix}\bigg)\!\!\left(\begin{smallmatrix} q \\[6pt] p \end{smallmatrix}\right)
% +x\left(\begin{smallmatrix} (\vartheta_0+\tilde\vartheta_1-1)q^2-2pq+4pq^2-2pq^3 \\[6pt]\!\! -2(\vartheta_0+\tilde\vartheta_1-1)pq+p^2-4p^2q+3p^2q^2\!\!\end{smallmatrix}\right)\!,
%\end{equation}
with an irregular singularity at $x=0$.

\begin{theorem}[Takano \cite{Ta, Ta90}, Shimomura \cite{Shi,Shi2}, Yoshida \cite{Yo2}]\label{theorem:wm-normalization0}~
\begin{enumerate}[label=(\roman*), wide=0pt, leftmargin=\parindent]	
\item \textbf{Formal normalization:}
The above system \eqref{eq:wm-PVstokes} can be brought to a formal normal form
\begin{equation}\label{eq:wm-v1v2}
 x^2\tfrac{d}{dx}u=
 \big(1-(2\vartheta_0+\tilde\vartheta_1-1)x+4u_1u_2x\big)\left(\begin{smallmatrix} 1&0 \\[4pt] 0&-1 \end{smallmatrix}\right)u,
 \qquad u=\left(\begin{smallmatrix} u_1 \\[4pt]  u_2 \end{smallmatrix}\right),
% +\text{higher order terms in }q,p,
\end{equation}
by means of a formal transversely symplectic (w.r.t. the canonical forms \eqref{eq:wm-omega} and $du_1\wedge du_2$) change of coordinates 
\[\left(\begin{smallmatrix} q \\[3pt] p\end{smallmatrix}\right)=\hat{\mbf\Psi}(u,x,0)=\sum_{k\geq 0}\psi^{(k0)}(u)x^k,\]
where $\psi^{(k0)}(u)$ are analytic on some polydisc $\sU=\{|u_1|,|u_2|<\delta_u\}$, $\delta_u>0$.

\item
\textbf{Sectoral normalization:}
The formal series $\hat{\mbf\Psi}$ of $x$ is divergent but Borel summable, with a pair of Borel sums $\mbf\Psi^{\uppie}(u,x,0)$ and $\mbf\Psi^{\downpie}(u,x,0)$
defined respectively above the sectors
\[x\in\begin{cases}
	\sX^{\uppie}(0)=\{|\arg x -\tfrac{\pi}{2}|<\pi-\eta, |x|<\delta_x\},\\
	\sX^{\downpie}(0)=\{|\arg x +\tfrac{\pi}{2}|<\pi-\eta, |x|<\delta_x\},
\end{cases}
\]
for some $0<\eta<\tfrac{\pi}{2}$ arbitrarily small and some $\delta_x>0$ (depending on $\eta$), and $u\in\sU$.
The sectorial transformations 
$\left(\begin{smallmatrix} q \\[3pt] p\end{smallmatrix}\right)=\mbf\Psi^\bullet(u,x,0)$, $\bullet=\uppie,\downpie$,
bring the system \eqref{eq:wm-PVstokes} to its formal normal form \eqref{eq:wm-v1v2}.

In particular, the formal transformation $\hat{\mbf\Psi}(u,x,0)$ and the sectorial transformations $\mbf\Psi^\bullet(u,x,0)$ satisfy the same 
$(\tfrac{\partial}{\partial u}, \tfrac{\partial}{\partial x})$-differential relations over the field of germs of meromorphic functions.
\end{enumerate}
\end{theorem}

\begin{figure}[t]
	\centering
	\includegraphics[width=0.35\textwidth]{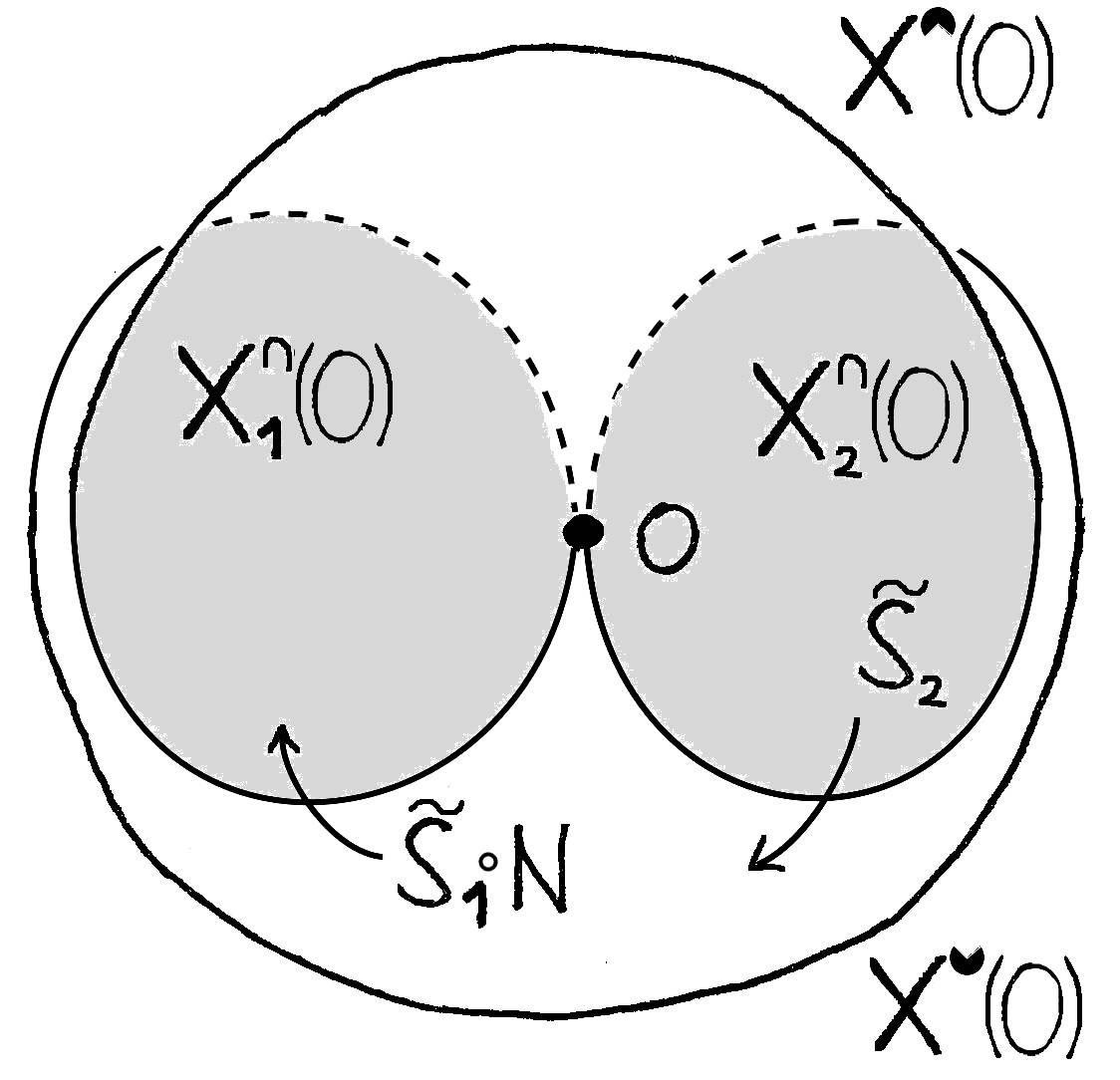}
	\caption{The sectorial domains $\sX^{\protect\uppie}(0)$ and $\sX^{\protect\downpie}(0)$ in the $x$-coordinate and the connecting transformations between the canonical solutions 
		$\left(q^{\protect\uppie}, p^{\protect\uppie}\right)\!(x,0;c)$ and 	$\left(q^{\protect\downpie}, p^{\protect\downpie}\right)\!(x,0;c)$ 
		on the left $\sX^{\protect\cap}_1(0)$ and right $\sX^{\protect\cap}_2(0)$ intersection sectors.}
	\label{figure:wm-sectors0}
\end{figure}

\paragraph{Canonical 2-parameter family of solutions:}
The system \eqref{eq:wm-v1v2}, which is Hamiltonian for the time-$x$-dependent Hamiltonian function
$\frac{\big(1-(2\vartheta_0+\tilde\vartheta_1-1)x\big)u_1u_2+2\big(u_1u_2\big)^2}{x^2}$
with respect to the standard symplectic form $du_1\wedge du_2$,
has a \emph{canonical 2-parameter family of solutions}
\begin{equation}\label{eq:wm-u0}
\begin{aligned}
u_1(x,0;c)&=c_1\,e^{-\tfrac{1}{x}}x^{-(2\vartheta_0+\tilde\vartheta_1-1)+4c_1c_2},\\
u_2(x,0;c)&=c_2\,e^{\,\tfrac{1}{x}}x^{(2\vartheta_0+\tilde\vartheta_1-1)-4c_1c_2},
\end{aligned}
\qquad c=\transp(c_1,c_2)\in\C^2,
\end{equation}
and an \emph{analytic first integral} 
\begin{equation}\label{eq:wm-h}
 h=u_1u_2=c_1c_2.
\end{equation}
We call the parameter $c=\transp(c_1,c_2)$ an \emph{initial condition}.

Let $u^{\uppie}(x,0;c)$, resp. $u^{\downpie}(x,0;c)$, be fixed branches of \eqref{eq:wm-u0} restricted to $\sX^{\uppie}(0)$, resp. $\sX^{\downpie}(0)$,
then we define\footnote{The notation $\mbf\Psi^\bullet(\cdot,x,0)$ stands for the function $u\mapsto \mbf\Psi^\bullet(u,x,0)$,
i.e. $\mbf\Psi^\bullet(\cdot,x,0)\circ u^\bullet(x,0;c)$ denotes the substitution $\mbf\Psi^\bullet(u,x,0)\big|_{u= u^\bullet(x,0;c)}$.}
\begin{equation}\label{eq:wm-qp0}
\left(\begin{smallmatrix} q^\bullet \\[3pt] p^\bullet\end{smallmatrix}\right)\!(x,0;c)=\mbf\Psi^\bullet(\cdot,x,0)\circ u^\bullet(x,0;c),\qquad \bullet=\uppie,\downpie,
\end{equation}
as the corresponding \emph{canonical 2-parameter family of solutions} to the Hamiltonian Painlevé system \eqref{eq:wm-PVstokes}.

\medskip

\paragraph{Sectoral center manifold solutions:}
\begin{corollary}\label{cor:wm-centermanifold0}
	The system \eqref{eq:wm-PVstokes} has a unique bounded analytic solution on each of the two sectors $\sX^{\bullet}(0)$, $\bullet=\uppie,\downpie$, given by 
	$\left(\begin{smallmatrix} q^\bullet \\[3pt] p^\bullet\end{smallmatrix}\right)\!(x,0;0)=\mbf\Psi^\bullet(0,x,0)$, corresponding to the initial condition $c=0$.
\end{corollary}

We call these solutions the ``\emph{sectorial center manifold}'' solutions, since they correspond to a sectorial center manifold of the saddle-node singularity of the foliation. Each of them can be characterized as the pole-free solution on its respective sector. These solutions are also known as the ``\emph{bi-tronqu\'ee}'' (double-truncated) solutions of $P_V$ \cite{AK}.

\medskip

\paragraph{Nonlinear exponential torus:} 
\begin{definition}
A (fibered) \emph{symmetry} of the system \eqref{eq:wm-v1v2} is a symplectic transformation $u\mapsto\phi(u,x)$ that preserves the system.
\end{definition}

\begin{proposition}[Symmetries of the formal normal form \hbox{\protect\cite[Proposition~31]{Kl4}}]\label{prop:wm-exptorus}
Any symmetry of \eqref{eq:wm-v1v2} that is bounded and analytic above one of the sectors $x\in\sX^\bullet(0)$, $\bullet=\uppie,\downpie$, $u\in\sU$, 
is in fact independent of $x$. It is given by the time-1 flow 
\begin{equation}\label{eq:wm-Ta}
\mbf T_{\alpha}:u \mapsto \left(\begin{smallmatrix}  e^{\alpha(u_1u_2)} & 0\\[3pt] 0&  e^{-\alpha(u_1u_2)} \end{smallmatrix}\right)u,\qquad
%c \mapsto \left(\begin{smallmatrix}  e^{a(c_1c_2)} & \\[3pt] &  e^{-a(c_1c_2)} \end{smallmatrix}\right)c,
\end{equation}
of a Hamiltonian vector field 
\begin{equation}\label{eq:wm-xia}
\xi=\alpha(u_1u_2)\big(u_1\tfrac{\partial}{\partial u_1}-u_2\tfrac{\partial}{\partial u_2}\big),\qquad \text{with }\ \alpha(h)\ \text{ an analytic germ}.
\end{equation}
The Lie group of these symmetries is commutative and connected.
\end{proposition}

Following the terminology of the differential Galois theory of linear systems, the Lie group of symmetries \eqref{eq:wm-Ta} is called the \emph{nonlinear exponential torus}.
Its Lie algebra of infinitesimal symmetries \eqref{eq:wm-xia} is called the \emph{nonlinear infinitesimal torus}.

\begin{corollary}
The normalizing sectorial transformations $\hat{\mbf\Psi}$ and $\mbf\Psi^\bullet$ of Theorem~\ref{theorem:wm-normalization0} are unique up to a right composition with the same analytic symmetry $\mbf T_\alpha(u)$.
\end{corollary}

 The nonlinear \emph{exponential torus} \eqref{eq:wm-Ta}  acts on the  Painlevé system \eqref{eq:wm-PVstokes} by the sectorial isotropies
 \[\mfr T_\alpha^{\bullet}(q,p,x):=\mbf\Psi^{\bullet}(\cdot,x,0)\circ \mbf T_\alpha(u),\qquad\bullet=\uppie,\downpie,\]
 given by the time-1-flow of the vector field which is the sectorial pullback of \eqref{eq:wm-xia} by $(\mbf\Psi^{\bullet})^{\circ(-1)}$.
 Its right action on the  canonical solutions
 $\mfr T_\alpha^{\bullet}(\cdot,x)\circ \left(\!\begin{smallmatrix} q^\bullet \\[0pt] p^\bullet\end{smallmatrix}\!\right)\!(x,0;c)=
 \left(\!\begin{smallmatrix} q^\bullet \\[0pt] p^\bullet\end{smallmatrix}\!\right)\!(x,0;\cdot)\circ \mbf T_\alpha(c)$ is given by the corresponding action of 
 $\mbf T_\alpha$ on the initial condition parameter $c$,
  \begin{equation}\label{eq:wm-Tc}
 	\mbf T_\alpha: c \mapsto \left(\begin{smallmatrix}  e^{\alpha(c_1c_2)} & 0\\[3pt] 0&  e^{-\alpha(c_1c_2)} \end{smallmatrix}\right)c,\qquad
 	\alpha\in\Cal O(\C,0).
 \end{equation}
 In particular, the nonlinear
 \emph{formal monodromy} (monodromy of the formal solution $\hat{\mbf\Psi}(\cdot,x,0)\circ u(x,0;c)$), acting on the initial condition $c$ as 
 \begin{equation}\label{eq:wm-nlN}
 	\mbf N:=\mbf T_{2\pi i(-2\vartheta_0-\tilde\vartheta_1+1+4c_1c_2)}:
 c \mapsto \left(\begin{smallmatrix}  e^{2\pi i(-2\vartheta_0-\tilde\vartheta_1+4c_1c_2)} & \\[3pt] &  e^{-2\pi i(-2\vartheta_0-\tilde\vartheta_1+4c_1c_2)} \end{smallmatrix}\right)c,
 \end{equation}
 is an element of the exponential torus.

\medskip

\paragraph{Nonlinear Stokes operators:}
Let
\begin{equation*}
\sX^\cap_1(0)=\{|\arg x|< \tfrac{\pi}{2}-\eta, |x|<\delta_x\},\quad\text{and}\quad
\sX^\cap_2(0)=\{|\arg x+\pi|< \tfrac{\pi}{2}-\eta, |x|<\delta_x\},
\end{equation*}
be the left and right intersection sectors of the overlapping sectors $\sX^{\uppie}(0)$, $\sX^{\downpie}(0)$ (see Figure~\ref{figure:wm-sectors0}).
The two transition maps
\begin{equation}\label{eq:wm-stokes}
\begin{aligned}
\mfr S_1(q,p,x,0)  &=\mbf\Psi^{\downpie}(\cdot,x,0)\circ (\mbf\Psi^{\uppie})^{\circ(-1)}(q,p,x,0),\qquad x\in\sX^\cap_1(0)\subset\sX^{\uppie}(0),\\
\mfr S_2(q,p,x,0)&=\mbf\Psi^{\uppie}(\cdot,x,0)\circ (\mbf\Psi^{\downpie})^{\circ(-1)}(q,p,x,0),\qquad x\in\sX^\cap_2(0)\subset\sX^{\downpie}(0),
\end{aligned}
\end{equation}
are called the \emph{nonlinear Stokes operators} of the  Painlevé system acting on the foliation \eqref{eq:wm-PVstokes} while preserving the fibers $x=const$.
They are exponentially close to identity on the sectors of their definition $\sX^\cap_1(0)$, resp. $\sX^\cap_2(0)$.
Let us remark, that unlike the monodromy, these Stokes operators are defined only locally on the foliation.
Their action on the canonical solutions \eqref{eq:wm-qp0} is represented by a pair of transformations $\Tilde{\mbf  S}_{1}(c)$, $\Tilde{\mbf  S}_{2}(c)$ of the initial condition $c$ defined by
\begin{equation}\label{eq:wm-tildeS}
 \begin{aligned}
\qquad  \mfr S_1(\cdot,x,0)\circ \left(\!\begin{smallmatrix} q^{\uppie} \\[0pt] p^{\uppie}\end{smallmatrix}\!\right)\!(x,0;c)
&=\left(\!\begin{smallmatrix} q^{\uppie} \\[0pt] p^{\uppie}\end{smallmatrix}\!\right)\!(x,0;\cdot)\circ\Tilde{\mbf  S}_{1}(c),\\
  \mfr S_2(\cdot,x,0)\circ \left(\!\begin{smallmatrix} q^{\downpie} \\[0pt] p^{\downpie}\end{smallmatrix}\!\right)\!(x,0;c)
&=\left(\!\begin{smallmatrix} q^{\downpie} \\[0pt] p^{\downpie}\end{smallmatrix}\!\right)\!(x,0;\cdot)\circ\Tilde{\mbf  S}_{2}(c),
\end{aligned}
\end{equation}
which are symplectic w.r.t. the canonical form $dc_1\wedge dc_2$ and independent of $x$.
This pair $\big(\Tilde{\mbf  S}_{1}(c),\Tilde{\mbf  S}_{2}(c)\big)$ 
is well defined up to a simultaneous conjugation by some symmetry $\mbf T_\alpha$. It provides a local analytic invariant of the system \eqref{eq:wm-PVstokes} with respect to fiber-preserving transversely symplectic transformations (e.g. \cite[Theorem~36]{Kl4}). 

\begin{remark}\label{remark:wm-stokesoperators}
	The operators $\Tilde{\mbf S}_1$, resp. $\Tilde{\mbf S}_2$, are the nonlinear Stokes operators corresponding to the singular directions $-\pi$, resp. $0$,
	whereas $\mbf N^{\circ(-1)}\circ\Tilde{\mbf S}_2\circ\mbf N$ is the nonlinear Stokes operator corresponding to the singular direction $\pi$.
\end{remark}

\paragraph{Nonlinear monodromy operator:}
The nonlinear \emph{total monodromy} $\mfr M (q,p,x,0)$ given by the analytic continuation along a simple positive loop around the origin in the $x$-coordinate, and its representation by its action $\Tilde{\mbf M}^\bullet(c)$ on the initial condition of the solution  $\left(\begin{smallmatrix} q^\bullet \\[3pt] p^\bullet\end{smallmatrix}\right)$, $\bullet=\uppie,\downpie$, 
\begin{equation}\label{eq:wm-nlMx}
	\mfr M (\cdot,x,0)\circ \left(\begin{smallmatrix} q^\bullet \\[3pt] p^\bullet\end{smallmatrix}\right)\!(x,0;c)
=\left(\begin{smallmatrix} q^\bullet \\[3pt] p^\bullet\end{smallmatrix}\right)\!(e^{2\pi i}x,0;c)
=\left(\begin{smallmatrix} q^\bullet \\[3pt] p^\bullet\end{smallmatrix}\right)\!(x,0;\cdot)\circ\Tilde{\mbf M}^\bullet(c),
\end{equation}
is then expressed (see \cite{Kl4}) as the composition
\begin{equation}\label{eq:wm-nlM}
	\Tilde{\mbf M}^{\uppie}=\Tilde{\mbf S}_2\circ \Tilde{\mbf S}_1\circ {\mbf N},\qquad
\Tilde{\mbf M}^{\downpie}=\Tilde{\mbf S}_1\circ {\mbf N}\circ \Tilde{\mbf S}_2.
\end{equation}

\medskip

\paragraph{Nonlinear wild monodromy pseudogroup:}

\begin{definition}\label{def:wm-pseudogroup0}
The pseudogroup action on the foliation of $P_V$ that is generated by both the Stokes operators, the total monodromy operator, and the exponential torus
\[\big\langle \mfr S_1, \mfr S_2, \mfr M, \{\mfr T_\alpha^{\uppie}, \mfr T_\alpha^{\downpie}\mid \alpha\in\Cal O(\C,0)\} \big\rangle\]
is called the \emph{nonlinear wild monodromy pseudo-group}.
%It contains also the total nonlinear monodromy operator, given by continuation of solutions along a loop around the singularity $x=0$. 
Its representation on the $c$-space of initial conditions is generated by
\[\big\langle \Tilde{\mbf S}_1, \Tilde{\mbf S}_2, \{\mbf T_\alpha\mid \alpha\in\Cal O(\C,0)\} \big\rangle.\]
\end{definition}

\smallskip 
\emph{The central problem of this paper is to relate this wild monodromy of $P_V$ to the monodromy of $P_{VI}$ and to describe its dynamics.}
\smallskip

It is not very surprising that the only monodromy operators of the confluent family of $P_{VI}$ that converge when $\epsilon\to 0$ are those corresponding to the loops in 
$\pi_1(\CP^1\sminus\{-\frac{1}{\epsilon},0,\infty\},t_0)$ that persist to the limit as loops in  $\pi_1(\CP^1\sminus\{0,\infty\},t_0)$,
while those corresponding to the vanishing loops diverge.
In fact, for each $\epsilon\neq 0$ the nonlinear monodromy group (of $P_{VI}$) is discretely generated, while for $\epsilon=0$ the nonlinear wild monodromy group (of $P_V$) is generated by a continuous family.  
However, the author's paper \cite{Kl4} shows that generators of the wild monodromy pseudogroup can be obtained through the accumulation of monodromy when $\epsilon\to 0$ along the sequences
$\{\epsilon_n\}_{n\in\pm\N}$,  
\begin{equation*}
\tfrac{1}{\epsilon_n}=\tfrac{1}{\epsilon_0}+n.
\end{equation*}
This will be explained in the next section.

\section{Unfolding of the nonlinear Stokes phenomenon}\label{section:wm-1a}

In this section we will summarize the results of \cite{Kl4} when applied to the confluence $P_{VI}\to P_V$.

\subsection{Sectoral normalization and the unfolded nonlinear Stokes operators}

In the coordinate 
\[x=\frac{1}{\tilde t}+\epsilon,\]
 the confluent Painlevé system \eqref{eq:wm-PVIepsilon} is written as 
\begin{equation}\label{eq:wm-PVIhamiltonianx}
\begin{aligned}
x(x-\epsilon)\frac{dq}{dx}=\, \frac{\partial H(q,p,x,\epsilon)\!}{\partial p},\qquad
x(x-\epsilon)\frac{dp}{dx}=-\frac{\partial H(q,p,x,\epsilon)\!}{\partial q},
\end{aligned}
\end{equation}
with
\begin{equation}\label{eq:wm-Hconf}
\begin{aligned}
H(q,p,x,\epsilon)=& -(1+\epsilon\tilde t) H_{VI}^{conf}(q,p,\tilde t,\epsilon)\\
=&\tfrac{(\vartheta_0+\tilde\vartheta_1-1)^2-(\vartheta_\infty-1)^2}{4}((x\!-\epsilon)q-x)+x\vartheta_0p\\
&+\big(1\!-\epsilon-(x\!-\epsilon)\vartheta_0-x(\vartheta_0\!+\!\tilde\vartheta_1\!-\!1)\big)qp+(2x\!-\epsilon)(qp)^2\\
&+(x\!-\epsilon)(\vartheta_0\!+\!\tilde\vartheta_1\!-\!1)q^2p-xqp^2-(x\!-\epsilon)q^3p^2.
\end{aligned}
\end{equation}
An essential tool in understanding the relation between the monodromy of this system for $\epsilon\neq 0$ and the wild monodromy of the limit system with $\epsilon=0$
is the following theorem  which is an unfolded generalization of Theorem~\ref{theorem:wm-normalization0}.

\goodbreak

\begin{theorem}[\hbox{\cite[Theorems 17 \& 43]{Kl4}}]\label{theorem:wm-normalization}~
\begin{enumerate}[label=(\roman*), wide=0pt, leftmargin=\parindent]

\item \textbf{Formal normalization:} 
The confluent Painlevé system \eqref{eq:wm-PVIhamiltonianx} can be brought to a formal normal form
\begin{equation}\label{eq:wm-normalform}
x(x\!-\epsilon)\tfrac{du}{dx}=
\big(1\!-\epsilon-(x\!-\epsilon)\vartheta_0-x(\vartheta_0\!+\!\tilde\vartheta_1\!-\!1)+2(2x\!-\epsilon)u_1u_2\big)\left(\begin{smallmatrix} 1 & 0 \\[6pt] 0 & -1 \end{smallmatrix}\right)u.
\end{equation}
by means of a formal transversely symplectic change of coordinates 
\begin{equation}\label{eq:wm-hatPhi}
 \left(\begin{smallmatrix} q \\[3pt] p\end{smallmatrix}\right)=\hat{\mbf\Psi}(u,x,\epsilon) =\sum_{k,l\geq 0}\psi^{(kl)}(u)\,x^k\epsilon^l,
\end{equation}
where $\psi^{(kl)}(u)$ are analytic on some fixed polydisc $\sU=\{|u_1|,|u_2|<\delta_u\}$, $\delta_u>0$.
%and their norm grow at most factorially in $k+l$ \[\max_{u\in \sU}\|\phi^{(kl)}(u)\|\leq L^{k+l}(k+l)!\quad\text{for some}\quad L>0.\]

The restriction of $\hat{\mbf\Psi}(u,x,\epsilon)$ to the strong invariant manifolds $x=0$:  $\hat{\mbf\Psi}(u,0,\epsilon)=\sum_{l\geq 0}\psi^{(0l)}(u)\,\epsilon^l$, 
and $x=\epsilon$:  $\hat{\mbf\Psi}(u,\epsilon,\epsilon)=\sum_{k,l\geq 0}\psi^{(kl)}(u)\,\epsilon^{k+l}$, are convergent for $(u,\epsilon)\in\sU\times\{|\epsilon|<\delta_\epsilon\}$ for some $\delta_\epsilon>0$.

\item
\textbf{Unfolded sectorial normalization:}
Let $\eta>0$ be some arbitrarily small constant, and let $\delta_x>\!\!>\delta_\epsilon>0$ denote the radii of small discs at the origin in the $x$-and $\epsilon$-space (depending on $\eta$).
Let 
\begin{equation}\label{eq:wm-sE}
\sE_\pm:=\{ |\epsilon|<\delta_\epsilon,\ |\arg(\pm\epsilon)|<\pi-2\eta \} 
\end{equation}
be two sectors in the $\epsilon$-space.
For $\epsilon\in \sE_\pm$, define a ``spiraling domain'' 
$\sX_\pm(\epsilon)$ (see Figure~\ref{figure:wm-monodromy})
as a simply connected ramified domain spanned by the complete real-time trajectories of the vector fields
\begin{equation}\label{eq:wm-xvectfield}
 e^{i\omega_\pm}x(x-\epsilon)\tfrac{\partial}{\partial x}
\end{equation}
that never leave the disc of radius $\delta_x$,
where $\omega_\pm$ is let vary in the interval
\begin{equation}\label{eq:wm-omegapm}
\left\{\!\!\begin{array}{l}   
\max\{0,\arg(\pm\epsilon)\}-\frac{\pi}{2}+\eta<\omega_\pm<\min\{0,\arg(\pm\epsilon)\}+\frac{\pi}{2}-\eta,\quad\text{for $\epsilon\neq 0$},\\[6pt]
|\omega_\pm|<\frac{\pi}{2}-\eta,\quad\text{for $\epsilon=0$.}
\end{array}\right.
\end{equation}
On this domain, there exists a bounded transversely symplectic change of coordinates 
\[\left(\begin{smallmatrix} q \\[3pt] p\end{smallmatrix}\right)=\mbf\Psi_\pm(u,x,\epsilon),\qquad u\in \sU,\ x\in\sX_\pm(\epsilon), \ \epsilon\in\sE_\pm,\]
analytic on the interior of the domain,
which brings the confluent Painlevé system \eqref{eq:wm-PVIhamiltonianx} to its formal normal form \eqref{eq:wm-normalform}.

When $\epsilon$ tends radially to $0$ with $\arg\epsilon=\beta$, then 
$\mbf\Psi_\pm(u,x,\epsilon)$ converges to $\mbf\Psi_\pm(u,x,0)$ uniformly on compact sets of the sub-domains 
$\lim_{\substack{\epsilon\to 0\\\arg\epsilon=\beta}}\sX_{\pm}(\epsilon)\subseteq \sX_\pm(0)$.
The limit domain $\sX_+(0)=\sX_-(0)$ consists of a pair of sectors $\sX^{\uppie}(0)$, $\sX^{\downpie}(0)$ with a common point at $0$,
and the transformation $\mbf\Psi_+(u,x,0)=\mbf\Psi_-(u,x,0)$ consists in fact of a pair of sectorial transformations $\mbf\Psi^{\uppie}(u,x,0),\ \mbf\Psi^{\downpie}(u,x,0)$ of Theorem~\ref{theorem:wm-normalization0} (it is a functional cochain in the terminology of \cite{MR1}).

%\smallskip
%\noindent
%\textbf{(iii)}
%If $\tilde\Psi(u,x,\epsilon)$ denotes the analytic extension of the function given by the convergent series
%\[\tilde\Psi(u,x,\epsilon)=\sum_{k,l\geq 0} \frac{\psi^{(kl)}(u)}{(k+l)!}x^k\epsilon^l, \quad\text{$\psi^{(kl)}$ as in \eqref{eq:wm-hatPhi}},\] 
%then for each point $(x,\epsilon)$ for which there is an angle $\theta \in\,]\!-\!\frac{\pi}{2},\frac{\pi}{2}[$ such that
%the set $\mbox{$\mathbf{S}_\theta\cdot(x,\epsilon)$}\subseteq \cup_{\epsilon\in\sE_\pm}\sX_\pm\times\{\epsilon\}$, with $\mathbf{S}_\theta\subset\C$ denoting the circle  through the points $0$ and $1$ with  center on $e^{i\theta}\R^+$,
%we can express $\mbf\Psi_\pm(u,x,\epsilon)$ through the 
%following Laplace transform of $\tilde\Psi$:
%\begin{equation}
%\mbf\Psi_\pm(u,x,\epsilon)=\int_0^{+\infty e^{i\theta}}\!\!\!\tilde\Psi(u,sx,s\epsilon)\,e^{-s}\,ds.
%\end{equation}
%In particular, $\mbf\Psi_\pm(u,x,0)$ is the pair of Borel sums $\mbf\Psi^{\uppie}(u,x,0)$, $\mbf\Psi^{\downpie}(u,x,0)$, of the formal series $\hat{\mbf\Psi}(u,x,0)$ of Theorem~\ref{theorem:wm-normalization0}.

The transformations $\mbf\Psi_\pm(u,x,\epsilon)$ are asymptotic to $\hat{\mbf\Psi}(u,x,\epsilon)$ when 
$\coprod_{\epsilon\in\sE_\pm}\sX_\pm(\epsilon)\ni(x,\epsilon)\to 0$,
and they satisfy the same $(\tfrac{\partial}{\partial u}, \tfrac{\partial}{\partial x}, \tfrac{\partial}{\partial\epsilon})$-differential relations with meromorphic coefficients.
\end{enumerate}
\end{theorem}

\begin{figure}[t]
	\centering
	\includegraphics[width=0.95\textwidth]{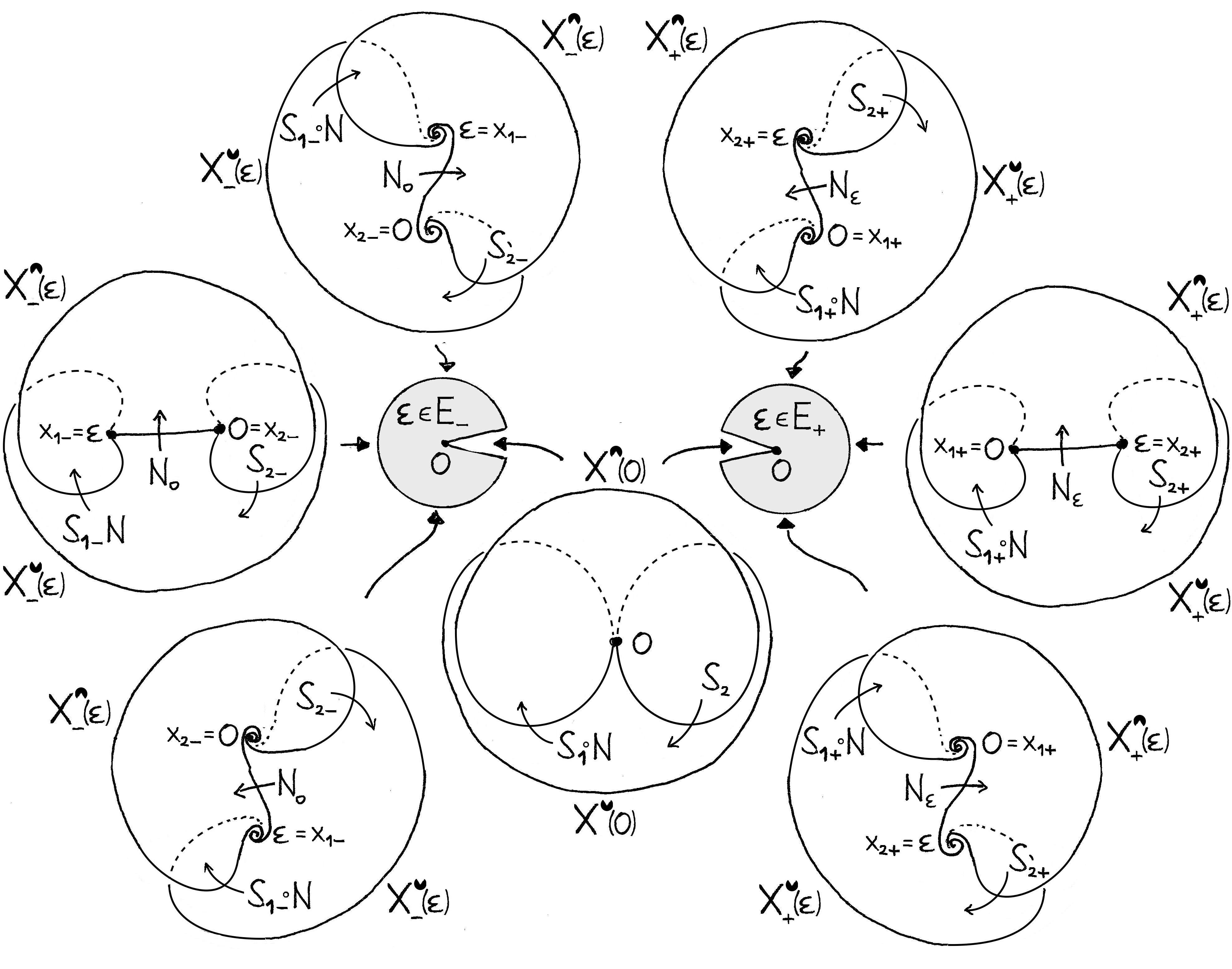}
	\caption{Connecting transformation between the general solutions
		$\left(q_\pm^\bullet, p_\pm^\bullet \right)\!(x,\epsilon;c) =\mbf\Psi_\pm(u^\bullet(x,\epsilon;c),x,\epsilon)$, %$\bullet=\protect\uppie,\protect\downpie$, 
		on the two parts $\sX_\pm^{\protect\uppie}(\epsilon),\sX_\pm^{\protect\downpie}(\epsilon)$ of the domain $\protect\sX_\pm(\epsilon)$.}
	\label{figure:wm-monodromy}
\end{figure}

\begin{remark}
	1)	Strictly speaking, the domains $\sX_{\pm}(\epsilon)$ are defined on the universal covering of $\{|x|<\delta_x\}\sminus\{0,\epsilon\}$.
	
	2) The definition of the domains $\sX_\pm(\epsilon)$ in Theorem~\ref{theorem:wm-normalization} is such that when $x$ approaches a singular point $x_{i\pm}(\epsilon)$, where
	\begin{equation}\label{eq:wm-x12}
		x_{1+}(\epsilon)=x_{2-}(\epsilon)=0, \qquad x_{1-}(\epsilon)=x_{2+}(\epsilon)=\epsilon, 
	\end{equation}
	from within the domain, then the corresponding components of the general solution \eqref{eq:wm-u} tend to
	\[u_i(x,\epsilon)\to \infty,\ u_{3-i}(x,\epsilon)=\tfrac{h}{u_i(x,\epsilon)}\to 0,\]
	when $x\to x_{i\pm}(\epsilon)$ along a real trajectory of \eqref{eq:wm-xvectfield}.
	
	3) Alternatively, the form of the domains $\sX_\pm(\epsilon)$ could be also understood from the point of view of exact WKB analysis of the system \eqref{eq:wm-PVIhamiltonianx}
	after a change of variable $x=\epsilon z$.
\end{remark}

\paragraph{Canonical 2-parameter family of solutions:}
The Hamiltonian \eqref{eq:wm-Hconf} satisfies $H\circ\hat{\mbf\Psi}(u,x,\epsilon)=G(u,x,\epsilon)+O(x(x\!-\!\epsilon))$, where
\[G(u,x,\epsilon)=\mbox{\scriptsize $\tfrac{(\vartheta_0+\tilde\vartheta_1-1)^2-(\vartheta_\infty-1)^2}{4}x+ \big(1\!-\epsilon-(x\!-\epsilon)\vartheta_0-x(\vartheta_0\!+\!\tilde\vartheta_1\!-\!1)\big)u_1u_2 + (2x\!-\epsilon)(u_1u_2)^2$},\]
and $\frac{G(u,x,\epsilon)}{x(x-\epsilon)}$ is a time-$x$-dependent Hamiltonian for the normal form \eqref{eq:wm-normalform}. The general solutions of \eqref{eq:wm-normalform} are of the form
\begin{equation}\label{eq:wm-u}
\begin{aligned}
u_1(x,\epsilon;c)&=c_1 E(c_1c_2,x,\epsilon),\\
u_2(x,\epsilon;c)&=c_2 E(c_1c_2,x,\epsilon)^{-1},
\end{aligned}
\qquad 
c=\transp(c_1,c_2)\in\C^2,
\end{equation}
where
\begin{equation}\label{eq:wm-E}
E(h,x,\epsilon)= 
\left\{\!\!\begin{array}{ll} 
x^{-\frac{1}{\epsilon}+1-\vartheta_0+2h} (x-\epsilon)^{\frac{1}{\epsilon}-\vartheta_0-\tilde\vartheta_1+2h}, & \text{for $\epsilon\neq 0$},\\[6pt]
e^{-\frac{1}{x}}x^{-2\vartheta_0-\tilde\vartheta_1+1+4h}, & \text{for $\epsilon=0$}.\\
\end{array}\right.
\end{equation}

In order for a branch of the solution $u(x,\epsilon;c)$ \eqref{eq:wm-u} to have a limit when $\epsilon\to 0$, one needs to further cut the domains $\sX_\pm(\epsilon)$, $\epsilon\neq 0$,
to two parts  
\begin{equation}\label{eq:wm-uppiedownpie}
\text{an upper part }\sX^{\uppie}_\pm(\epsilon),\qquad \text{and a lower part } \sX^{\downpie}_\pm(\epsilon),
\end{equation} 
corresponding to the two parts $\sX^{\uppie}(0)$ and $\sX^{\downpie}(0)$ of $\sX_+(0)=\sX_-(0)$,
by a cut in between the singular points $0$ and $\epsilon$ along a suitable trajectory of \eqref{eq:wm-xvectfield} (see Figure~\ref{figure:wm-monodromy}).
The two parts of $\sX_\pm(\epsilon)$ then intersect in two outer spiraling sectors 
\[\sX_{i\pm}^\cap(\epsilon)=\{x\in\sX_{\pm}(\epsilon): x_{i\pm}+e^{2\pi i}(x-x_{i\pm})\in\sX_{\pm}(\epsilon)\},\qquad i=1,2,\]
attached to the singularities $x_{i\pm}$ \eqref{eq:wm-x12}, and along the central cut between the singularities $\{x_{1\pm},x_{2\pm}\}=\{0,\epsilon\}$. 
%Hence
%\[\sX_{1\pm}^\cap(\epsilon)\subset\sX_{\pm}^{\downpie}(\epsilon),\qquad \sX_{2\pm}^\cap(\epsilon)\subset\sX_{\pm}^{\uppie}(\epsilon).\]

We now take $E_\pm^{\uppie}(h,x,\epsilon)$ and $E_\pm^{\downpie}(h,x,\epsilon)$ as two branches of $E(h,x,\epsilon)$ \eqref{eq:wm-E} on the two parts \eqref{eq:wm-uppiedownpie} of the domain, that agree on the right 
intersection sector $\sX_{2\pm}^\cap$, and have a limit when $\epsilon\to 0$. 
They determine a pair of general solutions \eqref{eq:wm-u} of the model system
\[u_\pm^{\bullet}(x,\epsilon;c),\qquad \bullet=\uppie,\downpie,\] 
and a pair of \emph{canonical 2-parameter families of solutions} of the original system \eqref{eq:wm-PVIhamiltonianx}
\begin{equation}\label{eq:wm-qp}
	\left(\begin{smallmatrix} q_\pm^\bullet \\[3pt] p_\pm^\bullet\end{smallmatrix}\right)\!(x,\epsilon;c) 
	:=\mbf\Psi_\pm(u_\pm^{\bullet}(x,\epsilon;c),x,\epsilon),\qquad \bullet=\uppie,\downpie.
\end{equation}

\medskip

\paragraph{Unfolded center manifold solution:}
\begin{corollary}\label{cor:wm-centermanifold}
The system \eqref{eq:wm-PVIhamiltonianx} has a unique bounded analytic solution on the domain $\sX_\pm^\bullet(\epsilon)$, $\epsilon\in\sE_\pm$, $\bullet=\uppie,\downpie$, given by $\left(\begin{smallmatrix} q_\pm^\bullet \\[3pt] p_\pm^\bullet\end{smallmatrix}\right)\!(x,\epsilon;0)=\mbf\Psi_\pm(0,x,\epsilon)$ (which agrees between
$\sX_\pm^{\uppie}(\epsilon)$ and $\sX_\pm^{\downpie}(\epsilon)$ along the central intersection).
We call it the ``\emph{unfolded sectorial center manifold}'' of the foliation. 
\end{corollary}

\begin{remark}
	This kind of solution seems to be new in the literature. In our opinion this is an analogue of the bi-tronqu\'ee solution for $P_{VI}$.
	While we have obtained it by the confluence, this solution is simply characterized by the fact that it is bounded at both of the singular points when approached along certain logarithmic spiral. Due to the overall symmetry of $P_{VI}$ there should be in general one such solution for every selection of a pair of singular points together with suitable paths of approach.
%	Geometrically speaking, the 3-dimensional vector field generating the foliation has at each singularity at least one (and up to two) 2-dimensional strong invariant manifold tangent to the $t$-coordinate ensured by the Hadamard--Perron theorem \cite{IlYa}, consisting of the solutions that are bounded along the suitably chosen logarithmic path. As long as the two 2-dimensional invariant manifolds at each of the two singularities are in a general position, their intersection will be such ``bi-tronqu\'ee'' solution.
\end{remark}

\paragraph{Unfolded exponential torus:}
\begin{proposition}[\hbox{\cite[Proposition~31]{Kl4}}]
	For $\epsilon\in\sE_\pm$, any symmetry of \eqref{eq:wm-normalform} that is bounded and analytic on $(u,x)\in\sU\times\sX_\pm(\epsilon)$ is of the form $\mbf T_{a_\epsilon}$ \eqref{eq:wm-Ta}
	with $a_\epsilon(h)$ analytic.
\end{proposition}

\begin{corollary}[\hbox{\cite[Corollary~32]{Kl4}}]
	The normalizing transformations $\hat{\mbf\Psi}$ and $\mbf\Psi_\pm$ of Theorem~\ref{theorem:wm-normalization} are unique up to a right composition with the same analytic symmetry $\mbf T_\alpha(u,\epsilon)$ where $\alpha(h,\epsilon)$ is an analytic germ in $(h,\epsilon)$ at the origin which is uniquely determined by the convergent series
	$\hat{\mbf\Psi}(u,0,\epsilon)$.
\end{corollary}

\paragraph{Unfolded nonlinear Stokes operators:}
The connecting transformations between the 2-parameter families of solutions \eqref{eq:wm-qp} 
$\left(\begin{smallmatrix} q_\pm^\bullet \\[3pt] p_\pm^\bullet\end{smallmatrix}\right)$ on the 3 intersections between $\sX_\pm^{\uppie}(\epsilon)$ and $\sX_\pm^{\downpie}(\epsilon)$, when represented by an action on the initial condition $c$, are as in Figure~\ref{figure:wm-monodromy}.

There, the operators $\mbf S_{1\pm}(c,\epsilon)$, $\mbf S_{2\pm}(c,\epsilon)$ are the representations of 
\emph{unfolded Stokes operators} $\mfr S_1(\cdot,x,\epsilon)$, $\mfr S_2(\cdot,x,\epsilon)$, defined in the same way as in
\eqref{eq:wm-stokes}, \eqref{eq:wm-tildeS} on the intersection sectors $\sX^\cap_{i\pm}(\epsilon)$, $i=1,2$,
and converge to the Stokes operators~\eqref{eq:wm-tildeS} when $\epsilon\to 0$ in the sector $\sE_\pm$
\[\mbf S_{i\pm}(c,\epsilon)\to \mbf S_{i\pm}(c,0)=\Tilde{\mbf S}_i(c)\ \text{ as }\ \sE_\pm\ni\epsilon\to 0.\]
On the other hand the  \emph{formal monodromies} (corresponding to the monodromies of the solutions \eqref{eq:wm-u} of the formal normal form \eqref{eq:wm-normalform})
\begin{equation}\label{eq:wm-N}
\begin{aligned}
\mbf N_0=\mbf T_{2\pi i(-\frac{1}{\epsilon}-\vartheta_0+2c_1c_2)}&:c\!\mapsto\!\!\left(\begin{smallmatrix}\!\! e^{2\pi i(-\frac{1}{\epsilon}-\vartheta_0+2c_1c_2)} \hskip-24pt & \\[3pt] &  e^{2\pi i(\frac{1}{\epsilon}+\vartheta_0-2c_1c_2)}\!\! \end{smallmatrix}\right)\! c,\\
\mbf N_\epsilon=\mbf T_{2\pi i(\frac{1}{\epsilon}-\vartheta_0-\tilde\vartheta_1+1+2c_1c_2)}&:c\!\mapsto\!\! \left(\begin{smallmatrix} \!\! e^{2\pi i(\frac{1}{\epsilon}-\vartheta_0-\tilde\vartheta_1+2c_1c_2)} \hskip-36pt & \\[3pt] & e^{2\pi i(-\frac{1}{\epsilon}+\vartheta_0+\tilde\vartheta_1-2c_1c_2)}\!\!\! \end{smallmatrix}\right)\!c,\\
\end{aligned}
\end{equation}
diverge when $\epsilon\to 0$, except for the total formal monodromy operator
\begin{equation}\label{eq:wm-formalmonodromy} 
\mbf N=\mbf N_0\circ \mbf N_\epsilon=\mbf T_{2\pi i(-2\vartheta_0-\tilde\vartheta_1+4c_1c_2)}.
\end{equation}

\paragraph{Decomposition of nonlinear monodromy operators:}
For $\epsilon\neq 0$, one can now represent the action of the \emph{nonlinear monodromy operators} $\mfr M_0(q,p,x,\epsilon)$ and $\mfr M_\epsilon(q,p,x,\epsilon)$ around the singular points $0$ and $\epsilon$ in positive direction by their action on the initial condition $c$ of the general solutions \eqref{eq:wm-qp},
\begin{equation*}
 \mfr M_{x_{i\pm}}(\cdot,x,\epsilon)\circ \left(\begin{smallmatrix} q_\pm^\bullet \\[3pt] p_\pm^\bullet\end{smallmatrix}\right)\!(x,\epsilon;c)= 
\left(\begin{smallmatrix} q_\pm^\bullet \\[3pt] p_\pm^\bullet\end{smallmatrix}\right)\!(x,\epsilon;\cdot)\circ\mbf M_{x_{i\pm}}^\bullet(c,\epsilon),\qquad
\bullet=\uppie,\downpie,
\end{equation*}
and express it as
\begin{equation}\label{eq:wm-Mdecomposition}
\begin{aligned}
\mbf M_{x_{1\pm}}^{\uppie}&=\mbf N_{x_{2\pm}}^{\circ(-1)}\circ \mbf S_{1\pm}\circ \mbf N,\qquad\qquad& \mbf M_{x_{1\pm}}^{\downpie}&=\mbf S_{1\pm}\circ \mbf N_{x_{1\pm}},\\
\mbf M_{x_{2\pm}}^{\uppie}&=\mbf S_{2\pm}\circ \mbf N_{x_{2\pm}}, & \mbf M_{x_{2\pm}}^{\downpie}&=\mbf N_{x_{2\pm}}\circ \mbf S_{2\pm},
\end{aligned}
\end{equation}
see Figure~\ref{figure:wm-monodromy}. % (note that the operator obtained as a composition of the connecting transformation acting on $c$ along a certain loop in their order gives the inverse of the  monodromy operator corresponding to the loop).

\subsection{Accumulation of monodromy}\label{section:wm-1a2}

We can now formulate the essential result of \cite{Kl4} which allows to obtain a representation of the wild monodromy pseudogroup of the limit system as an accumulation of the monodromy pseudogroup of the system when $\epsilon\to 0$.

Let $\{\epsilon_n\}_{n\in\pm\N}$ be sequence in $\sE_\pm\sminus\{0\}$ defined by
\begin{equation}\label{eq:wm-epsilon_n}
 \tfrac{1}{\epsilon_n}=\tfrac{1}{\epsilon_0}+n,\qquad \epsilon_0\in\sE_\pm\sminus\{0\}, \ n\in\pm\N,
\end{equation}
along which the divergent exponential factor $e^{\frac{2\pi i}{\epsilon}}$ in the formal monodromy \eqref{eq:wm-N} stays constant.
Let
\[\tilde{\mfr M}_{x_{i\pm}}(q,p,x;\kappa):=\lim_{n\to\pm\infty}\mfr M_{x_{i\pm}}(q,p,x,\epsilon_n),\quad \kappa:=e^{\frac{2\pi i}{\epsilon_0}},\] 
and let $\Tilde{\mbf M}_{x_{i\pm}}^\bullet(c;\kappa)$ be the corresponding limit action \eqref{eq:wm-nlMx} on the initial conditions of the general solution 
$\left(\begin{smallmatrix} q_\pm^\bullet \\[3pt] p_\pm^\bullet \end{smallmatrix}\right)\!(x,0;c)$.
Then
\begin{equation}\label{eq:wm-wildmonodromy}
\begin{aligned}
\Tilde{\mbf M}_{0+}^{\uppie}(c;\kappa)&= \tilde{\mbf N}_{\epsilon}(\cdot;\kappa)^{\circ(-1)}\circ\tilde{\mbf S}_{1}\circ\mbf N(c),&\qquad\qquad
\Tilde{\mbf M}_{0+}^{\downpie}(c;\kappa)&= \tilde{\mbf S}_{1}\circ\tilde{\mbf N}_{0}(c;\kappa),& \\[-3pt]
\Tilde{\mbf M}_{\epsilon+}^{\uppie}(c;\kappa)&= \tilde{\mbf S}_{2}\circ\tilde{\mbf N}_{\epsilon}(c;\kappa),& 
\Tilde{\mbf M}_{\epsilon+}^{\downpie}(c;\kappa)&= \tilde{\mbf N}_{\epsilon}(\cdot;\kappa)\circ \tilde{\mbf S}_{2}(c),& \\[3pt] 
\Tilde{\mbf M}_{\epsilon-}^{\uppie}(c;\kappa)&= \tilde{\mbf N}_{0}(\cdot;\kappa)^{\circ(-1)}\circ \tilde{\mbf S}_{1}\circ \mbf N(c),&
\Tilde{\mbf M}_{\epsilon-}^{\downpie}(c;\kappa)&= \tilde{\mbf S}_{1}\circ\tilde{\mbf N}_{\epsilon}(c;\kappa),&\\[-3pt]
\Tilde{\mbf M}_{0-}^{\uppie}(c;\kappa)&= \tilde{\mbf S}_{2}\circ\tilde{\mbf N}_{0}(c;\kappa),&
\Tilde{\mbf M}_{0-}^{\downpie}(c;\kappa)&= \tilde{\mbf N}_{0}(\cdot;\kappa)\circ \tilde{\mbf S}_{2}(c),& 
\end{aligned}
\end{equation}
where 
\begin{equation}\label{eq:wm-wildformalmonodromy}
\begin{aligned}
\tilde{\mbf N}_0(\cdot;\kappa)=\mbf T_{2\pi i(-\vartheta_0+2h)-\log\kappa} &:c\!\mapsto\!\!\left(\begin{smallmatrix}\!\! \frac{1}{\kappa} e^{2\pi i(-\vartheta_0+2c_1c_2)} \hskip-28pt & \\[3pt] & \kappa e^{2\pi i(\vartheta_0-2c_1c_2)}\!\! \end{smallmatrix}\right)\! c,\\
\tilde{\mbf N}_\epsilon(\cdot;\kappa)=\mbf T_{2\pi i(-\vartheta_0-\tilde\vartheta_1+2h)+\log\kappa}&:c\!\mapsto\!\! \left(\begin{smallmatrix} \!\!\kappa e^{2\pi i(-\vartheta_0-\tilde\vartheta_1+2c_1c_2)} \hskip-40pt & \\[3pt] & \frac{1}{\kappa} e^{2\pi i(\vartheta_0+\tilde\vartheta_1-2c_1c_2)}\!\!\! \end{smallmatrix}\right)\! c,
\end{aligned}
\end{equation}
and $\mbf N(c)=\tilde{\mbf N}_0(\cdot;\kappa)\circ\tilde{\mbf N}_\epsilon(c;\kappa)$, are elements of the nonlinear exponential torus. 
(The subscripts $0\pm$, $\epsilon\pm$ in \eqref{eq:wm-wildmonodromy}, \eqref{eq:wm-wildformalmonodromy} are purely symbolical and no longer related to the parameter $\epsilon$.) 

In order to express the \emph{nonlinear Stokes operators} from the monodromy one can substitute in \eqref{eq:wm-wildmonodromy}
$e^{2\pi i(-\vartheta_0+2c_1c_2)}$ for $\kappa$ in $\Tilde{\mbf M}_{0\pm}^{\bullet}$ to kill the factor  $\tilde{\mbf N}_0(\kappa)$, 
resp. $e^{2\pi i(\vartheta_0+\tilde\vartheta_1-2c_1c_2)}$ in $\Tilde{\mbf M}_{\epsilon\pm}^{\bullet}$ to kill the factor $\tilde{\mbf N}_\epsilon$,
hence e.g.
\begin{equation}\label{eq:wm-stokesops}
	\begin{aligned}
		\tilde{\mbf S}_1(c)&=\Tilde{\mbf M}_{0+}^{\downpie}\big(c;e^{2\pi i(-\vartheta_0+2h(c))}\big),\\ 
		\tilde{\mbf S}_2(c)&=\Tilde{\mbf M}_{\epsilon+}^{\downpie}\big(c;e^{2\pi i(\vartheta_0+\tilde\vartheta_1-2h(c))}\big)=
		\Tilde{\mbf M}_{\epsilon+}^{\uppie}\big(c;e^{2\pi i(\vartheta_0+\tilde\vartheta_1-2h(c))}\big),\\
		\mbf N^{\circ(-1)}\circ\tilde{\mbf S}_1\circ\mbf N(c)&=\Tilde{\mbf M}_{0+}^{\uppie}\big(c;e^{2\pi i(-\vartheta_0+2h(c))}\big),
	\end{aligned}
\end{equation}
where $h(c)=c_1c_2$.

\begin{proposition}\label{prop:wm-accumulation}
The wild monodromy pseudogroup of $P_V$ (Definition~\ref{def:wm-pseudogroup0}) is generated by the limit operators 
\[\big\langle\tilde{\mfr M}_{0\pm}(q,p,x;e^{\alpha(\mfr h^{\bullet})}),\ \tilde{\mfr M}_{\epsilon\pm}(q,p,x;e^{\alpha(\mfr h^{\bullet})})\,\big|\,\alpha\in\Cal O(\C,0) \big\rangle,\qquad
\text{where }\ \mfr h^{\bullet}\circ\Psi^{\bullet}=u_1u_2.\]
Its representation on the $c$-space of initial conditions over $\sX^\bullet(0)$, $\bullet=\uppie,\downpie$, is generated by,
\[\big\langle\tilde{\mbf M}_{0\pm}^{\bullet}(c;e^{\alpha(h(c))}),\ \tilde{\mbf M}_{\epsilon\pm}^{\bullet}(c;e^{\alpha(h(c))})\,\big|\,\alpha\in\Cal O(\C,0) \big\rangle,\qquad
\text{where }\  h(c)=c_1c_2.\] 	
\end{proposition}

\begin{proof}
Let for example $\sX_\pm^\bullet=\sX_+^{\uppie}$. Then by \eqref{eq:wm-wildmonodromy}, \eqref{eq:wm-wildformalmonodromy} both 	
$\tilde{\mbf M}_{0+}^{\uppie}(c;e^{\alpha(h)})$ and $\tilde{\mbf M}_{\epsilon+}^{\uppie}(c;e^{\alpha(h)})$ belong to the representation of the wild monodromy pseudogroup. Conversely, one can express $\mbf T_\alpha=\tilde{\mbf M}_{0+}^{\uppie}(\cdot;e^{-\alpha(\cdot)})\circ(\tilde{\mbf M}_{0+}^{\uppie})^{\circ(-1)}(c;e^{\alpha(h)})$,
and $\tilde{\mbf S}_1(c)$, $\tilde{\mbf S}_2(c)$ using \eqref{eq:wm-stokesops}.
\end{proof}	
	
Also the vector field  $(c_1\tfrac{\partial}{\partial c_1}-c_2\tfrac{\partial}{\partial c_2})$, which is Hamiltonian vector field of $h=c_1c_2$ with respect to $dc_1\wedge dc_2$, which is in a sense an \emph{``infinitesimal generator'' of the exponential torus} \eqref{eq:wm-Tc}, can be expressed as e.g.
\begin{equation}\label{eq:wm-inftorus}
\begin{aligned}
\dot c %&=\big(\kappa\tfrac{d}{d\kappa}\tilde{\mbf N}_{\epsilon}(\kappa)\big)\circ \tilde{\mbf N}_{\epsilon}(\kappa)^{\circ(-1)}(c)\\
&=-\big(\kappa\tfrac{d}{d\kappa}\tilde{\mbf N}_{0}(\,\cdot\,;\kappa)\big)\circ \tilde{\mbf N}_{0}(c;\kappa)^{\circ(-1)}
=-\big(\kappa\tfrac{d}{d\kappa}\tilde{\mbf M}_{0+}^{\uppie}(\,\cdot\,;\kappa)\big)\circ \tilde{\mbf M}_{0+}^{\uppie}(c;\kappa)^{\circ(-1)}.
\end{aligned}
\end{equation}
Among its Hamiltonians, which are determined only up to an additive constant, the function $h=c_1c_2$ is characterized by vanishing on the center manifold solution
of Corollary~\ref{cor:wm-centermanifold}.

\section{The character variety of $P_{VI}$ and the nonlinear monodromy action on it}\label{section:wm-1}

In this section we recall the usual approach to $P_{VI}$
through isomonodromic  deformations of $2\!\times\!2$ traceless linear systems with four Fuchsian singularities on $\CP^1$,
the Riemann-Hilbert correspondence between Fuchsian systems and their monodromy representations, 
the character variety of $P_{VI}$ and the modular group action on it, and gather some facts that will serve us later. 
Our main reference for this part are the articles of Iwasaki \cite{I} and of Inaba, Iwasaki and Saito \cite{IIS}.

\begin{notation}
A triple of indices $(i,j,k)$ will always denote a permutation of $(0,t,1)$, and a quadruple $(i,j,k,l)$ will denote a permutation of $(0,t,1,\infty)$.
\end{notation}

\subsection{Isomonodromic deformations of $2\!\times\!2$ systems and the Riemann--Hilbert correspondence}

The sixth Painlevé equation $P_{VI}(\vartheta)$ with a parameter 
\[\vartheta=(\vartheta_0,\vartheta_t,\vartheta_1,\vartheta_\infty)\in\C^4\]
governs isomonodromic deformations of traceless $2\!\times\!2$ linear differential systems with four Fuchsian singularities on $\CP^1$
\begin{equation}\label{eq:wm-A}
\frac{d\phi}{d z}=\Big[\frac{A_0(t)}{z}+\frac{A_t(t)}{z-t} +\frac{A_1(t)}{z-1}\Big]\phi 
\end{equation}
with the residue matrices $A_l\in\mfr{sl}_2(\C)$ having $\pm\tfrac{\vartheta_l}{2}$ as eigenvalues.
In general (if each $A_l$ is semi-simple and the system is irreducible), one can write 
\[A_i=\begin{pmatrix} v_i+\tfrac{\vartheta_i}{2} &  -u_iv_i \\ \tfrac{v_i+\vartheta_i}{u_i} & -v_i- \tfrac{\vartheta_i}{2}\end{pmatrix},
\quad i=0,t,1,\qquad 
-A_0-A_t-A_1=A_\infty=\begin{pmatrix} \tfrac{\vartheta_\infty}{2} & 0 \\ 0 &  -\tfrac{\vartheta_\infty}{2}\end{pmatrix},\]
for some functions $u_i(t),v_i(t)$.
The isomonodromicity of such system is expressed by the Schlesinger equations
\begin{equation}\label{eq:wm-schlesinger}
\frac{dA_0}{dt}=\frac{[A_t,A_0]}{t},\qquad
\frac{dA_t}{dt}=\frac{[A_0,A_t]}{t}+\frac{[A_1,A_t]}{t-1},\qquad
\frac{dA_1}{dt}=\frac{[A_t,A_1]}{t-1},
\end{equation}
corresponding to the integrability conditions on the logarithmic connection in variables $(z,t)$ 
\begin{equation}\label{eq:wm-connection} 
	\nabla(z,t)=d-\Big[A_0(t)\,d\log(z)+A_t(t)\,d\log(z\!-\!t)+A_1(t)\,d\log(z\!-\!1)\Big]
\end{equation} 
on the trivial rank 2 vector bundle.
Denoting $[A(z,t)]_{ij}$ the $(i,j)$-component of the matrix of the system \eqref{eq:wm-A},
then if the system is irreducible the 1-form $[A(z,t)]_{12}\,dz$ is non-null, and so it must have a unique zero at some point $z=q(t)$ 
\[q(t)=\frac{-t[A_0]_{12}}{t[A_t]_{12}+[A_1]_{12}}.\] 
This point is an apparent singularity of the second order linear ODE solved by the first component of any solution $\phi(z,t)$ of \eqref{eq:wm-A}.
Denoting 
\[p(t)=[A(q,t)]_{11}+\frac{\vartheta_0}{2q}+\frac{\vartheta_t}{2(q-t)}+\frac{\vartheta_1}{2(q-1)},\]
then the Schlesinger equations \eqref{eq:wm-schlesinger} are equivalent to the Hamiltonian system $\eqref{eq:wm-hamiltonianvectorfield}$ of $P_{VI}$ \cite{JM,Jim},
whose the Hamiltonian function \eqref{eq:wm-H_VI} is given by
\[H_{VI}\big(q(t),p(t),t\big)=\tr\Big[\big(\tfrac{A_0(t)}{t}+\tfrac{A_1(t)}{t-1}\big)A_t(t)\Big]-\tfrac{\vartheta_0\vartheta_t}{2t}-\tfrac{\vartheta_t\vartheta_1}{2(t-1)}.\]
The \emph{tau function} $\tau_{VI}$ of $P_{VI}$ is defined by
\[\tfrac{d}{dt}\log\tau_{VI}(t)=\tr\Big[\big(\tfrac{A_0(t)}{t}+\tfrac{A_1(t)}{t-1}\big)A_t(t)\Big],\]
which is the coefficient of $\tfrac{1}{z-t}$ in $\tfrac12\tr A(z,t)^2$.

Choosing a germ of a fundamental matrix solution $\Phi(z,t)$ of the system \eqref{eq:wm-A} near some nonsingular point $z_0$, one has a linear monodromy representation (anti-homomorphism)
\[\rho:\pi_1(\CP^1\sminus\{0,t,1,\infty\},z_0)\to\SL_2(\C),\] 
such that the analytic continuation $\Phi(z,t)$ along a path $\gamma$ defines another fundamental matrix solution $\Phi(z,t)\rho(\gamma)$.
The conjugation class of such monodromy representation in $\SL_2(\C)$ is independent of the choice of  $\Phi$.
The \emph{isomonodromic condition} \eqref{eq:wm-schlesinger} on the system \eqref{eq:wm-A} is equivalent to the conjugation class of the monodromy being locally constant with respect to $t$, or equivalently to the
existence of a fundamental matrix solution $\Phi(z,t)$ whose actual monodromy is locally independent of $t$ \cite{Bol}.

The \emph{Riemann--Hilbert correspondence} in this setting is given by the monodromy map between the space of linear systems \eqref{eq:wm-A} with prescribed poles and local eigenvalue data $\pm\tfrac{\vartheta_l}{2}$ modulo global gauge transformations (conjugation by $\SL_2(\C)$) on one side,
and the space of monodromy representations with prescribed local exponents $e_l,e_l^{-1}$
\begin{equation}
e_l:=e^{\pi \vartheta_l},\quad l\in\{0,t,1,\infty\},
\end{equation}
modulo the adjoint action (conjugation) of $\SL_2(\C)$ on the other side (see \cite{IIS} for much more precise setting of the correspondence).
Therefore it can be also translated as a correspondence between solutions of $P_{VI}$ and equivalence classes of monodromy representations.
%between the Painlevé flow \eqref{eq:wm-hamiltonianvectorfield} on the moduli space of linear systems \eqref{eq:wm-A} and the locally constant ``isomonodromic'' flow on the moduli space of monodromy representations.

\subsection{The character variety of $P_{VI}$}

Given a representation (anti-homomorphism)
\begin{equation}\label{eq:PV-rho}
 \rho:\pi_1(\CP^1\sminus\{0,t,1,\infty\},z_0)\to\SL_2(\C),\qquad \rho(\gamma_i\gamma_j)=\rho(\gamma_j)\rho(\gamma_i),
\end{equation}
(where $\gamma_i\gamma_j$ denotes the concatenation of paths, i.e. the path following first $\gamma_i$ and then $\gamma_j$),
let $\gamma_0,\gamma_t,\gamma_1,\gamma_\infty$
be simple loops in the $z$-space around $0,t,1,\infty$ respectively such that $\gamma_0\gamma_t\gamma_1\gamma_\infty=\id$ (see Figure \ref{figure:wm-loops}),
and denote $M_l=\rho(\gamma_l)$ the corresponding monodromy matrices
\[M_\infty M_1M_tM_0=I.\]

\begin{figure}[t]
	\centering
	\begin{tikzpicture} [scale=0.5] 
		\draw (3,3.5) node[above] {$z_0$};
		\filldraw (0,0) circle (2pt);
		\filldraw (2,0) circle (2pt);
		\filldraw (4,0) circle (2pt);
		\filldraw (6,0) circle (2pt);
		\draw (0,-.6) node[below] {$0$};
		\draw (2,-.6) node[below] {$t$};
		\draw (4,-.6) node[below] {$1$};
		\draw (6,-.6) node[below] {$\infty$};
		\begin{scope}[thick,decoration={markings, mark=at position 0.5 with {\arrow{>}}}] 
			\draw[postaction={decorate}] (3,3.5) -- (80:0.5) arc(80:400:0.5) -- cycle;
			\draw[postaction={decorate}] (3,3.5) -- ($(2,0)+(100:0.5)$) arc(100:420:0.5) -- cycle;
			\draw[postaction={decorate}] (3,3.5) -- ($(4,0)+(120:0.5)$) arc(120:440:0.5) -- cycle;
			\draw[postaction={decorate}] (3,3.5) -- ($(6,0)+(140:0.5)$) arc(140:460:0.5) -- cycle;
		\end{scope}
	\end{tikzpicture}
	\caption{The loops  $\gamma_0,\gamma_t,\gamma_1,\gamma_\infty\in\pi_1(\CP^1\sminus\{0,t,1,\infty\},z_0)$.}
	\label{figure:wm-loops}
\end{figure} 

The conjugacy class of an irreducible monodromy representation is completely determined by its trace coordinates 
by a theorem of Fricke, Klein and Vogt (cf. \cite{Mag}). 
These coordinates are given by the four parameters
\begin{equation}\label{eq:wm-a}
a_l=\tr(M_l)=e_l+\tfrac{1}{e_l}=2\cos(\pi\vartheta_l),\quad l=0,t,1,\infty, 
\end{equation}
and the three variables
\begin{equation}\label{eq:wm-X}
X_0=\tr(M_1M_t),\quad X_t=\tr(M_0M_1),\quad X_1=\tr(M_tM_0), 
\end{equation}
satisfying the \emph{Fricke relation}
\begin{equation}\label{eq:wm-Fricke}
F(X,\theta(a))=0,
\end{equation}
where
\begin{equation}\label{eq:wm-F}
 F(X,\theta):=X_0X_tX_1+X_0^2+X_t^2+X_1^2-\theta_0X_0-\theta_tX_t-\theta_1X_1+\theta_\infty,
\end{equation}
with
\[\theta_i(a)=a_ia_\infty+a_ja_k, \ \text{ for } i=0,t,1,\quad\text{and}\quad \theta_\infty(a)=a_0a_ta_1a_\infty+a_0^2+a_t^2+a_1^2+a_\infty^2-4.\]

\begin{definition}
We call \emph{the character variety of $P_{VI}$} the complex cubic surface 
\begin{equation}\label{eq:wm-S}
 \Cal S_{VI}(\theta)=\{X\in\C^3: F(X,\theta)=0\}.
\end{equation}
\end{definition}

In this setting, the Riemann--Hilbert correspondence 
can be seen as a map between the Hamiltonian flow of the Painlevé system 
(defined on the Okamoto fibration $\Cal M_{VI}(\vartheta)\to\CP^1\sminus\{0,1,\infty\}$)
on one side and a locally constant flow on the character variety on the other side
\[RH_{VI,t}: \Cal M_{VI,t}(\vartheta)\to \Cal S_{VI}(\theta).\]
Under this correspondence the Okamoto space of initial conditions $\Cal M_{VI,t}(\vartheta)$ is a minimal resolution of singularities of $\Cal S_{VI}(\theta)$
\cite{IIS}.

The character variety $\Cal S_{VI}(\theta)$ is equipped with a natural algebraic symplectic form %given by the \emph{Poincaré residue}
\begin{equation}\label{eq:wm-Omega}
\omega_{\Cal S_{VI}}=\frac{dX_t\wedge dX_0}{2\pi i F_{\!X_1}}=\frac{dX_1\wedge dX_t}{2\pi i F_{\!X_0}}=\frac{dX_0\wedge dX_1}{2\pi i F_{\!X_t}},
\end{equation}
where 
\[\,F_{\!X_i}=\tfrac{dF}{dX_i}=X_jX_k+2X_i-\theta_i.\] 
%and any other algebraic symplectic form on  $\Cal S_{VI}(\theta)$ is proportionate to it \cite[Corollary 5.2]{Obl04}.
The Poisson bracket associated to $-2\pi i\,\omega_{\Cal S_{VI}}$ is the \emph{Goldman bracket}
\[\{X_i,X_j\}=F_{\!X_k}, \quad\text{$(i,j,k)$ a cyclic permutation of $(0,t,1)$.}\]

\begin{proposition}\label{prop:wm-omega}
The standard symplectic form $\omega$ \eqref{eq:wm-omega} on the $(q,p)$-space corresponds through the Riemann-Hilbert correspondence to the symplectic form \eqref{eq:wm-Omega}. 
\end{proposition}

\begin{proposition}[\emph{Jimbo's asymptotic formula} \cite{Jim}]\label{prop:wm-Jimbo}
Given a monodromy representation $\rho$ as above and its associated coordinate $X=(X_0,X_t,X_1)$ \eqref{eq:wm-X} on the character variety, 
then the solution of $P_{VI}(\vartheta)$ corresponding to a point $X\in\Cal S_{VI}(\theta)$ has the following asymptotics when $t\to 0$ in the sector $|\arg t|<\pi$:
\begin{equation}\label{eq:wm-Jimbo}
\begin{aligned}
t^{\sigma_1-1}q&=\alpha(X,\vartheta)+ O(t^{\sigma_1})+ O(t^{1-\sigma_1})\\
%+\tfrac{\vartheta_0^2-\vartheta_t^2+\sigma_1^2}{2\sigma_1^2}t^{\sigma_1}+ \tfrac{(\sigma_1^2-(\vartheta_0-\vartheta_t)^2)(\sigma_1^2-(\vartheta_0+\vartheta_t)^2)}{\sigma_1^2\alpha(X,\vartheta)}t^{2\sigma_1}
t^{1-\sigma_1}p&=\frac{\vartheta_0+\vartheta_t-\sigma_1}{2\alpha(X,\vartheta)}+O(t^{\sigma_1})+O(t^{1-\sigma_1}),
\end{aligned}
\end{equation}
where $\sigma_1\neq 0$ is such that $2\cos\pi\sigma_1=X_1$ and $0\leq\Re\sigma_1<1$, 
\[\alpha=\frac{(\vartheta_0+\vartheta_t+\sigma_1)(-\vartheta_0+\vartheta_t+\sigma_1)(\vartheta_\infty+\vartheta_1+\sigma_1)\cdot d(\sigma_1,\vartheta)}
{4\sigma_1^2\,(\vartheta_\infty+\vartheta_1-\sigma_1)\cdot c(\sigma_1,\vartheta)\,[a(\sigma_1,X_0,\vartheta)+b(\sigma_1,X_1,\vartheta)]},\]
with
\begin{align*}
a&=\tfrac{1}{4}e^{\pi i\sigma_1}[2i\sin\pi\sigma_1\cdot X_0-\theta_t]=\tfrac{1}{4}[(e^{\pi i\sigma_1}X_1-2)X_0-\theta_t],\\
b&=\tfrac{1}{4}[2i\sin\pi\sigma_1\cdot X_t+\theta_0]=\tfrac{1}{4}[(2e^{\pi i\sigma_1}-X_1)X_t+\theta_0],\\
c&=\frac{\Gamma(1-\sigma_1)^2 \Gamma\big(1+\tfrac{\vartheta_0+\vartheta_t+\sigma_1}{2}\big) \Gamma\big(1+\tfrac{-\vartheta_0+\vartheta_t+\sigma_1}{2}\big)
\Gamma\big(1+\tfrac{\vartheta_\infty+\vartheta_1+\sigma_1}{2}\big) \Gamma\big(1+\tfrac{-\vartheta_\infty+\vartheta_1+\sigma_1}{2}\big)}
{\Gamma(1+\sigma_1)^2 \Gamma\big(1+\tfrac{\vartheta_0+\vartheta_t-\sigma_1}{2}\big) \Gamma\big(1+\tfrac{-\vartheta_0+\vartheta_t-\sigma_1}{2}\big)
\Gamma\big(1+\tfrac{\vartheta_\infty+\vartheta_1-\sigma_1}{2}\big) \Gamma\big(1+\tfrac{-\vartheta_\infty+\vartheta_1-\sigma_1}{2}\big)},\\
d&=4\sin\tfrac{\pi(\vartheta_0+\vartheta_t-\sigma_1)}{2} \sin\tfrac{\pi(-\vartheta_0+\vartheta_t-\sigma_1)}{2}
\sin\tfrac{\pi(\vartheta_\infty+\vartheta_1-\sigma_1)}{2} \sin\tfrac{\pi(-\vartheta_\infty+\vartheta_1-\sigma_1)}{2},
\end{align*}
under the assumption that
\[\vartheta_0,\ \vartheta_t, \vartheta_1, \vartheta_\infty, 
\tfrac{\vartheta_0+\vartheta_t\pm\sigma_1}{2},\ \tfrac{-\vartheta_0+\vartheta_t\pm\sigma_1}{2},\
\tfrac{\vartheta_\infty+\vartheta_1\pm\sigma_1}{2},\ \tfrac{-\vartheta_\infty+\vartheta_1\pm\sigma_1}{2}\ \notin\Z, \]
see \cite[p.191]{Bo05}. 	
\end{proposition}

\begin{proof}[Proof of Proposition \ref{prop:wm-omega}]
We will use the Jimbo's asymptotic formula \eqref{eq:wm-Jimbo}. 
Since the point $X\in\Cal S_{VI}(\theta)$ corresponding to a given solution $\left(q(t),p(t)\right)$ of $P_{VI}$ is locally independent of $t$, we can take limit $t\to 0$ and ignore higher order terms. Restricting to the subvariety where $\Re\sigma_1\neq0$, we have
\begin{equation*}
\omega=dq\wedge dp=\frac{d\sigma_1\wedge d\alpha}{2\alpha}
=\frac{dX_1\wedge(da+db)}{4(a+b)\pi\sin\pi\sigma_1}
=-\frac{dX_1\wedge(e^{\pi i\sigma_1}dX_0+dX_t)}{2\pi i(e^{\pi i\sigma_1}F_{\!X_t}-F_{\!X_0})}=\frac{dX_1\wedge dX_t}{2\pi i F_{\!X_0}},
\end{equation*}
using that 
\begin{equation}\label{eq:wm-ab}
4(a+b)=e^{\pi i\sigma_1}F_{\!X_t}-F_{\!X_0},	
\end{equation}
and the identity $\frac{dX_1\wedge dX_0}{F_{\!X_t}}=\frac{dX_t\wedge dX_1}{F_{\!X_0}}$.
\end{proof}

\subsection{Lines and singularities of $\Cal S_{VI}(\theta)$}\label{paragraph:wm-singularitiesSVI}

The projective completion of the character variety  $\Cal S_{VI}(\theta)$ in $\CP^3$ is a compact cubic surface. In the classical theory of classification of cubic surfaces
a major role is played by the configurations of complex lines inside the surface.\footnote{The author is grateful to E.~Paul and J.P.~Ramis for pointing out the major significance of these lines in the theory of Painlev\'e equations and for illuminating discussions on this subject.} We have: 
 
\begin{proposition}[Lines of $\Cal S_{VI}(\theta)$]\label{prop:wm-lines}
The Fricke polynomial $F(X,\theta)$  \eqref{eq:wm-F} can be decomposed as
\begin{align*}
F(X,\theta)=&\,
(X_k-\tfrac{e_i}{e_j}-\tfrac{e_j}{e_i})(F_{\!X_k}-X_k+\tfrac{e_i}{e_j}+\tfrac{e_j}{e_i})\\
&-\tfrac{1}{e_ie_j}(e_iX_i+e_jX_j-a_\infty-e_ie_ja_k) (e_iX_j+e_jX_i-a_k-e_ie_ja_\infty),\\[3pt]
=&\, (X_k-e_ie_j-\tfrac{1}{e_ie_j})(F_{\!X_k}-X_k+e_ie_j+\tfrac{1}{e_ie_j})\\
&-\tfrac{1}{e_ie_j}(e_ie_jX_i+X_j-e_ja_\infty-e_ia_k) (e_ie_jX_j+X_i-e_ja_k-e_ia_\infty),\\[3pt]
=&\, (X_k-\tfrac{e_k}{e_\infty}-\tfrac{e_\infty}{e_k})(F_{\!X_k}-X_k+\tfrac{e_k}{e_\infty}+\tfrac{e_\infty}{e_k})\\
&-\tfrac{1}{e_ke_\infty}(e_\infty X_i+e_kX_j-a_i-e_ke_\infty a_j) (e_k X_i+e_\infty X_j-a_j-e_k e_\infty a_i),\\[3pt]
=&\, (X_k-e_k e_\infty-\tfrac{1}{e_k e_\infty})(F_{\!X_k}-X_k+e_k e_\infty+\tfrac{1}{e_k e_\infty})\\
&-\tfrac{1}{e_ke_\infty}(X_j+e_ke_\infty X_i-e_ka_i-e_\infty a_j) (X_i+e_ke_\infty X_j-e_ka_j-e_\infty a_i).
\end{align*}
In particular, the following 24 lines (counted with multiplicity) are contained in $\Cal S_{VI}(\theta)$:
\begin{align*}
&\{X_k=\tfrac{e_i}{e_j}+\tfrac{e_j}{e_i}, \quad e_iX_i+e_jX_j=a_\infty +e_ie_ja_k\},\\
&\{X_k=\tfrac{e_i}{e_j}+\tfrac{e_j}{e_i}, \quad e_iX_j+e_jX_i=a_k +e_ie_ja_\infty\},\\
&\{X_k=e_ie_j+\tfrac{1}{e_ie_j}, \quad X_i+e_ie_jX_j=e_ja_k+e_ia_\infty\},\\
&\{X_k=e_ie_j+\tfrac{1}{e_ie_j}, \quad X_j+e_ie_jX_i=e_ja_\infty+e_ia_k\},\\
&\{X_k=\tfrac{e_k}{e_\infty}+\tfrac{e_\infty}{e_k},\quad  e_\infty X_i+e_kX_j=a_i+e_ke_\infty a_j\},\\
&\{X_k=\tfrac{e_k}{e_\infty}+\tfrac{e_\infty}{e_k},\quad  e_k X_i+e_\infty X_j=a_j+e_k e_\infty a_i\},\\
&\{X_k=e_k e_\infty+\tfrac{1}{e_k e_\infty},\quad X_i+e_ke_\infty X_j=e_ka_j+e_\infty a_i\},\\
&\{X_k=e_k e_\infty+\tfrac{1}{e_k e_\infty},\quad X_j+e_ke_\infty X_i=e_ka_i+e_\infty a_j\}.
\end{align*}
For each pair $l,m\in\{0,t,1,\infty\}$ each of the two planes 
\begin{equation}\label{eq:wm-lines}
\begin{cases} X_n=e_le_m+\tfrac{1}{e_le_m},\\ X_n=\tfrac{e_l}{e_m}+\tfrac{e_m}{e_l},\end{cases}\qquad
\text{with }\ (l,m,n)=\begin{cases}(i,j,k) \\ (k,\infty,k) \end{cases},
\end{equation}
intersects $\Cal S_{VI}(\theta)$ at 2 lines. The resulting 4 lines correspond to the reducibility of the pair of matrices $\{M_l,M_m\}$, i.e. to the existence of common invariant space for the pair.
\end{proposition}

\begin{remark}\label{remark:wm-mixedbasis}
To be more precise, if for each $l=0,t,1,\infty$ the monodromy matrix $M_l$ is diagonalizable, then there exists a pair of invariant subspaces of the space of solutions of \eqref{eq:wm-A}, giving rise to a basis of solutions (sometimes called Levelt basis) with respect to which $M_l$ is diagonal. 
When each of these basic solutions is analytically continued towards the base-point $z_0$ of $\pi_1(\CP^1\sminus\{0,t,1,\infty\},z_0)$ along the path encircled by the loop $\gamma_i$ defining the monodromy (Figure~\ref{figure:wm-loops}),
then for every pair $\{M_l,M_m\}$ there are 4 different possibi lities how one  can form a \emph{mixed basis} out of the two pairs of solutions. 
Each of the 4 lines at the intersection of $\Cal S_{VI}(\theta)$ with the planes \eqref{eq:wm-lines} correspond to the degeneracy of one of these 4 mixed bases.
\end{remark}

The projective completion of $\Cal S_{VI}(\theta)$ in $\CP^3$ contains 3 additional lines at infinity, giving the total of 27 lines provided by the classical Cayley--Salmon theorem \cite{Cay69,BW79}. They are all distinct if and only if $\Cal S_{VI}(\theta)$ is non-singular.

\begin{remark}
By the classic theory of cubic surfaces, there are 45 tri-tangent planes, i.e. containing 3 lines forming a triangle, and each line belongs to exactly 5 of these planes. The above 12 decompositions of the cubic corresponds to the $3\times 4$ planes that contain one of the 3 lines at infinity, the fifth plane being the plane at infinity.
\end{remark}

\smallskip

\paragraph{Singular points of $\Cal S_{VI}(\theta)$.}
The surface $\Cal S_{VI}(\theta)$ is simply connected (cf. \cite{CL}), and it may or may not have singular points depending on $a$, but it never has more than 4 singular points \cite[Corollary 4.6]{Obl04}.
The singularities that appear correspond to unstable monodromy representations, which are of two kinds:
\begin{itemize}
\item Either $M_l=\pm I$ for some $l\in\{0,t,1,\infty\}$, hence $e_l=\pm 1$. 

\smallskip
If $l=i\in\{0,t,1\}$, then $\displaystyle\ a_i=\pm 2,\quad X_i=\pm a_\infty,\quad X_j=\pm a_k,\quad X_k=\pm a_j.$
%and the singularity lies on the intersection 
%\[L_j(a)\cap L_j^\natural(a)\cap L_k(a)\cap L_k^\natural(a)\cap J_j(a)\cap J_j^\natural(a)\cap J_k(a)\cap J_k^\natural(a).\]

\smallskip
If $l=\infty$, then $\displaystyle \ a_\infty=\pm 2,\quad X_i=\pm a_i,\ i=0,t,1.$

\item Or the representation is reducible, in which case $\,M_l=\left(\begin{smallmatrix}e_l^{\delta_l} & *\\ 0&e_l^{-\delta_l}\end{smallmatrix}\right),\,$ $l=0,t,1,\infty$, 
for some quadruple of signs
$(\delta_0,\delta_t,\delta_1)\in\{\pm 1\}^3$, $\delta_\infty=1$, and
\[e_0^{\delta_0}e_t^{\delta_t}e_1^{\delta_1}e_\infty=1,\qquad X_i=e_j^{\delta_j}e_k^{\delta_k}+e_j^{-\delta_j}e_k^{-\delta_k}.\]
%If $\delta_i+\delta_j=0$, then the singularity lies on the intersection $L_k(a)\cap L_k^\natural(a)$.\\
%If $\delta_i-\delta_j=0$, then the singularity lies on the intersection $J_k(a)\cap J_k^\natural(a)$.
\end{itemize}
The surface is therefore singular if and only if 
\[\prod_{l\in\{0,t,1,\infty\}}\!\!\!\!(a_l^2-4)\cdot w(a)=0,\]
where
\begin{align*}
w(a)&:=(a_0\!+a_t\!+a_1\!+a_\infty)(a_0\!+a_\infty\!-a_t\!-a_1)(a_t\!+a_\infty\!-a_1\!-a_0)(a_1\!+a_\infty\!-a_0\!-a_t)\,-\\
&\qquad -(a_0a_\infty\!-a_ta_1)(a_ta_\infty\!-a_1a_0)(a_1a_\infty\!-a_0a_t)\\
&=\frac{(e_0e_te_1e_\infty-1)(e_0e_t-e_1e_\infty)(e_te_1-e_0e_\infty)(e_1e_0-e_te_\infty)}{(e_0e_te_1e_\infty)^4}\!\!\prod_{l\in\{0,t,1,\infty\}}\!\!\!\!(e_l-e_ie_je_k),
\end{align*}
see \cite{I2}.
All the singularities of the projective completion of $\Cal S_{VI}(\theta)$ are contained in its finite part, where they are situated on the intersection of several lines.

The \emph{singular locus} of $\Cal S_{VI}(\theta)$ corresponds through the Riemann--Hilbert correspondence to so called \emph{Riccati solutions}  of $P_{VI}$ \cite{IIS}.
% Let $\Cal S_{VI}^\circ(\theta)=\Cal S_{VI}(\theta)\sminus\text{Sing}(\Cal S_{VI}(\theta))$ be the smooth locus.

\subsection{The braid group action on $\Cal S_{VI}(\theta)$}

The nonlinear \emph{monodromy action} on the space of $\SL_2(\C)$-monodromy representations, 
is given by the action of moving $t$ along loops in $\CP^1\sminus\{0,1,\infty\}$ 
while keeping the representation constant.
When $t$ returns to the initial position $t_0$, the loops generating $\pi_1(\CP^1\sminus\{0,t,1,\infty\},z_0)$
will not be the same as before. It induces an automorphism of the fundamental group through which it acts on the space of monodromy representations.
The movement of $t$ can be also seen as an action of the pure-braid group $\Cal P_3$ on three strands $(0,t,1)$,
generated by the pure braids $\beta_{0t}^2, \beta_{t1}^2\in\Cal P_3$ (Figure \ref{figure:wm-purebraids}).
\begin{figure}[h!]
\begin{center}
\begin{tikzpicture} [scale=0.25] 
 \draw (-3,1.5) node {$\beta_{0t}^2:$};
 \draw (0,-1) node {$0$};
 \draw (2,-1) node {$t$};
 \draw (4,-1) node {$1$};
 \draw (0,4) node {$0$};
 \draw (2,4) node {$t$};
 \draw (4,4) node {$1$};
% \draw[white,line width=5pt] (13,5) .. controls +(0,-2) and +(0,2) .. (6,0);
 \draw (-0.6,1.5) .. controls +(0,-1) and +(0,1) .. (2,0);
 \draw[white,line width=4pt] (0,1.5) .. controls +(0,-1) and +(0,1) .. (0,0);
 \draw (0,1.5) .. controls +(0,-1) and +(0,1) .. (0,0);
 \draw[<-] (0,3) .. controls +(0,-1) and +(0,1) .. (0,1.5);
 \draw[white,line width=4pt] (2,3) .. controls +(0,-1) and +(0,1) .. (-0.6,1.5);
 \draw[<-] (2,3) .. controls +(0,-1) and +(0,1) .. (-0.6,1.5);
 \draw[<-] (4,3) -- (4,0);
\end{tikzpicture}\qquad\qquad
\begin{tikzpicture} [scale=0.25] 
 \draw (-3,1.5) node {$\beta_{t1}^2:$};
 \draw (0,-1) node {$0$};
 \draw (2,-1) node {$t$};
 \draw (4,-1) node {$1$};
 \draw (0,4) node {$0$};
 \draw (2,4) node {$t$};
 \draw (4,4) node {$1$};
% \draw[white,line width=5pt] (13,5) .. controls +(0,-2) and +(0,2) .. (6,0);

 \draw (4,1.5) .. controls +(0,-1) and +(0,1) .. (4,0);
  \draw[white,line width=4pt] (4.6,1.5) .. controls +(0,-1) and +(0,1) .. (2,0);
  \draw (4.6,1.5) .. controls +(0,-1) and +(0,1) .. (2,0);
  \draw[<-] (2,3) .. controls +(0,-1) and +(0,1) .. (4.6,1.5);
  \draw[white,line width=4pt] (4,3) .. controls +(0,-1) and +(0,1) .. (4,1.5);
  \draw[<-] (4,3) .. controls +(0,-1) and +(0,1) .. (4,1.5);
  \draw[<-] (0,3) -- (0,0);
\end{tikzpicture}
\end{center}
\vskip-12pt\caption{Elementary pure braids.}
\label{figure:wm-purebraids} 
\end{figure}
These actions were considered by Dubrovin \cite{Dub96,Dub99} and described in detail by Dubrovin and Mazzocco \cite{DM, Maz01} and Iwasaki \cite{I},
and their dynamics was further studied by Cantat \& Loray \cite{CL} in connection to the problem of transcendentness of $P_{VI}$.

It is advantageous to consider the elementary monodromy actions as square iterates of ``half-monodromy'' actions.
That way one considers the action of the whole braid group $\Cal B_3$ on three strands $(0,t,1)$, generated by two braids  $\beta_{0t}$, $\beta_{t1}$
(Figure \ref{figure:wm-braids}). 
However such ``half-monodromy'' actions don't act on the given Painlev\'e foliation, but rather on a whole set of such foliations since they change the parameter $\vartheta$ by permutation of its components.
As such there are different ways how to define such actions and how to interpret them. 

Instead of the four singularities $(0,t,1,\infty)$, let us consider an ordered quadruple of distinct points 
$(t_i,t_j,t_l,t_m)$ in $\CP^1$, where $(i,j,l,m)$ is a cyclic permutation of $(0,t,1,\infty)$.
The action of the braid $\beta_{ij}$ consists in turning the two  points $t_i,t_j$, around each other by a half turn while fixing $t_l,t_m$.
\begin{figure}[h!]
\vskip-6pt
\begin{center}
\begin{tikzpicture} [scale=0.25] 
	\draw (-3,1.5) node {$\beta_{ij}:$};
	\draw (0,-1) node {$t_i$};
	\draw (2,-1) node {$t_j$};
	\draw (4,-1) node {$t_l$};
	\draw (6,-1) node {$t_m$};
	\draw (0,4) node {$t_j'$};
	\draw (2,4) node {$t_i'$};
	\draw (4,4) node {$t_l'$};
	\draw (6,4) node {$t_m'$};
	% \draw[white,line width=5pt] (13,5) .. controls +(0,-2) and +(0,2) .. (6,0);
	\draw[<-] (0,3) .. controls +(0,-1) and +(0,1) .. (2,0);
	\draw[white,line width=4pt] (2,3) .. controls +(0,-1) and +(0,1) .. (0,0);
	\draw[<-] (2,3) .. controls +(0,-1) and +(0,1) .. (0,0);
	\draw[<-] (4,3) -- (4,0);
	\draw[<-] (6,3) -- (6,0);
\end{tikzpicture}
\end{center}
\vskip-12pt\caption{Elementary braid $\beta_{ij}$, $(i,j,l,m)$ is a cyclic permutation of $(0,t,1,\infty)$.}
\label{figure:wm-braids} 
\end{figure}

The braids act on the fundamental group $\pi_1(\CP^1\sminus\{t_i,t_j,t_l,t_m\},z_0)$ by transforming the loops, \
$\beta:\gamma\mapsto\gamma^\beta,$ %\qquad \big(\gamma^\beta\big)^{\beta'}=\gamma^{\beta\beta'},

\hbox{\hskip0.1\textwidth
\vtop{\begin{tikzpicture} [scale=0.25]
	
	\node at (-5,0)[circle,fill,inner sep=1pt]{}; \draw (-5,-1) node {$t_i$};
	\node at (-3,0)[circle,fill,inner sep=1pt]{}; \draw (-3,-1) node {$t_j$}; 
	\node at (-1,0)[circle,fill,inner sep=1pt]{}; \draw (-1,-1) node {$t_l$};
	\node at (1,0)[circle,fill,inner sep=1pt]{}; \draw (1,-1) node {$t_m$}; 
	\draw (-2,6) node {$z_0$};  \draw[-] (-5,0) -- (-2,5); \draw[-] (-3,0) -- (-2,5);  \draw[-] (-1,0) -- (-2,5); \draw[-] (1,0) -- (-2,5);
	\draw[->] (3,2) -- (6,2); \draw (4.5,3) node {$\beta_{ij}$};
	\node at (9,0)[circle,fill,inner sep=1pt]{}; \draw (9,-1) node {$t_j'$};
    \node at (11,0)[circle,fill,inner sep=1pt]{}; \draw (11.5,-1) node {$t_i'$}; 
   	\node at (13,0)[circle,fill,inner sep=1pt]{}; \draw (13,-1) node {$t_l'$};
    \node at (15,0)[circle,fill,inner sep=1pt]{}; \draw (15,-1) node {$t_m'$}; 
	\draw (12,6) node {$z_0$}; 
	\draw[-] (9,0) -- (12,5); \draw[-] (7.5,0) .. controls +(0,1.5) and +(-1,-1) .. (12,5); \draw[-] (11,0) .. controls +(0,-3) and +(0,-3) .. (7.5,0);
	\draw[-] (13,0) -- (12,5); \draw[-] (15,0) -- (12,5);
	\end{tikzpicture}}	
\hskip-0.9\textwidth	
\vbox{\begin{equation*}
	\begin{aligned}
	\beta_{ij}:\
	\gamma_i &\mapsto \gamma_i'=\gamma_i\gamma_j\gamma_i^{-1} ,\qquad\\
	\gamma_j &\mapsto \gamma_j'=\gamma_i,\\
	\gamma_l &\mapsto \gamma_l'=\gamma_l,\\
	\gamma_m &\mapsto \gamma_m'=\gamma_m,\\
	\end{aligned}
	\end{equation*}}
}
\noindent
(in the above picture we draw just the connecting paths from $z_0$ to the singularities $i,j$ that are encircled by the loops $\gamma_i$, $\gamma_j$),
preserving the relation $\gamma_i\gamma_j\gamma_l\gamma_m=\id$.

This in turn induces an action $\beta_*:\rho\mapsto\rho^\beta$ on the monodromy representation \eqref{eq:PV-rho} defined by $\quad \rho^\beta(\gamma^\beta)=\rho(\gamma)$,
and satisfying
$(\beta\beta')_*=\beta_*\circ\beta'_*:\rho\mapsto\rho^{\beta\beta'}=(\rho^{\beta})^{\beta'}$
is given by
\begin{equation}\label{eq:wm-beta}
\begin{aligned}
(\beta_{ij})_*:\
M_i &\mapsto \rho^{\beta_{ij}}(\gamma_i)=M_j , &\qquad\qquad
	(\beta_{ji})_*=(\beta_{ij}^{-1})_*:\ M_i &\mapsto M_i^{-1}M_jM_i,\\ 
M_j &\mapsto \rho^{\beta_{ij}}(\gamma_j)=M_jM_{i}M_j^{-1},&
	M_j &\mapsto M_i,\\
M_l &\mapsto \rho^{\beta_{ij}}(\gamma_l)= M_l,&
	M_l &\mapsto  M_l,\\
M_m &\mapsto \rho^{\beta_{ij}}(\gamma_m)= M_m,&
	M_m &\mapsto M_m.
\end{aligned}
\end{equation}
%
%\begin{equation}\label{eq:wm-betalm}
%	\begin{aligned}
%		(\beta_{lm})_*:\
%		M_i &\mapsto M_i,\\
%		M_j &\mapsto M_j,\\ 
%		M_l &\mapsto  M_m,\\
%		M_m &\mapsto M_mM_lM_m^{-1},
%	\end{aligned}
%\end{equation}

\emph{Then the pure braid actions of the two square iterates  $\beta_{ij}^{\circ 2}$ and $\beta_{lm}^{\circ 2}$
are conjugated, and therefore  induce the same action on the character variety.}
However the braid actions of $\beta_{ij}$ and $\beta_{lm}$ are different.

Now for either of $\{i,j\}=\begin{cases}\{0,t\}\\\{t,1\}\end{cases}$ let $k$ be the third finite index, $\{i,j,k\}=\{0,t,1\}$.
Then the map induced by $\beta_{ij}$ between the character varieties is given by
\begin{equation}\label{eq:wm-gij}
\begin{aligned}
g_{ij}:\ \Cal S_{VI}(\theta)&\to \Cal S_{VI}\big(g_{ij}(\theta)\big),\hskip-36pt&&&&\\[4pt]
a_i&\mapsto a_j, \quad& X_i&\mapsto X_j-F_{\!X_j},    \quad& F_{\!X_i}&\mapsto -F_{\!X_j}, \\
a_j&\mapsto a_i, \quad& X_j&\mapsto X_i,\quad& F_{\!X_j}&\mapsto F_{\!X_i}-F_{\!X_j}X_k, \\
a_k&\mapsto a_k, \quad& X_k&\mapsto X_k,	\quad& F_{\!X_k}&\mapsto F_{\!X_k}-F_{\!X_j}X_i,\\
a_\infty&\mapsto a_\infty, &&&&
\end{aligned}
\end{equation}
and hence $(\theta_i,\theta_j,\theta_k,\theta_\infty)\mapsto(\theta_j,\theta_i,\theta_k,\theta_\infty)$ where $\theta$ are the coefficients in \eqref{eq:wm-F}.
Moreover in this notation:
\[g_{ij}^{\circ -1}=g_{ji}.\]

On the other hand for either of $\{l,m\}=\begin{cases}\{1,\infty\}\\\{\infty,0\}\end{cases}$ let $k=\begin{cases}1\\ 0\end{cases}$ be the finite of them, and
$\{i,j\}=\begin{cases}\{0,t\}\\\{t,1\}\end{cases}$ the remainig two indices.
Then the induced map by $\beta_{lm}$ between the character varieties is given by
\begin{equation}\label{eq:wm-glm}
	\begin{aligned}
		g_{lm}:\ \Cal S_{VI}(\theta)&\to \Cal S_{VI}\big(g_{lm}(\theta)\big),\hskip-36pt&&&&\\[4pt]
		a_i&\mapsto a_i, \quad& X_i&\mapsto X_j-F_{\!X_j},    \quad& F_{\!X_i}&\mapsto -F_{\!X_j}, \\
		a_j&\mapsto a_j, \quad& X_j&\mapsto X_i,\quad& F_{\!X_j}&\mapsto F_{\!X_i}-F_{\!X_j}X_k, \\
		a_l&\mapsto a_m, \quad& X_k&\mapsto X_k,	\quad& F_{\!X_k}&\mapsto F_{\!X_k}-F_{\!X_j}X_i,\\
		a_m&\mapsto a_l, &&&&
	\end{aligned}
\end{equation}
and hence $(\theta_i,\theta_j,\theta_k,\theta_\infty)\mapsto(\theta_j,\theta_i,\theta_k,\theta_\infty)$.
It is the same as $g_{ij}$ \eqref{eq:wm-gij} except for the way it acts on $a$, where it differs from $g_{ij}$  by composition with the involutive permutation \[(a_i,a_j,a_l,a_m)\mapsto (a_j,a_i,a_m,a_l)\] 
which preserves $\theta(a)$ and commutes with both $g_{ij}$ and $g_{lm}$. 
This subtle difference between $g_{ij}$ \eqref{eq:wm-gij} and $g_{lm}$ \eqref{eq:wm-glm} will be amplified during the confluence.
Nevertheless, the square iterates of both actions are equal, 
\[g_{0t}^{\circ 2}=g_{1\infty}^{\circ 2}, \qquad g_{t1}^{\circ 2}=g_{\infty0}^{\circ 2}.\]

Namely (for the sake of reference):
\begin{equation}\label{eq:wm-ginfty0}
	\begin{aligned}
		g_{t1},\, g_{\infty 0}:\ 		
		\theta_0&\mapsto \theta_0, \quad& X_0&\mapsto X_0,	\quad& F_{\!X_0}&\mapsto F_{\!X_0}-F_{\!X_1}X_t,\\
		\theta_t&\mapsto \theta_1, \quad& X_t&\mapsto X_1-F_{\!X_1},    \quad& F_{\!X_t}&\mapsto -F_{\!X_1}, \\
		\theta_1&\mapsto \theta_t, \quad& X_1&\mapsto X_t,\quad& F_{\!X_1}&\mapsto F_{\!X_t}-F_{\!X_1}X_0, \\
		\theta_\infty&\mapsto \theta_\infty, &&&&
	\end{aligned}
\end{equation}
$g_{t1}: a\mapsto (a_0,a_1,a_t,a_\infty)$, $g_{\infty0}: a\mapsto (a_\infty,a_t,a_1,a_0)$.

%They satisfy
%\begin{equation*}%\label{eq:wm-modular}
%	g_{ij}=g_{ji}^{\circ -1}, \quad 
%	g_{ij}\circ g_{jk}\circ g_{ij}=g_{jk}\circ g_{ij}\circ g_{jk}, \quad
%	g_{ki}=g_{ji}\circ g_{jk}\circ g_{ij},\quad 
%	(g_{ij}\circ g_{jk})^{\circ 3}=\id,
%\end{equation*}

\goodbreak

\begin{proposition}[Dubrovin \& Mazzocco \cite{DM}, Iwasaki \cite{I,I2}]\label{prop:wm-modular}~
\begin{enumerate}[wide=0pt, leftmargin=\parindent]
\item  
For any permutation $(i,j,k)$ of $(0,t,1)$ the above \emph{half-monodromy} actions $g_{ij}:\Cal S_{VI}(\theta)\to \Cal S_{VI}\big(g_{ij}(\theta)\big)$
preserve the Fricke relation \eqref{eq:wm-Fricke}: $F\circ g_{ij}=F$, and therefore also the 2-form $\omega_{\Cal S_{VI}}$. 
The group $\Gamma=\langle g_{0t},g_{t1}\rangle$ generated by the actions of the braids $\beta_{0t}$ and $\beta_{t1}$ (generators of $\Cal B_3$)
is isomorphic to the \emph{modular group} $\PSL_2(\Z)$, with the standard generators
\begin{equation*}%\label{eq:wm-modularST}
S= g_{0t}^{\circ 2}\circ g_{t1}, \quad T=g_{0t}^{\circ(-1)}, \qquad\text{satisfying}\qquad S^{\circ 2}=\id=(T\circ S)^{\circ 3}. 
\end{equation*}

\item
The action of the \emph{monodromy group of $P_{VI}$} on the character variety $\Cal S_{VI}(\theta)$ is induced by the action of the pure braids $\beta_{0t}^2, \beta_{t1}^2\in\Cal P_3$ on the monodromy representations.
It is isomorphic to the  \emph{principal congruence subgroup} of the modular group
\[\Gamma(2)=\langle g_{0t}^{\circ 2},g_{t1}^{\circ 2}\rangle\subseteq\Aut_{\omega_{\Cal S_{VI}}}(\Cal S_{VI}(\theta)),\]
where
\begin{equation}\label{eq:wm-gij2}
\begin{aligned}
	g_{ij}^{\circ 2}:	  X_i&\mapsto X_i-F_{\!X_i}+X_kF_{\!X_j},	\ & 
	\scriptstyle F_{\!X_i}&\begin{scriptstyle}\mapsto -F_{\!X_i}+X_kF_{\!X_j}\end{scriptstyle}, \\
	X_j&\mapsto X_j-F_{\!X_j},		\quad& 
	\scriptstyle F_{\!X_j}&\begin{scriptstyle}\mapsto -F_{\!X_j}-X_kF_{\!X_i}+X_k^2F_{\!X_j}\end{scriptstyle}, \\
	X_k&\mapsto X_k,    			\quad& 
	\scriptstyle F_{\!X_k}&\begin{scriptstyle}\mapsto F_{\!X_k}-X_iF_{\!X_j}-X_jF_{\!X_i}+F_{\!X_j}F_{\!X_i}+X_kX_jF_{\!X_j}-X_kF_{\!X_j}^2\end{scriptstyle}.\hskip-12pt
\end{aligned}	
\end{equation}

% \[\mbox{\scriptsize$
% \begin{aligned}
% g_{t0}^{\circ 2}: \quad X_0&\mapsto X_0-F_{\!X_0},    		\quad& F_{\!X_0}&\mapsto -F_{\!X_0}-X_1F_{\!X_t}+X_1^2F_{\!X_0}, \\
%  		  \quad X_t&\mapsto X_t-F_{\!X_t}+X_1F_{\!X_0},  \quad& F_{\!X_t}&\mapsto -F_{\!X_t}+X_1F_{\!X_0}, \\
%  		  \quad X_1&\mapsto X_1,				\quad& F_{\!X_1}&\mapsto F_{\!X_1}-X_tF_{\!X_0}-X_0F_{\!X_t}+F_{\!X_0}F_{\!X_t}+X_1X_0F_{\!X_0}-X_1F_{\!X_0}^2,\\[6pt]
% g_{1t}^{\circ 2}: \quad X_0&\mapsto X_0,    			\quad& F_{\!X_0}&\mapsto F_{\!X_0}-X_1F_{\!X_t}-X_tF_{\!X_1}+F_{\!X_t}F_{\!X_1}+X_0X_tF_{\!X_t}-X_0F_{\!X_t}^2, \\
%  		  \quad X_t&\mapsto X_t-F_{\!X_t},			\quad& F_{\!X_t}&\mapsto -F_{\!X_t}-X_0F_{\!X_1}+X_0^2F_{\!X_t}, \\
%  		  \quad X_1&\mapsto X_1-F_{\!X_1}+X_0F_{\!X_t},	\quad& F_{\!X_1}&\mapsto -F_{\!X_1}+X_0F_{\!X_t},
% \end{aligned}$}\] 
while preserving the parameter $a=(a_0,a_t,a_1,a_\infty)$.

The fixed points of this $\Gamma(2)$-action are exactly the singularities of $\Cal S_{VI}$, and
its restriction on the smooth locus of $\Cal S_{VI}(\theta)$ represents faithfully the nonlinear monodromy action on the non-Riccati
locus of space of initial conditions $\Cal M_{VI,t_0}(\vartheta)$ (i.e. on the initial conditions corresponding to non-Riccati solutions of $P_{VI}$). 
\end{enumerate}
\end{proposition}

As $t$ varies in $\CP^1\sminus\{0,1,\infty\}$ the character varieties define a fibration above $\CP^1\sminus\{0,1,\infty\}$ with fibers isomorphic $\Cal S_{VI}(\theta)$, for which the coordinates $X=(X_0,X_t,X_1)$ \eqref{eq:wm-X} are local trivializations.
The transformation maps $g_{0t}^{\circ2}$, $g_{t1}^{\circ2}$ are then the gluing maps when $t$ makes a round around $0$ or $1$, see Figure~\ref{figure:wm-glueing}.

\begin{figure}[t]
	\centering
	\begin{tikzpicture} [scale=0.25]
	\draw[dashed, ultra thick] (-20,0) -- (-6,0);
	\node at (-6,0)[circle,fill,inner sep=2pt]{}; \draw (-6,-1.5) node {$t\!=\!0$};
	\draw[->, thick] (-16.5,-2) .. controls +(-0.5,1) and +(-0.5,-1) .. (-16.5,2); \draw (-18,-1.5) node {$g_{0t}^{\circ 2}$};
	\draw[dashed, ultra thick] (6,0) -- (20,0);
	\node at (6,0)[circle,fill,inner sep=2pt]{}; \draw (6,-1.5) node {$t\!=\!1$};
	\draw[->, thick] (16.5,2) .. controls +(0.5,-1) and +(0.5,1) .. (16.5,-2); \draw (18.2,1.5) node {$g_{t1}^{\circ 2}$};
	
	\begin{scope}[shift={(0,2)}]
	\node at (-2,0)[circle,fill,inner sep=1pt]{}; \draw (-2,-1) node {$0$};
	\node at (0,0)[circle,fill,inner sep=1pt]{}; \draw (0,-1) node {$t$}; 
	\node at (2,0)[circle,fill,inner sep=1pt]{}; \draw (2,-1) node {$1$};
	\draw (0,3.6) node {$z_0$};  \draw[-] (-2,0) -- (0,3); \draw[-] (0,0) -- (0,3); \draw[-] (2,0) -- (0,3);
	\end{scope}
	
	\begin{scope}[shift={(13,-7)}]
	\node at (-2,0)[circle,fill,inner sep=1pt]{}; \draw (-2,-1) node {$0$};
	\node at (0,0)[circle,fill,inner sep=1pt]{}; \draw (0,-1) node {$1$};
	\node at (2,0)[circle,fill,inner sep=1pt]{}; \draw (2.5,-1) node {$t$}; 
	\draw (0,3.6) node {$z_0$}; 
	\draw[-] (-2,0) -- (0,3); 
	\draw[-] (0,0) -- (0,3); 
	\draw[-] (-1,0) .. controls +(1,3) and +(-1,-3) .. (0,3); \draw[-] (2,0) .. controls +(0,-3) and +(-1,-3) .. (-1,0);
	\end{scope}
	
	\begin{scope}[shift={(-13,-7)}]
	\node at (-2,0)[circle,fill,inner sep=1pt]{}; \draw (-2.5,-1) node {$t$}; 
	\node at (0,0)[circle,fill,inner sep=1pt]{}; \draw (0,-1) node {$0$};
	\node at (2,0)[circle,fill,inner sep=1pt]{}; \draw (2,-1) node {$1$};
	\draw (0,3.6) node {$z_0$}; 
	\draw[-] (2,0) -- (0,3); 
	\draw[-] (0,0) -- (0,3); 
	\draw[-] (1,0) .. controls +(-1,3) and +(1,-3) .. (0,3); \draw[-] (-2,0) .. controls +(0,-3) and +(1,-3) .. (1,0);
	\end{scope}
	
	\begin{scope}[shift={(-13,5)}]
	\node at (-2,0)[circle,fill,inner sep=1pt]{}; \draw (-2,-1) node {$t$};
	\node at (0,0)[circle,fill,inner sep=1pt]{}; \draw (0.2,-1) node {$0$}; 
	\node at (2,0)[circle,fill,inner sep=1pt]{}; \draw (2,-1) node {$1$};
	\draw (0,3.6) node {$z_0$}; 
	\draw[-] (2,0) -- (0,3); 
	\draw[-] (-2,0) -- (0,3); 
	\draw[-] (-3.5,0) .. controls +(0.5,1) and +(-2,-1.5) .. (0,3); \draw[-] (0,0) .. controls +(-1,-3) and +(-1,-2) .. (-3.5,0);
	\end{scope}
	
	\begin{scope}[shift={(13,5)}]
	\node at (-2,0)[circle,fill,inner sep=1pt]{}; \draw (-2,-1) node {$0$};
	\node at (0,0)[circle,fill,inner sep=1pt]{}; \draw (-0.2,-1) node {$1$}; 
	\node at (2,0)[circle,fill,inner sep=1pt]{}; \draw (2,-1) node {$t$};
	\draw (0,3.6) node {$z_0$}; 
	\draw[-] (-2,0) -- (0,3); 
	\draw[-] (2,0) -- (0,3); 
	\draw[-] (3.5,0) .. controls +(-0.5,1) and +(2,-1.5) .. (0,3); \draw[-] (0,0) .. controls +(1,-3) and +(1,-2) .. (3.5,0);
	\end{scope}
	\end{tikzpicture}
	\caption{The paths around which are taken the loops $\gamma_0,\gamma_t,\gamma_1$ defining the monodromy representation $\rho$ 
	and therefore the coordinate $X$ on $\Cal S_{VI}(\theta)$ in dependence on $t\in \CP^1\sminus\{0,1,\infty\}$, and the corresponding transition maps $g_{0t}^{\circ2},\ g_{t1}^{\circ2}:\Cal S_{VI}(\theta)\to \Cal S_{VI}(\theta)$.}
	\label{figure:wm-glueing}
\end{figure}

\section{The confluence $P_{VI}\to P_{V}$ and the character varieties}\label{section:wm-2}

 We will study the degeneration of the $\SL_2(\C)$-isomonodromic problem \eqref{eq:wm-A} of $P_{VI}$ to the one of $P_V$ \cite{JM} using the description of confluence in linear systems by Hurtubise, Lambert \& Rousseau \cite{LR,HLR} in order to understand the degeneration of the character variety of $P_{VI}$ to the wild character variety of $P_V$. 
 The goal is to be able to consider the sequential limits \eqref{eq:wm-epsilon_n} in the parameter $\epsilon$  
 of the braid group actions on the character varieties of $P_{VI}$, 
 in order to apply the results of Section~\ref{section:wm-1a2} on their accumulation to the generators of the wild monodromy pseudogroup of $P_V$, which will be done in Section~\ref{section:wm-2wild}.

The wild character variety of $P_V$ in the form of a generalized Fricke formula was constructed by van der Put \& Saito \cite{PS}, as well as by Chekhov, Mazzocco \& Rubtsov \cite{CMR15} who also describe the confluence but from a very different point of view. 
The idea of sequential limits (discretization) in the parameter $\epsilon$ from $P_{VI}$ to $P_V$ was previously exploited by Kitaev \cite{Kit} in relation to the asymptotics of the corresponding Riemann--Hilbert problem and the tau function. But our approach is different.

\subsection{Confluence of isomonodromic systems}

The substitution
\begin{equation}\label{eq:wm-tt}
t=1+\epsilon\tilde t,\quad 
\vartheta_t= \tfrac{1}{\epsilon},
\quad \vartheta_1= -\tfrac{1}{\epsilon}+\tilde\vartheta_1, 
\end{equation}  
in the system \eqref{eq:wm-A} with
\begin{equation}\label{eq:wm-tildev1}
\begin{aligned}
v_t&=\tfrac{\tilde v_1}{\epsilon\tilde t},& v_1&=-\tfrac{\tilde v_1}{\epsilon\tilde t}-v_0+\kappa_2,\\
u_t&=\tilde u_1, & u_1&=\tilde u_1+\epsilon\tilde t\tfrac{u_0v_0-\tilde u_1(v_0-\kappa_2)}{\tilde v_1+\epsilon\tilde t(v_0-\kappa_2)},
\end{aligned}
\quad\text{where}\quad \kappa_2=-\tfrac{\vartheta_0+\tilde\vartheta_1+\vartheta_\infty}{2},		
\end{equation}
gives a parametric family (depending on the parameter $\epsilon$) of isomonodromic deformations:
\begin{equation}\label{eq:wm-tildeA}
\frac{d\phi}{d z}=\Big[\frac{\tilde A_0(\tilde t)}{z}+
\frac{\tilde A_1^{(0)}(\tilde t)+(z-\!1\!-\!\epsilon\tilde t)\tilde A_1^{(1)}(\tilde t)}
{(z-\!1)(z-\!1\!-\!\epsilon\tilde t)}\Big]\phi,
\end{equation}
where
\begin{align*}
\tilde A_0&=A_0,\\
\tilde A_1^{(0)}&=\epsilon\tilde tA_t=\mbox{\scriptsize$\begin{pmatrix} \tilde v_1+\tfrac{\tilde t}{2}\!\!\! &  -\tilde u_1\tilde v_1 \\[3pt] 
\tfrac{\tilde v_1+t}{\tilde u_1} & \! -\tilde v_1-\tfrac{\tilde t}{2}\end{pmatrix}$},\\
%\ \text{and}\quad 
\tilde A_1^{(1)}&=A_t+A_1=-A_0-A_\infty=\mbox{\scriptsize$\begin{pmatrix} -v_0-\tfrac{\vartheta_0+\vartheta_\infty}{2} \!\!\! &  u_0v_0\\[3pt]  -\tfrac{v_0+\vartheta_0}{u_0} &  \!\!\! v_0+\tfrac{\vartheta_0+\vartheta_\infty}{2}\end{pmatrix}$},
\end{align*}
which then have well defined limits when $\epsilon\to 0$.
The matrix $\tilde A_1^{(0)}$ has eigenvalues  $\pm\frac{\tilde t}{2}$,
therefore the matrix function $\frac{\tilde A_1^{(0)}(\tilde t)+(z-\!1\!-\!\epsilon\tilde t)\tilde A_1^{(1)}(\tilde t)}
{(z-\!1)(z-\!1\!-\!\epsilon\tilde t)}$ 
can be diagonnalized on a uniform neighborhood of $z=1$ for $\epsilon$ small (for fixed $\tilde t\neq 0$),
with eigenvalues
\begin{equation}\label{eq:wm-formalinvariants}
\pm\tfrac12\left(\tfrac{\vartheta_1}{z-1}+\tfrac{\vartheta_t}{z-t}\right)+O(1)=\pm\tfrac{\tilde t+\tilde\vartheta_1(z-1-\epsilon\tilde t)}{2(z-1)(z-1-\epsilon\tilde t)}+O(1),
\end{equation}
which form the set of local formal invariants for the confluent pair of singularities $\{1,1+\epsilon\tilde t\}$.

The connection \eqref{eq:wm-connection} becomes
\[d-\tilde A(z,t)dz+\tfrac{\tilde A_1^{(0)}(t)}{\tilde t(z-1-\epsilon\tilde t)}d\tilde t,\]
the flatness of which is the isomonodromicity condition on \eqref{eq:wm-tildeA}. 
The variables $(q,p)$ of the confluent Painlev\'e system \eqref{eq:wm-PVIepsilon} are defined as before by
\[q=-(1+\epsilon\tilde t)\frac{[\tilde A_0]_{12}}{[\tilde A_1^0]_{12}+[\tilde A_1^1]_{12}},\qquad 
p=[\tilde A(q,t)]_{11}+\tfrac{\vartheta_0}{2q} +\tfrac{\tilde t}{2(q-1)(q-1-\epsilon\tilde t)}+\tfrac{\tilde\vartheta_1}{2(q-1)}.\]

\subsection{Confluence on character varieties}
The confluence of singularities in linear systems has been studied by many authors, including Garnier \cite{Gar}, Ramis \cite{Ra1}, Sch\"{a}fke \cite{Sch2}, Duval \cite{Duv}, Glutsyuk \cite{Gl,Gl1}, Zhang \cite{Zh}, etc..
Here we will use a description following from a theorem of sectorial normalization for unfolding of non-resonant irregular singularities due to Hurtubise, Lambert and Rousseau \cite{LR, HLR}, and in the case of Poincaré rank 1 also due to Parise \cite{Pa}.

The local analytic invariants of the limit system \eqref{eq:wm-tildeA} with $\epsilon=0$ are usually expressed in terms of a pair of Stokes matrices.
The work \cite{LR, HLR} shows that one can define ``unfolded Stokes matrices'' also for small $\epsilon\neq 0$.  
More precisely, for $\epsilon$ in each of the sectors $\sE_+,\sE_-$ \eqref{eq:wm-sE} (the same sectors as those for normalization of the confluent family of nonlinear Painlev\'e equations!), the system \eqref{eq:wm-tildeA} has certain privileged fundamental solution matrix with respect to which the two monodromies $M_t,M_1$ are one upper triangular and the other lower triangular.
This allows to decompose each of the matrices $M_t,M_1$ as a product of a diagonal ``formal monodromy matrix'' (with divergent limit $\epsilon\to 0
$) and of a unipotent ``unfolded Stokes matrix'' (with convergent limit $\epsilon\to 0
$).

In fact the two columns of this privileged fundamental solution matrix form a \emph{mixed basis of solutions} (cf. Remark~\ref{remark:wm-mixedbasis}):
the first column, which is an ``eigen-solution'' for the eigenvalue $e_t$ of $M_t$,
is characterized as a \emph{subdominant solution} which vanishes when $z$ approaches $t=1+\epsilon\tilde t$ along some suitable curve, while
the second column, which is an ``eigen-solution'' for the eigenvalue $e_1$ of $M_1$, is a \emph{subdominant solution} which vanishes when $z$ approaches $1$ along some suitable curve.
Here the suitable curves of approach of the singular points $1+\epsilon\tilde t$ and $1$ are the real time trajectories of the vector field
\[\frac{dz}{d\tilde t}=e^{i\omega_\pm}\tfrac{(z-1)(z-1-\epsilon\tilde t)}{\tilde t+\tilde\vartheta_1(z-1-\epsilon\tilde t)},\quad \tilde t\in\R,\]
or equivalently, leaves of the \emph{horizontal foliation} of the meromorphic differential form
$e^{-i\omega_\pm}\tfrac{\tilde t+\tilde\vartheta_1(z-1-\epsilon\tilde t)}{(z-1)(z-1-\epsilon\tilde t)}dz$ which expresses the formal invariants \eqref{eq:wm-formalinvariants},
where $\omega_\pm$ is as in \eqref{eq:wm-omegapm} (in particular, for $\epsilon\in\{|\arg(\pm\epsilon)|<\frac{\pi}{2}-\eta\}\subset\sE_\pm$ one can take $\omega_\pm=0$).
When appropriately normalized, such mixed solution basis has a well-defined limit when $\epsilon\to 0$, $\epsilon\in\sE_\pm$, which is the canonical sectorial basis of solutions at the limit irregular singularity. 
%Indeed, the canonical sectorial bases of solutions at the irregular singularity of \eqref{eq:wm-tildeA} on each sector consist exactly of the pair of subdominant solutions when $\tilde t$ approaches the singularity respectively along the left or the right, boundary ray of the sector (this is true for any irregular singularity of a $2\!\times\!2$ linear differential system in general, for more information on irregular singularities see e.g. \cite{BaVa, MR2, Ba}). 

In order to simplify the description we will consider only the confluence in the sector $\epsilon\in\sE_+$.
Then for $0\neq\epsilon$ we have a monodromy representation 
\[\rho_+:\pi_1\left(\C\sminus\{0,1+\epsilon\tilde t,1\},z_0\right)\to\SL_2(\C)\] 
with respect to this privileged mixed basis of solutions which is of the form: 
\begin{equation}\label{eq:wm-confluentM}
\begin{aligned}
M_{0+}&=\left(\begin{smallmatrix} \alpha & \beta \\ \gamma & \delta \end{smallmatrix}\right), \\
M_{t+}&=N_tS_{2+}=\left(\begin{smallmatrix} e_t & e_ts_{2+} \\ 0 & \frac{1}{e_t} \end{smallmatrix}\right),\\
%&M_{t-}&=S_{1-}N_t=\left(\begin{smallmatrix} e_t & 0 \\ e_ts_{1-} & \frac{1}{e_t} \end{smallmatrix}\right),\\
M_{1+}&=S_{1+}N_1=\left(\begin{smallmatrix} e_1 & 0 \\ e_1s_{1+} & \frac{1}{e_1} \end{smallmatrix}\right),\\
%&M_{1-}&=N_1S_{2-}=\left(\begin{smallmatrix} e_1 &  e_1s_{2-} \\ 0 & \frac{1}{e_1} \end{smallmatrix}\right),\\
M_{\infty+}&=(S_{1+}N_1N_tS_{2+}M_{0+})^{-1}=
\left(\begin{smallmatrix} e_te_1(\beta s_{1+}+\delta s_{1+}s_{2+})+\frac{\delta_+}{e_te_1}, & -e_te_1(\beta+\delta s_{2+}) \\ 
	-e_te_1(\alpha s_{1+}+\gamma s_{1+}s_{2+})-\frac{\gamma}{e_te_1}, & e_te_1(\alpha+\gamma s_{2+}) \end{smallmatrix}\right),
\end{aligned}
\end{equation}
see Figure~\ref{figure:wm-confluentrho}, where 
\[S_{1+}(\epsilon)=\left(\begin{smallmatrix} 1 & 0 \\[3pt] s_{1+}(\epsilon) & 1 \end{smallmatrix}\right), \ S_{2+}(\epsilon)=\left(\begin{smallmatrix} 1 & s_{2+}(\epsilon) \\[3pt] 0 & 1 \end{smallmatrix}\right),\]
are \emph{unfolded Stokes matrices} of the family \eqref{eq:wm-tildeA}, which tend to the \emph{Stokes matrices} of the limit system $\lim_{\sE_+\ni\epsilon\to 0}S_{i+}(\epsilon)=S_{i}(0)$, $i=1,2$,
and where
\[N_i(\epsilon)=\left(\begin{smallmatrix} e_i & 0 \\ 0 & \frac{1}{e_i} \end{smallmatrix}\right),\ \ i=t,1,\qquad
N=N_tN_1=\left(\begin{smallmatrix} e_te_1 & 0 \\ 0 & \frac{1}{e_te_1} \end{smallmatrix}\right),\]
with 
\[e_t=(M_{t+})_{11}=e^{\frac{\pi i}{\epsilon}},\quad e_1=(M_{1+})_{11}=e^{\pi i\tilde\vartheta_1-\frac{\pi i}{\epsilon}},\]
are the \emph{formal monodromy matrices} around the points $t=1+\epsilon t$ and $1$ for $\epsilon\neq 0$.
\emph{Such monodromy representation is determined uniquely up to conjugation by diagonal matrices.}

\begin{figure}[t]
\centering	
\includegraphics[width=.6\textwidth]{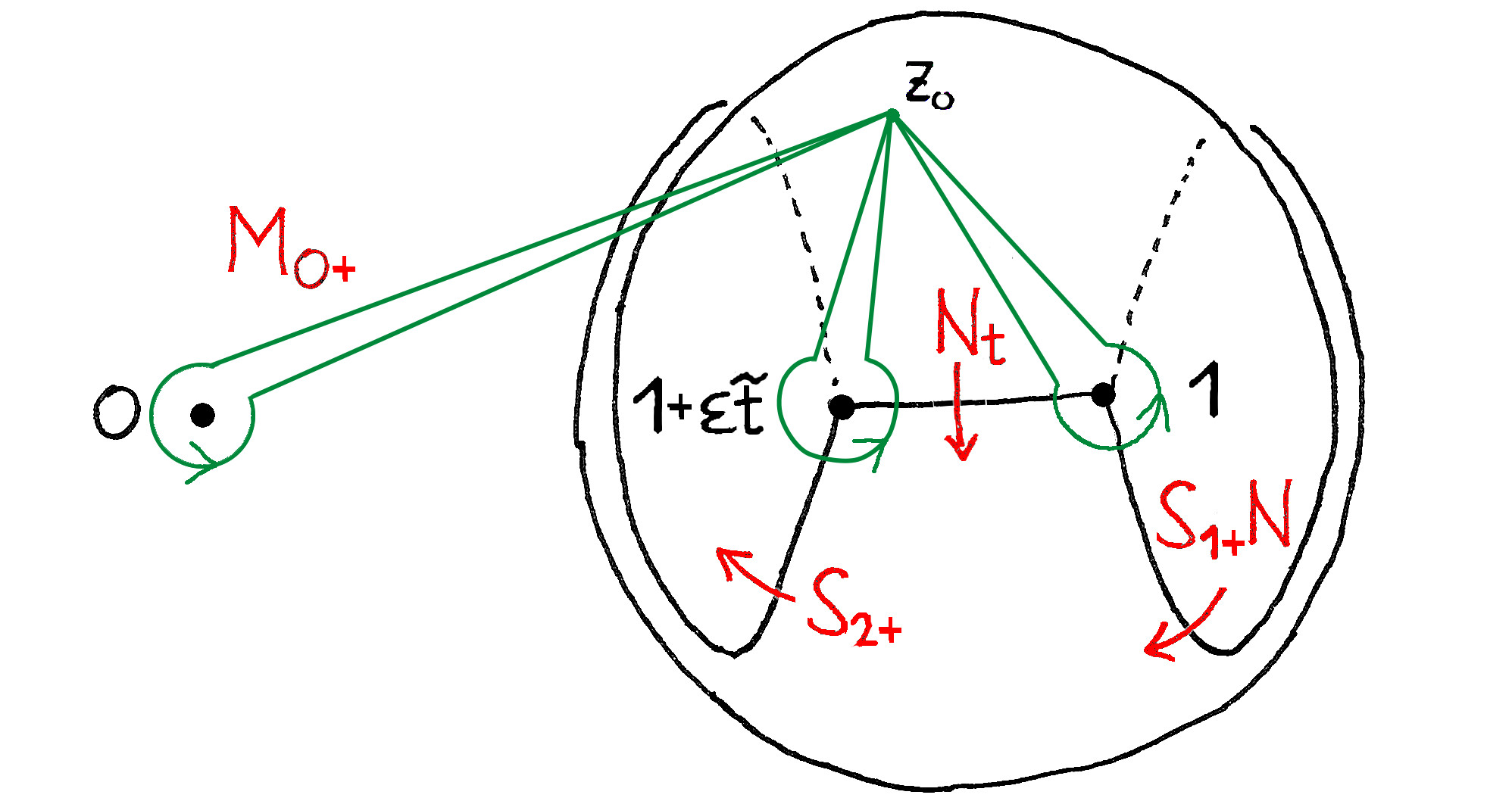}
\caption{Monodromy of the system \eqref{eq:wm-tildeA} when $\epsilon\to 0+$.}
\label{figure:wm-confluentrho}
\end{figure} 

The monodromy matrices are subject to the conditions
\begin{align*}
1&=\det\,M_{0+}=\alpha\delta-\gamma\beta,\\
a_0&=\tr\,M_{0+}=\alpha+\delta, \\
a_\infty&=\tr\,M_{\infty+}=\tfrac{\delta}{e_te_1}+e_te_1(\alpha+\beta s_{1+}+\gamma s_{2+}+\delta s_{1+}s_{2+}),\\
\tilde e_1:=e^{\pi i\tilde\vartheta_1}&=[M_{1+}M_{t+}]_{11}=e_te_1.
\end{align*}
The trace coordinates $X_i$ \eqref{eq:wm-X} for $0\neq\epsilon\in\sE_+$ are given by
\begin{equation}\label{eq:wm-XX}
\begin{aligned}
X_0&=\tr(M_{1+}M_{t+})=e_te_1+\tfrac{1}{e_te_1}+e_te_1s_{1+}s_{2+},\\
%&X_0&=\tr(M_{1-}M_{t-})=e_te_1+\tfrac{1}{e_te_1}+e_te_1s_{1-}s_{2-},\\ 
X_t&=\tr(M_{1+}M_{0+})=e_1(\alpha +\beta s_{1+})+\tfrac{\delta}{e_1},\\
%&X_t&=\tr(M_{1-}M_{0-})=e_1(\alpha_- +\gamma_- s_{2-})+\tfrac{\delta_-}{e_1},\\
X_1&=\tr(M_{0+}M_{t+})=e_t(\alpha +\gamma s_{2+})+\tfrac{\delta}{e_t},\\
%&X_1&=\tr(M_{0-}M_{t-})=e_t(\alpha_- +\beta_- s_{1-})+\tfrac{\delta_-}{e_t}.
\end{aligned}
\end{equation}
Only the parameters
\[a_0=2\cos(\pi\vartheta_0),\qquad \tilde e_1=e_te_1=e^{\pi i\tilde\vartheta_1},\qquad a_\infty=2\cos(\pi\vartheta_\infty),\]
have well defined limits when $\epsilon\to 0$, 
while $e_t=e^{\frac{\pi i}{\epsilon}}$ and $e_1=e^{\pi i\tilde\vartheta_1-\frac{\pi i}{\epsilon}}$ diverge.
Therefore the coordinate $X_0$ passes well to the limit, but not $X_t,X_1$ which need be replaced by new coordinates.
The new coordinates should be functions of $M_{0+},S_{1+},S_{2+},N,M_{\infty+}$ invariant by diagonal conjugation
(since the decomposition \eqref{eq:wm-confluentM} for $\epsilon\neq 0$, as well as the pair of Stokes matrices $S_{1+},S_{2+}$ for $\epsilon=0$, are determined only up to conjugation by diagonal matrices).
Following \cite{PS}, we choose them as the lower diagonal elements of $M_{\infty+}$ and $M_{0+}$:
\begin{equation}\label{eq:wm-UW}
\begin{aligned}
%\pX_{0\pm}&=X_0,&
\pXone&=\tr (M_{1+}M_{t+0})=X_0,\\
%\pX_{t\pm}&=(M_{\infty\pm})_{22}=e_te_1(\alpha_\pm +\beta_\pm s_{1\pm}),&  
\pXzero&=[M_{0+}]_{22}=\delta,\\
%\pX_{1\pm}&=\delta_\pm,&
\pXinf&=[M_{\infty+}]_{22}=e_te_1(\alpha +\gamma s_{2+}).
\end{aligned}
\end{equation}

A substitution in the identity 
\[(e_te_1)^2s_{1+}s_{2+}(\alpha\delta-\beta\gamma-1)=0,\]
gives the Fricke relation in the new coordinates $\pX=(\pXzero,\pXone,\pXinf)$
\[\pF(\pX,\ptheta(\tilde a))=0,\]
where
\begin{equation}\label{eq:wm-tildeF}
\pF(\pX,\ptheta):=\pXzero\pXone\pXinf+\pXzero^2+\pXinf^2-\pthetazero\pXzero-\pthetaone\pXone-\pthetainf\pXinf+\pthetat=0
\end{equation}
where $\ptheta=(\pthetazero,\pthetaone,\pthetainf,\pthetat)$ is a function of the parameter 
\[\tilde{a}=(a_0,\tilde e_1,a_\infty),\] 
independent of $\epsilon$,
\begin{equation}\label{eq:wm-tildetheta}
\pthetazero=a_0+\tilde e_1a_\infty,\quad
\pthetaone=\tilde e_1,\quad
\pthetainf=a_\infty+\tilde e_1a_0,\quad
\pthetat=1+\tilde e_1a_0 a_\infty+\tilde e_1^2.
\end{equation}
For $\epsilon=0$, the relation \eqref{eq:wm-tildeF} for the character variety of $P_V$ was derived in \cite[section 3.2]{PS}.

\begin{definition}
The \emph{wild character variety of $P_V$} is the complex cubic surface 
\begin{equation}
 \Cal S_{V}(\ptheta)=\{\pX\in\C^3: 
 \pF(\pX,\ptheta)=0\},
\end{equation}
where $\pF$ is \eqref{eq:wm-tildeF}, endowed with the algebraic symplectic form 
\begin{equation}\label{eq:wm-tildeomega}
\tilde\omega_{\Cal S_{V}}=\frac{d\pXzero\wedge d\pXone}{2\pi i\pF_{\!\pXinf}}=\frac{d\pXinf\wedge d\pXzero}{2\pi i\pF_{\!\pXone}}=\frac{d\pXone\wedge d\pXinf}{2\pi i\pF_{\!\pXzero}}.
\end{equation}
\end{definition}

Denote
\begin{align*}
\pF_{\!\pXone}&:=\tfrac{\partial\pF}{\partial \pXone}=\pXinf\pXzero-\pthetaone, \\
\pF_{\!\pXzero}&:=\tfrac{\partial\pF}{\partial \pXzero}=\pXone\pXinf+2\pXzero-\pthetazero, \\
\pF_{\!\pXinf}&:=\tfrac{\partial\pF}{\partial \pXinf}=\pXone\pXzero+2\pXinf-\pthetainf.
\end{align*}

The trace coordinates $X$ \eqref{eq:wm-XX} on the character variety
$\Cal S_{VI}(\theta)$ for $\epsilon\neq0$ and the new coordinates $\pX$ \eqref{eq:wm-UW} on the wild character variety 
$\Cal S_{V}(\ptheta)$ are related by the following birational transformations:

\begin{theorem}[$\epsilon\neq 0$]\label{theorem:wm-Phi}
The cubic surfaces $\Cal S_{VI}(\theta)$ and $\Cal S_{V}(\ptheta)$ are birationally equivalent through the change of variables
\[X=\Phi_+(\pX),\]% \quad\text{and}\quad \tilde Y_\pm=\tilde g_{t1}(\pX_\pm),
where
\begin{align*}
\Phi_+(\cdot,a_0,e_t,e_1,a_\infty): \Cal S_{V}(\ptheta)&\to \Cal S_{VI}(\theta)\\
\pX &\mapsto X 
\end{align*}
is given by
\begin{equation}\label{eq:wm-XXXUW}
\begin{aligned}
\Phi_+:\
X_0&=\pXone,& 
F_{\!X_0}\circ\Phi_+&= \tfrac{\pXzero}{e_t}\big(\tfrac{\pF_{\!\pXzero}}{e_1}-\tfrac{\pF_{\!\pXinf}}{e_t}\big)-\tfrac{\pF_{\!\pXone}}{e_te_1}(\pXone-\tfrac{e_t}{e_1}-\tfrac{e_1}{e_t}) 
\\
X_t&=\tfrac{\pXzero}{e_1}+\tfrac{\pXinf-\pF_{\!\pXinf}}{e_t}, &
F_{\!X_t}\circ\Phi_+&= \tfrac{\pF_{\!\pXzero}}{e_1}-\tfrac{\pF_{\!\pXinf}}{e_t}, 
\\
X_1&=\tfrac{\pXinf}{e_1}+\tfrac{\pXzero}{e_t}, &
F_{\!X_1}\circ\Phi_+&= \tfrac{\pF_{\!\pXzero}}{e_t}+\tfrac{\pF_{\!\pXinf}}{e_1} -\tfrac{\pXone\pF_{\!\pXinf}}{e_t}.
\end{aligned}
\end{equation}
The inverse map 
$\Phi_+^{\circ(-1)}: \Cal S_{VI}(\theta) \dashrightarrow \Cal S_{V}(\ptheta)$, $X\mapsto\pX_+$, 
is given by
\begin{equation}\label{eq:wm-UWXXX}
\begin{aligned}
\hskip-24pt\Phi_+^{\circ(-1)}\!:\ \pXone&=X_0,
\\
\pXzero&=-\frac{e_tX_t+e_1X_1-a_\infty-e_te_1a_0}{X_0-\tfrac{e_t}{e_1}-\tfrac{e_1}{e_t}}
=-\frac{e_te_1\big(X_0-\tfrac{e_t}{e_1}-\tfrac{e_1}{e_t}-F_{\!X_0}\big)}{e_1X_t+e_tX_1-a_0-e_te_1a_\infty},\hskip-24pt\\
\pXone&=X_0,\\
\pXinf&=
-\frac{e_1X_t+e_tX_1-a_0-e_te_1a_\infty-e_1F_{\!X_t}}{X_0-\tfrac{e_t}{e_1}-\tfrac{e_1}{e_t}},
%= {e_1}X_1-\tfrac{e_1}{e_t}\pXzero,
%=-e_te_1\frac{(e_1X_t+e_tX_1-\pthetainf-e_tF_{\!X_1})(X_0-\tfrac{e_t}{e_1}-\tfrac{e_1}{e_t}-F_{\!X_0})}
%{(e_1X_t+e_tX_1-\pthetainf)(e_tX_t+e_1X_1-\pthetazero)},
\end{aligned}
\end{equation}
(cf. Proposition \ref{prop:wm-lines}), and is singular on the line:
\begin{equation}\label{eq:wm-lineL0}
L_0:=\{X_0=\tfrac{e_t}{e_1}+\tfrac{e_1}{e_t},\ \ e_1X_t+e_tX_1=\pthetazero\}.
\end{equation}
The two Fricke relations are related by
\[F\circ\Phi_+=\tfrac{-1}{e_te_1}(X_0-\tfrac{e_t}{e_1}-\tfrac{e_1}{e_t})\cdot \pF.\]
The restriction 
\[\Phi_+: \Cal S_{V}(\ptheta)\to \Cal S_{VI}(\theta)\sminus L_0\]
is an isomorphism.
The pull-back of the symplectic form $\omega_{\Cal S_{VI}}$ \eqref{eq:wm-Omega} by $\Phi_+$ is the symplectic form \eqref{eq:wm-tildeomega}.
\end{theorem}

\begin{proof} %[Proof of Proposition \ref{prop:wm-omega}.]
From \eqref{eq:wm-XX} and \eqref{eq:wm-UW} by a direct calculation. 
\end{proof}

\begin{remark}[$\epsilon\neq 0$]
	In the trace coordinates \eqref{eq:wm-a} on the space of monodromy representations, the eigenvalues $e_i$ and $\tfrac{1}{e_i}$ of $M_i$ are interchangeable.
	On the other hand in the above confluent description of the monodromy data %\eqref{eq:wm-confluentM} 
	this is no longer true for $i=t,1$.
	Indeed, during the confluence we are fixing an appropriately normalized mixed basis of solutions 
	consisting of one ``eigen-solution'' for the eigenvalue $e_t$ of $M_t$ and another for $e_1$ of $M_1$.
	The general theory of confluence \cite{HLR} tells us that this mixed basis tends, when $\epsilon\to 0$ in the sector $\sE_+$, to the appropriate sectorial basis 
	at the irregular singularity of the limit system. This means that, for small enough $\epsilon$, the chosen mixed basis can never be degenerate.
	On the other hand, the singular line $L_0$ \eqref{eq:wm-lineL0} for $\epsilon\neq 0$ corresponds by Remark~\ref{remark:wm-mixedbasis} precisely to the monodromy representations for which our mixed basis degenerates, so there is no place for them in our confluent picture.
\end{remark}

\begin{remark}\label{remark:wm-I}
A very simple way to obtain the coordinates $\pX$ on the wild character variety $\Cal S_{V}(\ptheta)$ is by taking the following limit:
\begin{enumerate}[label=(\roman*), wide=0pt, leftmargin=\parindent]
\item When $\epsilon\to 0$ in a sector $\eta<\arg\epsilon<\pi-\eta$, $\eta>0$, then $e_t\to\infty$, $e_1\to 0$, hence $\frac{a_t}{a_1}\to \tilde e_1$,
	\begin{equation*}
	\Big(\frac{X_t}{a_1},X_0,	\frac{X_1}{a_1} \Big)\to \big(\pXzero,\pXone,\pXinf\big),\qquad
	\Big(\frac{\theta_t}{a_1},\frac{\theta_0}{a_1^2},\frac{\theta_1}{a_1},\frac{\theta_\infty}{a_1^2}\Big)\to 
	\big(\pthetazero,\pthetaone,\pthetainf,\pthetat\big),
	\end{equation*}
	and $\frac{1}{a_1^2}F(X,\theta)\to \pF(\pX,\ptheta)$.
\item When $\epsilon\to 0$ in a sector $-\pi+\eta<\arg\epsilon<-\eta$, $\eta>0$, then $e_t\to 0$, $e_1\to \infty$, hence $\frac{a_1}{a_t}\to \tilde e_1$,
	\begin{equation*}
	\Big(\frac{X_1}{a_t}, X_0, \frac{X_t-F_{\!X_t}}{a_t} \Big)\to \big(\pXzero,\pXone,\pXinf\big),\qquad
	\Big(\frac{\theta_1}{a_t}, \frac{\theta_0}{a_t^2}, \frac{\theta_t}{a_t}, \frac{\theta_\infty}{a_t^2}\Big)\to 
	\big(\pthetazero,\pthetaone,\pthetainf,\pthetat\big),
\end{equation*}
	$\frac{X_t}{a_t}\to \pXinf-\pF_{\!\pXinf}$, and $\frac{1}{a_t^2}F(X,\theta)\to \pF(\pX,\ptheta)$.
\end{enumerate}
These two limits are related by the action of the half-monodromy operator $g_{1t}=g_{t1}^{\circ(-1)}$ \eqref{eq:wm-ginfty0} on the left side (and identity on the right side).
This limiting procedure should be related to Glutsyuk's description of confluence \cite{Gl,Gl1} and has been well observed before (for example \cite{CMR15}).
Whenever it leads to a convergent limit it can be used as a simple way to obtain limit formulas. However 
it doesn't allow to treat the confluence when $\epsilon\to 0$ along the sequences $\{\epsilon_n\}_{n\in\N}$ with $\frac{1}{\epsilon_n}=\frac{1}{\epsilon_0}+n$,
that accumulate to $0$ tangent to $\R_{>0}$,
which is essential for us in order to obtain the generators of the wild monodromy.
\end{remark}

\subsection{The total monodromy action on $\Cal S_V(\ptheta)$}

Both of the two different half-monodromy operators $g_{t1}$ and $g_{\infty 0}$ \eqref{eq:wm-ginfty0}, satisfying $g_{t1}^{\circ 2}=g_{\infty 0}^{\circ 2}$, and acting on the space of monodromy representations
associated to $P_{VI}$, pass well to the confluent limit if properly interpreted.
However, the action $g_{\infty 0}$ (which differs from $g_{t1}$ in how it acts on $(a_0,a_t,a_1,a_\infty)$) seems to be the one that is better suited:
this has to do with the fact that its action is on the points $z=0,\infty$, thus away from where the confluence happens.

\begin{proposition}\label{prop:wm-gt1}
\begin{enumerate}[label=(\roman*), wide=0pt, leftmargin=\parindent]
\item	
The pullback of the half-monodromy operators $g_{ij}:\Cal S_{VI}(\theta)\to\Cal S_{VI}(g_{ij}(\theta))$, $(i,j)=(t,1),(\infty,0)$, and their inverses $g_{ji}=g_{ij}^{\circ(-1)}$, by the transformation $\Phi_+$ \eqref{eq:wm-UWXXX} is given by 
\[\tilde g_{ij}=\Phi_+^{\circ(-1)}\circ g_{ij}\circ\Phi_+:\Cal S_{V}(\ptheta)\to \Cal S_{V}(\tilde g_{ij}(\ptheta))\]
\begin{align*}
\tilde g_{t1}:\ 
a_0&\mapsto a_0,& 	
\pthetazero&\mapsto \tfrac{\pthetainf}{\pthetaone},&
\pXzero&\mapsto \tfrac{1}{\pthetaone}\big(\pXinf-\pF_{\!\pXinf}\big), & 
\scriptstyle\pF_{\!\pXzero}&\begin{scriptstyle}\mapsto -\tfrac{1}{\pthetaone}\pF_{\!\pXinf}\end{scriptstyle},\\
\tilde e_1 & \mapsto \tfrac{1}{\tilde e_1},  & 
\pthetaone&\mapsto \tfrac{1}{\pthetaone},&
\pXone&\mapsto \pXone,  &  
\scriptstyle\pF_{\!\pXone}&\begin{scriptstyle}\mapsto \pF_{\!\pXone}-\pF_{\!\pXinf}\pXzero\end{scriptstyle},\\
a_\infty&\mapsto a_\infty, & 
\pthetainf&\mapsto \tfrac{\pthetazero}{\pthetaone},&
\pXinf&\mapsto \tfrac{1}{\pthetaone}\pXzero,	& 
\scriptstyle\pF_{\!\pXinf}&\begin{scriptstyle}\mapsto \tfrac{1}{\pthetaone}\big(\pF_{\!\pXzero}-\pF_{\!\pXinf}\pXone\big)\end{scriptstyle},\\
&&
\pthetat&\mapsto \tfrac{\pthetat}{\pthetaone^2}.&
&&&
\end{align*}

\begin{align*}
\tilde g_{1t}:\ 
a_0&\mapsto a_0, & 
\pthetazero&\mapsto \tfrac{\pthetainf}{\pthetaone},&
\pXzero&\mapsto \tfrac{1}{\pthetaone}\pXinf,	& 
\scriptstyle\pF_{\!\pXzero}&\begin{scriptstyle}\mapsto \tfrac{1}{\pthetaone}\big(\pF_{\!\pXinf}-\pF_{\!\pXzero}\pXone\big)\end{scriptstyle},\\
\tilde e_1 & \mapsto \tfrac{1}{\tilde e_1}, & 
\pthetaone&\mapsto \tfrac{1}{\pthetaone},&
\pXone&\mapsto \pXone,  &  
\scriptstyle\pF_{\!\pXone}&\begin{scriptstyle}\mapsto \pF_{\!\pXone}-\pF_{\!\pXzero}\pXinf\end{scriptstyle},\\
a_\infty&\mapsto a_\infty, & 	
\pthetainf&\mapsto \tfrac{\pthetazero}{\pthetaone},&
\pXinf&\mapsto \tfrac{1}{\pthetaone}\big(\pXzero-\pF_{\!\pXzero}\big), & 
\scriptstyle\pF_{\!\pXinf}&\begin{scriptstyle}\mapsto -\tfrac{1}{\pthetaone}\pF_{\!\pXzero}\end{scriptstyle}\\
&&
\pthetat&\mapsto \tfrac{\pthetat}{\pthetaone^2}.&
&&&
\end{align*}

\begin{align*}
\tilde{g}_{\infty 0}:\ 
a_0&\mapsto a_\infty, &
\pthetazero&\mapsto \pthetainf,&
\pXzero&\mapsto \pXinf-\pF_{\!\pXinf}, & 
\scriptstyle\pF_{\!\pXzero}&\begin{scriptstyle}\mapsto -\pF_{\!\pXinf}\end{scriptstyle},\\
\tilde e_1 & \mapsto \tilde e_1,&	
\pthetaone&\mapsto \pthetaone,&
\pXone&\mapsto \pXone,  & 
\scriptstyle\pF_{\!\pXone}&\begin{scriptstyle}\mapsto \pF_{\!\pXone}-\pF_{\!\pXinf}\pXzero\end{scriptstyle},\\
a_\infty&\mapsto a_0,  & 
\pthetainf&\mapsto \pthetazero,&
\pXinf&\mapsto \pXzero,	& 
\scriptstyle\pF_{\!\pXinf}&\begin{scriptstyle}\mapsto \pF_{\!\pXzero}-\pF_{\!\pXinf}\pXone\end{scriptstyle}\\
&&
\pthetat&\mapsto \pthetat.&
&&&
\end{align*}

\begin{align*}
\tilde{g}_{0\infty}:\ 
a_0&\mapsto a_\infty,  & 
\pthetazero&\mapsto \pthetainf,&
\pXzero&\mapsto \pXinf,	& 
\scriptstyle\pF_{\!\pXzero}&\begin{scriptstyle}\mapsto \pF_{\!\pXinf}-\pF_{\!\pXzero}\pXone\end{scriptstyle},\\
\tilde e_1 & \mapsto \tilde e_1,&	
\pthetaone&\mapsto \pthetaone,&
\pXone&\mapsto \pXone,  & 
\scriptstyle\pF_{\!\pXone}&\begin{scriptstyle}\mapsto \pF_{\!\pXone}-\pF_{\!\pXzero}\pXinf\end{scriptstyle},\\
a_\infty&\mapsto a_0, & 		
\pthetainf&\mapsto \pthetazero,&
\pXinf&\mapsto \pXzero-\pF_{\!\pXzero}, & 
\scriptstyle\pF_{\!\pXinf}&\begin{scriptstyle}\mapsto -\pF_{\!\pXzero}\end{scriptstyle}\\
&&
\pthetat&\mapsto \pthetat.&
&&&
\end{align*}

\item
The pullback $\tilde g_{t1}^{\circ 2}=\tilde g_{\infty0}^{\circ 2}\in\Aut_{\omega_{\Cal S_{V}}}(\Cal S_{V}(\ptheta))$ of the monodromy operator $g_{t1}^{\circ 2}=g_{\infty0}^{\circ 2}\in\Aut_{\omega_{\Cal S_{VI}}}(\Cal S_{VI}(\theta))$ by the transformation $\Phi_+$ \eqref{eq:wm-UWXXX}
is given by the square iterate of the above operators $\tilde g_{t1}$, $\tilde g_{\infty0}$
\begin{equation*}
\begin{aligned}
\tilde g_{\infty0}^{\circ 2}:\ 
\pXzero&\mapsto \pXzero-\pF_{\!\pXzero}+\pXone\pF_{\!\pXinf}, &
	\scriptstyle\pF_{\!\pXzero}&\begin{scriptstyle}\mapsto-\pF_{\!\pXzero}+\pXone\pF_{\!\pXinf} \end{scriptstyle},\\[0pt]
\pXone&\mapsto \pXone, &
	\scriptstyle\pF_{\!\pXone}&\begin{scriptstyle}\mapsto\pF_{\!\pXone}-\pXzero\pF_{\!\pXinf}-\pXinf\pF_{\!\pXzero}+\pF_{\!\pXzero}\pF_{\!\pXinf}+\pXone\pXinf\pF_{\!\pXinf}-\pXone\pF_{\!\pXinf}^2\end{scriptstyle},\\[0pt]
\pXinf&\mapsto \pXinf-\pF_{\!\pXinf},&
	\scriptstyle\pF_{\!\pXinf}&\begin{scriptstyle}\mapsto -\pF_{\!\pXinf}-\pXone\pF_{\!\pXzero}+\pXone^2\pF_{\!\pXinf}\end{scriptstyle},
	%\\[8pt]
	\end{aligned}
\end{equation*}
and its inverse
\begin{equation*}
\begin{aligned}
\tilde g_{0\infty}^{\circ 2}:\  
\pXzero&\mapsto \pXzero-\pF_{\!\pXzero}, &
	\scriptstyle\pF_{\!\pXzero}&\begin{scriptstyle}\mapsto-\pF_{\!\pXzero}-\pXone\pF_{\!\pXinf}+\pXone^2\pF_{\!\pXzero}\end{scriptstyle},\\[0pt]
\pXone&\mapsto \pXone, &
	\scriptstyle\pF_{\!\pXone}&\begin{scriptstyle}\mapsto\pF_{\!\pXone}-\pXzero\pF_{\!\pXinf}-\pXinf\pF_{\!\pXzero}+\pF_{\!\pXzero}\pF_{\!\pXinf}+\pXone\pXzero\pF_{\!\pXzero}-\pXone\pF_{\!\pXzero}^2\end{scriptstyle},\\[0pt]
\pXinf&\mapsto \pXinf-\pF_{\!\pXinf}+\pXone\pF_{\!\pXzero},&\ 
	\scriptstyle\pF_{\!\pXinf}&\begin{scriptstyle}\mapsto -\pF_{\!\pXinf}+\pXone\pF_{\!\pXzero}\end{scriptstyle}.	
	\end{aligned}
\end{equation*}
They preserve the Fricke relation: $\pF\circ \tilde g_{\infty0}^{\circ 2}=\pF.$
\end{enumerate}
\end{proposition}

The ``half-monodromy'' action $\tilde g_{t1}$ was previously described in \cite{PR}, and the monodromy action $g_{t1}^{\circ 2}=g_{\infty0}^{\circ 2}$ was considered in \cite{PS,PR}.

\begin{proof}
One way to obtain the action $\tilde g_{t1}$ is by the limiting procedure of Remark~\ref{remark:wm-I}. For example in (i), for $\Im\epsilon>0$:
\begin{equation*}
\begin{array}{rcl}
\big(\tfrac{X_t}{a_1},X_0,\tfrac{X_1}{a_1}\big) & \xmapsto{g_{t1}} &\big(\tfrac{X_1-F_{\!X_1}}{a_t},X_0,\tfrac{X_t}{a_t}\big) 	\\[6pt]
\Big\downarrow \epsilon\to 0 && \qquad\quad\Big\downarrow\epsilon\to 0 \\[6pt]
\big(\pXzero,\pXone,\pXinf\big) & \xmapsto{\tilde g_{t1}} &\big(\tfrac{\pXinf-\pF_{\!\pXinf}}{\tilde e_1},\pXone,\tfrac{\pXzero}{\tilde e_1}\big), 	\\
\end{array}	
\end{equation*}
and similarly in (ii), for $\Im\epsilon<0$.	

A different way	to obtain the action $\tilde g_{t1}$ is to look on how the braid $\beta_{t1}$ acts on the monodromy representation $\rho_+$ \eqref{eq:wm-confluentM}.
	After a conjugation by a matrix $PM_{1+}^{-1}$ the action \eqref{eq:wm-beta} is written as 
	\begin{equation*}
	\begin{aligned}
	(\beta_{t1})_*:\
	M_{0+} &\mapsto \rho^{\beta_{t1}}(\gamma_0)=PM_{1+}^{-1}M_{0+}M_{1+}P ,\qquad\\
	M_{t+} &\mapsto \rho^{\beta_{t1}}(\gamma_t)=PM_{1+}P,\\
	M_{1+} &\mapsto \rho^{\beta_{t1}}(\gamma_1)=PM_{0+}P,\\
	M_{\infty+} &\mapsto \rho^{\beta_{t1}}(\gamma_0)=PM_{1+}^{-1}M_{\infty+}M_{1+}P ,
	\end{aligned}
	\end{equation*}
	where $P=\left(\begin{smallmatrix} 0 & 1 \\[3pt] 1 & 0 \end{smallmatrix}\right)$, so that $\rho^{\beta_{t1}}(\gamma_t)$ is upper-triangular, and  $\rho^{\beta_{t1}}(\gamma_1)$ is lower-triangular. Therefore $g_{t1}$ maps $(e_t,e_1)\mapsto(\frac{1}{e_1},\frac{1}{e_t})$, and we have 
	\begin{align*}
	\tilde e_1&\mapsto\tfrac{1}{\tilde e_1},\\
	\pXone&\mapsto\pXone,\\
	\pXzero&\mapsto\big(M_{1+}^{-1}M_{0+}M_{1+}\big)_{11}=\alpha+\beta s_{1+}=\tfrac{\pXinf-\pF_{\!\pXinf}}{\tilde e_1},\\
	\pXinf&\mapsto\big(M_{1+}^{-1}M_{\infty+}M_{1+}\big)_{11}=\tfrac{\delta}{e_te_1}=\tfrac{\pXzero}{\tilde e_1}.
	\end{align*}

Likewise, the formula for the action of $\tilde g_{\infty 0}$ can be obtained from  $g_{\infty 0}$ by the the limiting procedure of Remark~\ref{remark:wm-I}. 
\end{proof}

\subsection{Lines and singularities of $\Cal S_V(\ptheta)$}

\begin{proposition}[Lines of $\Cal S_{V}(\ptheta)$] \label{prop:wm-linesV}
	The polynomial $\pF(\pX,\ptheta)$  \eqref{eq:wm-tildeF} can be decomposed as
	\begin{align*}
	\pF(\pX,\ptheta)
	&=(\pXone-\tilde e_1-\tfrac{1}{\tilde e_1})\pF_{\!\pXone}
	\ + \ \tilde e_1(\pXinf+\tfrac{\pXzero}{\tilde e_1}-a_\infty)(\pXzero+\tfrac{\pXinf}{\tilde e_1}-a_0),\\
	&=(\pXone-e_0e_\infty-\tfrac{1}{e_0e_\infty})\pF_{\!\pXone}
	\ + \ (e_0\pXinf+\tfrac{\pXzero}{e_\infty} -\tilde e_1-\tfrac{e_0}{e_\infty}) (\tfrac{\pXinf}{e_0}+ e_\infty \pXzero-\tilde e_1-\tfrac{e_\infty}{e_0}),\\
	&=(\pXone-\tfrac{e_0}{e_\infty}-\tfrac{e_\infty}{e_0})\pF_{\!\pXone}
	\ + \ (e_0\pXinf+e_\infty \pXzero-\tilde e_1-e_0e_\infty)(\tfrac{\pXinf}{e_0}+\tfrac{\pXzero}{e_\infty}-\tilde e_1-\tfrac{1}{e_0e_\infty}),\\
	&=(\pXzero-e_0)(\pF_{\!\pXzero}-\pXzero+e_0)
	\ + \ (\pXinf-\tfrac{\tilde e_1}{e_0})(e_0 \pXone+\pXinf-\tilde e_1 e_0-a_\infty),\\
	&=(\pXzero-e_0)(\tfrac{\tilde e_1}{e_0} \pXone+\pXzero-\tfrac{1}{e_0}-\tilde e_1 a_\infty)
	\ + \  (\pXinf-\tfrac{\tilde e_1}{e_0})(\pF_{\!\pXinf}-\pXinf+\tfrac{\tilde e_1}{e_0}),\\	
	&=(\pXzero-\tfrac{1}{e_0})(\pF_{\!\pXzero}-\pXzero+\tfrac{1}{e_0})
	\ + \  (\pXinf-\tilde e_1 e_0)(\tfrac{\pXone}{e_0}+\pXinf-\tfrac{\tilde e_1}{e_0}-a_\infty),\\
	&=(\pXzero-\tfrac{1}{e_0})(\tilde e_1 e_0 \pXone+\pXzero- e_0-\tilde e_1 a_\infty)
	 \ + \   (\pXinf-\tilde e_1 e_0)(\pF_{\!\pXinf}-\pXinf+\tilde e_1 e_0),\\
	&=(\pXinf-e_\infty)(\pF_{\!\pXinf}-\pXinf+e_\infty)
	\ + \ (\pXzero-\tfrac{\tilde e_1}{e_\infty})(e_\infty \pXone+\pXzero-\tilde e_1 e_\infty-a_0),\\
	&=(\pXinf-e_\infty)(\tfrac{\tilde e_1}{e_\infty} \pXone+\pXinf-\tfrac{1}{e_\infty}-\tilde e_1 a_0)
	\ + \  (\pXzero-\tfrac{\tilde e_1}{e_\infty})(\pF_{\!\pXzero}-\pXzero+\tfrac{\tilde e_1}{e_\infty}),\\
	&=(\pXinf-\tfrac{1}{e_\infty})(\pF_{\!\pXinf}-\pXinf+\tfrac{1}{e_\infty})
	\ + \ (\pXzero-\tilde e_1 e_\infty)(\tfrac{\pXone}{e_\infty}+\pXzero-\tfrac{\tilde e_1}{e_\infty}-a_0),\\
	&=(\pXinf-\tfrac{1}{e_\infty})(\tilde e_1 e_\infty \pXone+\pXinf- e_\infty-\tilde e_1 a_0)
	\ + \  (\pXzero-\tilde e_1 e_\infty)(\pF_{\!\pXzero}-\pXzero+\tilde e_1 e_\infty).
\end{align*}
	defining thus 18 lines on $\Cal S_{V}(\ptheta)$ (note that some lines appear twice in the above decomposition).
\end{proposition}

\begin{remark}
Each of the monodromies $M_0,M_\infty$ (assuming diagonalizable) define a pair of invariant eigenspaces of solutions, while each of the Stokes matrices $S_{1+},S_{2+}$
(assuming nontrivial) define one invariant eigenspaces of solutions.  	
As in Remark~\ref{remark:wm-mixedbasis}, the lines in $\Cal S_{V}(\ptheta)$ correspond to degeneration of certain mixed bases of solutions. 
More about this in a future work.
\end{remark}

The projective completion of $\Cal S_{V}(\ptheta)$ in $\CP^3$ contains 3 additional lines at infinity, giving a total 21 lines, and there is a singularity of type $A_1$ at the infinity (see \cite{BW79}). The surface $\Cal S_{V}(\ptheta)$ can have additional singularities, which happens if and only if some the 18 lines coincide.

\begin{proposition}[Singular points of $\Cal S_{V}(\ptheta)$]\label{prop:singularitiesV}
	The affine cubic variety $\Cal S_{V}(\ptheta)$ has singular points if and only if 
	\begin{equation*}
	(a_0^2-4)(a_\infty^2-4)\tilde w(\ptheta)=0,
	\end{equation*}
	where
	\begin{align*}
	\tilde w(\tilde a)&=(a_0^2+\tilde a_1^2+a_\infty^2-a_0\tilde a_1a_\infty-4), \quad \text{with}\quad \tilde a_1=\tilde e_1+\tfrac{1}{\tilde e_1},\\
	&=\frac{(e_0\tilde e_1e_\infty-1)(e_0\tilde e_1-e_\infty)(\tilde e_1e_\infty-e_0)(\tilde e_1-e_0e_\infty)}{(e_0\tilde e_1e_\infty)^2},
	\end{align*}
	(cf. \cite[section 3.2.2]{PS}).
	The corresponding possible singularities are the following:
	\begin{itemize}
		\item if $a_\infty=\pm 2$: 
		$\pXzero=\pm\tilde e_1,\quad \pXone=\pm a_0,\quad \pXinf=\pm 1$,
		\item if $a_0=\pm 2$: 
		$\pXzero=\pm 1,\quad \pXone=\pm a_\infty,\quad \pXinf=\pm\tilde e_1$,
		\item if $\tilde e_1=e_0^{\delta_0}e_\infty^{\delta_\infty}$, $(\delta_0,\delta_\infty)\in\{\pm1\}^2$: 
		$\pXzero=e_0^{\delta_0},\quad \pXone=\tilde e_1+\frac{1}{\tilde e_1},\quad \pXinf=e_\infty^{\delta_\infty}$.
	\end{itemize}
	Setting $\pX=(\tfrac{x_0}{v},\tfrac{x_1}{v},\tfrac{x_\infty}{v})$,
	the projective completion of $\Cal S_{V}(\ptheta)$ in $\CP^3$ has also a singularity at the point $(x_0:x_1:x_\infty:v)=(0:1:0:0)$ for any value of the parameters. 
\end{proposition}

\begin{proof}
	By Theorem~\ref{theorem:wm-Phi}, the surface $\Cal S_{V}(\ptheta)$ is isomorphic to $\Cal S_{VI}(\theta)\sminus L_0$,
	one can therefore use the description of the singular points of $\Cal S_{VI}(\theta)$ given in  Section~\ref{paragraph:wm-singularitiesSVI}.
\end{proof}

\subsection{The center manifold solution}

\begin{proposition}\label{prop:wm-centermanifold}
For $\epsilon\in\sE_+$ let $\pX$  be the coordinate \eqref{eq:wm-UW} on $\Cal S_{V}(\ptheta)$ which for $\epsilon\neq 0$ is given by the monodromy representation $\rho_+$ depending on $x=\frac{1}{\tilde t}+\epsilon=\frac{\epsilon t}{t-1}$ as in Figure~\ref{figure:wm-glueingx}.
Then the upper ``sectorial center manifold'' solution $\Psi_+^{\uppie}(0,x,\epsilon)$ of Corollaries~\ref{cor:wm-centermanifold} and~\ref{cor:wm-centermanifold0}
over the domain $\sX_+^{\uppie}$ corresponds to the point
\begin{equation}\label{eq:wm-sectorialcm}
	(\pXzero,\pXone,\pXinf)=(\tfrac{1}{e_0},\ \tilde e_1+e_0a_\infty-e_0^2\tilde e_1,\ e_0\tilde e_1)\quad\text{for}\ 
x\in\sX_+^{\uppie}(\epsilon),\ \epsilon\in\sE_+.
\end{equation}
\end{proposition}

%\begin{remark}
%The above point on $\Cal S_V(\ptheta)$ correspond to a condition on the matrix $S_2M_{0+}S_2^{-1}$ in the representation \eqref{eq:wm-confluentM} being diagonal.
%\end{remark}
	
%\label{page:sectorial}
%
%\begin{proof}[Proof of Theorem~\ref{thm:sectorial}]
%The point \eqref{eq:wm-sectorialcm} corresponding to the ``upper sectorial center manifold'' solution on $\sX_+^{\uppie}$  is fixed by the action of $\tilde g_{0\infty}^{\circ 2}$ (Proposition~\ref{prop:wm-gt1}) which corresponds to the monodromy along a loop around $x=0,\epsilon$ for all $\epsilon\in\sE_+$ and therefore also at the limit $\epsilon=0$.
%
%The ``lower sectorial center manifold'' solution on $\sX_+^{\downpie}$ is transformed to an upper one (for a different value of $(\vartheta_0,\tilde\vartheta_1,\vartheta_\infty)$) by the Okamoto transformation \eqref{eq:wm-Okamototransformation}.
%\end{proof}	

The proof of Proposition~\ref{prop:wm-centermanifold} is based on the following reformulation of the Jimbo's asymptotic formula.

\begin{proposition}[Jimbo's formula for the confluent system]\label{prop:wm-Jimboconfluent}
For $\epsilon\in\sE_+\sminus\{0\}$ let $\pX$  be the coordinate \eqref{eq:wm-UW} on $\Cal S_{V}(\ptheta)$ given by the monodromy representation $\rho_+$ depending on $x=\frac{1}{\tilde t}+\epsilon=\frac{\epsilon t}{t-1}$ as in Figure~\ref{figure:wm-glueingx}.
Let $\transp(q(x),p(x))=\transp(q_+^{\uppie}(x),p_+^{\uppie}(x))$ be the solution of the confluent Painlev\'e system \eqref{eq:wm-PVIhamiltonianx} corresponding to a point $\pX\in \Cal S_{V}(\ptheta)$ over the domain $x\in\sX_+^{\uppie}$, which corresponds (up to analytic extension) to the upper half-plane in Figure~\ref{figure:wm-glueingx}.
\begin{enumerate}
\item When $x\to 0$, $x\in\sX_+^{\uppie}(\epsilon)$, $\epsilon\in\sE_+$:
	\begin{align*}
	q&\sim\alpha(\vartheta,X)\left(\frac{x}{x-\epsilon}\right)^{1-\sigma_1}+\ O\big(\big(\tfrac{x}{x-\epsilon}\big)^{2-2\sigma_1}\big), \\
	p&\sim\frac{\vartheta_0+\vartheta_t-\sigma_1}{2\alpha(\vartheta,X)}\left(\frac{x}{x-\epsilon}\right)^{\sigma_1-1}+\ O(1),
	\end{align*}
	where $\sigma_1$, defined by $X_1=e^{\pi i\sigma_1}+e^{-\pi i\sigma_1}$, and  $\alpha(X,\vartheta)$ are as in Proposition~\ref{prop:wm-Jimbo}, and $(X,\vartheta)$ is related to $(\pX,\tilde\vartheta,\epsilon)$ by 
	the birational transformation \eqref{eq:wm-XXXUW} and \eqref{eq:wm-tt}.
	
\item When $x\to \epsilon$, $x\in\sX_+^{\uppie}(\epsilon)$, $\epsilon\in\sE_+$:
	\begin{align*}
	q&\sim\frac{1}{\alpha(\vartheta',X')}\left(\frac{x-\epsilon}{x}\right)^{\sigma_1'-1}+\ O(1), \\
	p&\sim\alpha(\vartheta',X')\frac{\vartheta_0+\vartheta_1+\sigma_1'-2}{2}\left(\frac{x-\epsilon}{x}\right)^{1-\sigma_1'}
	+\ O\big(\big(\tfrac{x-\epsilon}{x}\big)^{2-2\sigma_1'}\big),
	\end{align*}
	where $(X',\vartheta')=g_{\infty 0}(X,\theta)$, i.e.
	\[\vartheta'=(\vartheta_\infty,\vartheta_t,\vartheta_1,\vartheta_0), \qquad X'=(X_0,\, X_1-F_{\!X_1},\, X_t),\] 
	and $(X,\vartheta)$ are related to $(\pX,\tilde\vartheta,\epsilon)$ by \eqref{eq:wm-XXXUW} and \eqref{eq:wm-tt}.
\end{enumerate} 
\end{proposition}

\begin{proof}[Proof of Proposition~\ref{prop:wm-centermanifold}]
By Proposition~\ref{prop:wm-Jimboconfluent}, in order for the solution $\transp(q_+^{\uppie}(x),p_+^{\uppie}(x))$  to be bounded at $x=0$ we need to have $\sigma_1=\vartheta_0+\vartheta_t$, while $\alpha(\vartheta,X)\neq 0,\infty$. Using \eqref{eq:wm-ab} let us write $\alpha(\vartheta,X)=\frac{A(\vartheta,\sigma_1)}{e^{\pi i\sigma_1}F_{\!X_t}-F_{\!X_0}}$ for some function $A$. 
This gives us the following condition:
\begin{equation}\label{eq:wm-cond1}
X_1=e_0e_t+\tfrac{1}{e_0e_t},	\qquad e_0e_t F_{\!X_t}-F_{\!X_0}\neq0.
\end{equation}
Similarly, in order for the solution $\transp(q_+^{\uppie}(x),p_+^{\uppie}(x))$ to be bounded at $x=\epsilon$ we need to have 
$\sigma_1'=2-\vartheta_0-\vartheta_1$, while $\alpha(\vartheta',X')\neq 0,\infty$, meaning that $e^{\pi i\sigma_1'} F'_{\!X'_t}-F'_{\!X'_0}\neq0$. 
This gives us the following condition:
\begin{equation}\label{eq:wm-cond2}
X_t=e_0e_1+\tfrac{1}{e_0e_1},	\qquad e_0e_1 F_{\!X_1}-F_{\!X_0}\neq0.
\end{equation}
By Proposition~\ref{prop:wm-lines}, the identity in \eqref{eq:wm-cond2} means  that either
\[\begin{cases} \text{(a)} & e_0e_1X_0+X_1-e_1a_\infty-e_0a_t=0,\\
	\text{(b)} & e_0e_1X_1+X_0-e_0a_\infty-e_1a_t=0,\end{cases}\]
from which we have respectively
\[\begin{cases} \text{(a)} & X_0=\tfrac{1}{e_te_1}+\tfrac{a_\infty}{e_0}-\tfrac{1}{e_0^2e_te_1},\\
	\text{(b)} & X_0=e_te_1+e_0a_\infty-e_0^2e_te_1.\end{cases} \]
The case (a) doesn't satisfy neither the condition \eqref{eq:wm-cond1} nor \eqref{eq:wm-cond2}.
In the case (b) using the formulas \eqref{eq:wm-UWXXX} we obtain that
\[\pXzero=\tfrac{1}{e_0},\quad \pXone=e_te_1+e_0 a_\infty-e_0^2e_te_1,\quad  \pXinf=e_0e_te_1.\]
\end{proof}

\begin{figure}[t]
	\centering
	\begin{tikzpicture} [scale=0.25]
	\draw[dashed, ultra thick] (-20,0) -- (20,0);
	\node at (-6,0)[circle,fill,inner sep=2pt]{}; \draw (-6,-1.5) node {$x\!=\!0$};
	\node at (6,0)[circle,fill,inner sep=2pt]{}; \draw (6,-1.5) node {$x\!=\!\epsilon$};
	\draw[<->, thick] (-16.5,2) .. controls +(-0.5,-1) and +(-0.5,1) .. (-16.5,-2); \draw (-18,1.5) node {$\id$};	
	\draw[->, thick] (0,2) -- (0,-2); \draw (0,1.5) node[left] {$\tilde g_{0t}^{\circ 2}$};
	\draw[->, thick] (16.5,-2) .. controls +(0.5,1) and +(0.5,-1) .. (16.5,2); \draw (16.5,-1.5) node[right] {$\tilde g_{t1}^{\circ 2}=\tilde g_{\infty 0}^{\circ 2}$};
	
	\begin{scope}[shift={(-13,6)}]
	\node at (-3,0)[circle,fill,inner sep=1pt]{}; \draw (-3.2,-1) node {$0$};
	\node at (-1,0)[circle,fill,inner sep=1pt]{}; \draw (-1,-1) node {$t$}; 
	\node at (1,0)[circle,fill,inner sep=1pt]{}; \draw (1,-1) node {$1$};
	\node at (3,0)[circle,fill,inner sep=1pt]{}; \draw (3.2,-1) node {$\infty$};
	\draw (0,3.6) node {$z_0$};  
	\draw[-] (-3,0) -- (0,3); 
	\draw[-] (-1,0) -- (0,3); 
	\draw[-] (1,0) -- (0,3);
	\draw[-] (3,0) -- (0,3); 
	\end{scope}

	\begin{scope}[shift={(0,6)}]
	\node at (-3,0)[circle,fill,inner sep=1pt]{}; \draw (-3.5,-1) node {$t$}; 
	\node at (-1,0)[circle,fill,inner sep=1pt]{}; \draw (-1,-1) node {$0$};
	\node at (1,0)[circle,fill,inner sep=1pt]{}; \draw (1,-1) node {$1$};
	\node at (3,0)[circle,fill,inner sep=1pt]{}; \draw (3.2,-1) node {$\infty$};
	\draw (0,3.6) node {$z_0$}; 
	\draw[-] (-3,0) .. controls +(0,-3) and +(0,-3) .. (0,0) -- (0,3); 
	\draw[-] (-1,0) -- (0,3); 
	\draw[-] (1,0) -- (0,3); 
	\draw[-] (3,0) -- (0,3); 
	\end{scope}
	
	\begin{scope}[shift={(13,6)}]
	\node at (-3,0)[circle,fill,inner sep=1pt]{}; \draw (-3.2,-1) node {$0$};
	\node at (-1,0)[circle,fill,inner sep=1pt]{}; \draw (-1,-1) node {$1$};
	\node at (1,0)[circle,fill,inner sep=1pt]{}; \draw (1.5,-1) node {$t$}; 
	\node at (3,0)[circle,fill,inner sep=1pt]{}; \draw (3.2,-1) node {$\infty$};
	\draw (0,3.6) node {$z_0$}; 
	\draw[-] (-3,0) -- (0,3); 
	\draw[-] (-1,0) -- (0,3); 
	\draw[-] (1,0) .. controls +(0,-3) and +(-1,-3) .. (-2,0) -- (0,3); 
	\draw[-] (3,0) -- (0,3); 
	\end{scope}
	
	\begin{scope}[shift={(-13,-7)}]
	\node at (-3,0)[circle,fill,inner sep=1pt]{}; \draw (-3.2,-1) node {$0$};
	\node at (-1,0)[circle,fill,inner sep=1pt]{}; \draw (-1,-1) node {$t$}; 
	\node at (1,0)[circle,fill,inner sep=1pt]{}; \draw (1,-1) node {$1$};
	\node at (3,0)[circle,fill,inner sep=1pt]{}; \draw (3.2,-1) node {$\infty$};
	\draw (0,3.6) node {$z_0$};  
	\draw[-] (-3,0) -- (0,3); 
	\draw[-] (-1,0) -- (0,3); 
	\draw[-] (1,0) -- (0,3);
	\draw[-] (3,0) -- (0,3); 
	\end{scope}
	
	\begin{scope}[shift={(0,-7)}]
	\node at (-3,0)[circle,fill,inner sep=1pt]{}; \draw (-3.2,-1) node {$t$};
	\node at (-1,0)[circle,fill,inner sep=1pt]{}; \draw (-0.7,-1) node {$0$}; 
	\node at (1,0)[circle,fill,inner sep=1pt]{}; \draw (1,-1) node {$1$};
	\node at (3,0)[circle,fill,inner sep=1pt]{}; \draw (3.2,-1) node {$\infty$};
	\draw (0,3.6) node {$z_0$}; 
	\draw[-] (-3,0) -- (0,3); 
	\draw[-] (-1,0) .. controls +(-0.5,-3) and +(-2.5,-2) .. (-4.5,0) -- (0,3); 
	\draw[-] (1,0) -- (0,3); 
	\draw[-] (3,0) -- (0,3); 
	\end{scope}
	
	\begin{scope}[shift={(13,-7)}]
	\node at (-3,0)[circle,fill,inner sep=1pt]{}; \draw (-3,-1) node {$0$};
	\node at (-1,0)[circle,fill,inner sep=1pt]{}; \draw (-1.2,-1) node {$1$}; 
	\node at (1,0)[circle,fill,inner sep=1pt]{}; \draw (1,-1) node {$t$};
	\node at (3,0)[circle,fill,inner sep=1pt]{}; \draw (3.4,-1) node {$\infty$};
	\draw (0,3.6) node {$z_0$}; 
	\draw[-] (-3,0) -- (0,3); 
	\draw[-] (-1,0) .. controls +(1,-3) and +(1,-2) .. (2,0) -- (0,3); 
	\draw[-] (1,0) -- (0,3); 
	\draw[-] (3,0) -- (0,3); 
	\end{scope}

	\end{tikzpicture}
	\vskip-10pt
	\caption{The paths around which are taken the loops $\gamma_0,\gamma_t,\gamma_1$ defining the confluent monodromy representation $\rho_+$  
		and therefore the coordinate $\pX$ on $\Cal S_{V}(\ptheta)$ for $\epsilon\in\sE_+\sminus\{0\}$ in dependence on $x=\frac{\epsilon t}{t-1}\in \C\sminus\{0,\epsilon\}$ (cf. Figure~\ref{figure:wm-glueing}), and the corresponding transition maps $\tilde g_{0t}^{\circ2},\ \tilde g_{\infty0}^{\circ2}(e^{\frac{2\pi i}{\epsilon}}):\Cal S_V(\ptheta)\to \Cal S_V(\ptheta)$.}
	\label{figure:wm-glueingx}
\end{figure}

\begin{proof}[Proof of Proposition~\ref{prop:wm-Jimboconfluent}]
The asymptotics at $x=0$ is simply the Jimbo's formula of Proposition~\ref{prop:wm-Jimbo}, which for $\epsilon\in\R_{>0}$ is defined for $|\arg(-x)|<\pi$.

The asymptotics at $x=\epsilon$ is obtained from the Jimbo's formula using the Okamoto's transformation
\[t'=t^{-1},\quad q'=q^{-1}, \quad p'=q\cdot(\tfrac{\vartheta_0+\vartheta_t+\vartheta_1+\vartheta_\infty}{2}-1-qp),\quad \vartheta'=(\vartheta_\infty,\vartheta_t,\vartheta_1,\vartheta_0),\]	
which preserves the Hamiltonian system of $P_{VI}$,
i.e.
\begin{equation}\label{eq:wm-Okamototransformation}
x'=\epsilon-x,\quad q'=q^{-1}, \quad p'=q\cdot(\tfrac{\vartheta_0+\tilde\vartheta_1+\vartheta_\infty}{2}-1-qp),\quad (\vartheta_0',\tilde\vartheta_1',\vartheta_\infty')=(\vartheta_\infty,\tilde\vartheta_1,\vartheta_0),
\end{equation}
which preserves the system \eqref{eq:wm-PVIhamiltonianx}.
This corresponds to the transformation
\[z'=z^{-1},\quad t'=t^{-1},\quad (A_0',A_{t'}',A_1',A_\infty')=(A_\infty,A_t,A_1,A_0),\]
of the isomonodromic problem \eqref{eq:wm-A}.
It transforms the monodromy representation $\rho_+$ (Figure~\ref{figure:wm-glueingx}) to $\rho'$, $\rho'(\gamma'_{j'})=M_j$, $j=0,t,1,\infty$,
%\[\rho':(\gamma_0',\gamma_t',\gamma_1',\gamma_\infty')\mapsto (M_\infty,M_t,M_1,M_0),\]
where for $x\in\sX_+^{\uppie}(\epsilon)$, $\epsilon\in\sE_+$, one has up to a conjugation (see Figure~\ref{figure:wm-glueingxx})
\[\gamma'_{0'}=\gamma_\infty,\quad 
\gamma'_{t'}=\gamma_t,\quad 
\gamma'_{1'}=\gamma_1,\quad 
\gamma'_{\infty'}=\gamma_\infty\gamma_0\gamma_\infty^{-1}.\]
This means that up to a conjugation
\begin{align*}
M_0'&=\rho'(\gamma_0)=M_0M_\infty M_0^{-1}, \quad \\
M_t'&=\rho'(\gamma_t)=M_t, \quad \\
M_1'&=\rho'(\gamma_1)=M_1, \quad \\
M_\infty'&=\rho'(\gamma_\infty)=M_0,
\end{align*}
i.e. $\rho'=(\beta_{\infty0})_*\rho_+$ \eqref{eq:wm-beta},
and hence \eqref{eq:wm-ginfty0}
\[X'=g_{\infty0}(X)=(X_0,\ X_1-F_{\!X_1},\ X_t).\]
\end{proof}

\begin{figure}[t]
\centering
\begin{tikzpicture} [scale=0.25]
	\draw[dashed, ultra thick] (-20,0) -- (20,0);
	\node at (-6,0)[circle,fill,inner sep=2pt]{}; \draw (-6,-1.5) node {$x'\!=\!0$};
	\node at (6,0)[circle,fill,inner sep=2pt]{}; \draw (6,-1.5) node {$x'\!=\!\epsilon$};
	\draw[<->, thick] (16.5,-2) .. controls +(0.5,1) and +(0.5,-1) .. (16.5,2); \draw (18,-1.5) node {$\id$};	
	\draw[->, thick] (0,-2) -- (0,2); \draw (0,-1.5) node[left] {$\tilde g_{0t}^{\circ 2}$};
	\draw[->, thick] (-16.5,2) .. controls +(-0.5,-1) and +(-0.5,1) .. (-16.5,-2); \draw (-16.5,1.5) node[left] {$\tilde g_{t1}^{\circ 2}=\tilde g_{\infty 0}^{\circ 2}$};
		
	\begin{scope}[shift={(-13,8)}]
		\draw (0,-3.8) node {$z_0'$}; 
		\node at (-3,0)[circle,fill,inner sep=1pt]{}; \draw (-3.5,1) node {$\infty'$};
		\node at (-1,0)[circle,fill,inner sep=1pt]{}; \draw (-0.7,1) node {$t'$};
		\node at (1,0)[circle,fill,inner sep=1pt]{}; \draw (1.7,1) node {$1'$}; 
		\node at (3,0)[circle,fill,inner sep=1pt]{}; \draw (3.5,1) node {$0'$}; 
		\draw[-] (-3,0) -- (0,-3); 
		\draw[-] (-1,0) -- (0,-3); 
		\draw[-] (1,0) .. controls +(0,4) and +(-2,3) .. (-2,0) -- (0,-3); 
		\draw[-] (3,0) -- (0,-3); 
	\end{scope}

	\begin{scope}[shift={(0,8)}]
		\draw (0,-3.8) node {$z_0'$}; 
		\node at (-3,0)[circle,fill,inner sep=1pt]{}; \draw (-3.2,1) node {$t'$};
		\node at (-1,0)[circle,fill,inner sep=1pt]{}; \draw (-1,1) node {$\infty'$}; 
		\node at (1,0)[circle,fill,inner sep=1pt]{}; \draw (1,1) node {$1'$};	
		\node at (3,0)[circle,fill,inner sep=1pt]{}; \draw (3.7,1) node {$0'$};
		\draw[-] (-3,0) -- (0,-3); 
		\draw[-] (-1,0) -- (0,-3); 
		\draw[-] (1,0) -- (0,-3); 
		\draw[-] (3,0) .. controls +(0,4) and +(-2,3) .. (-2,0) -- (0,-3); 
	\end{scope}

	\begin{scope}[shift={(13,8)}]
		\draw (0,-3.8) node {$z_0'$}; 
		\node at (-3,0)[circle,fill,inner sep=1pt]{}; \draw (-3,1) node {$\infty'$};
		\node at (-1,0)[circle,fill,inner sep=1pt]{}; \draw (-1,1) node {$1'$};
		\node at (1,0)[circle,fill,inner sep=1pt]{}; \draw (1.2,1) node {$t'$}; 
		\node at (3,0)[circle,fill,inner sep=1pt]{}; \draw (3.4,1) node {$0'$};
		\draw[-] (-3,0) -- (0,-3);
		\draw[-] (-1,0) -- (0,-3); 
		\draw[-] (1,0) -- (0,-3); 
		\draw[-] (3,0) -- (0,-3); 
	\end{scope}

	\begin{scope}[shift={(-13,-6)}]
		\draw (0,-3.8) node {$z_0'$}; 
		\node at (-3,0)[circle,fill,inner sep=1pt]{}; \draw (-3.5,1) node {$\infty'$};
		\node at (-1,0)[circle,fill,inner sep=1pt]{}; \draw (-1.5,1) node {$t'$};
		\node at (1,0)[circle,fill,inner sep=1pt]{}; \draw (1,1) node {$1'$}; 
		\node at (3,0)[circle,fill,inner sep=1pt]{}; \draw (3.5,1) node {$0'$}; 
		\draw[-] (-3,0) -- (0,-3); 
		\draw[-] (-1,0) .. controls +(0,4) and +(2,3) .. (2,0) -- (0,-3); 
		\draw[-] (1,0) -- (0,-3); 
		\draw[-] (3,0) -- (0,-3); 
	\end{scope}
	
	\begin{scope}[shift={(0,-6)}]
		\draw (0,-3.8) node {$z_0'$}; 
		\node at (-3,0)[circle,fill,inner sep=1pt]{}; \draw (-3.5,1) node {$t'$};
		\node at (-1,0)[circle,fill,inner sep=1pt]{}; \draw (-1,1) node {$\infty'$}; 
		\node at (1,0)[circle,fill,inner sep=1pt]{}; \draw (1,1) node {$1'$};	
		\node at (3,0)[circle,fill,inner sep=1pt]{}; \draw (3.7,1) node {$0'$};
		\draw[-] (-3,0) .. controls +(0,4) and +(2,3) .. (2,0) -- (0,-3); 
		\draw[-] (-1,0) -- (0,-3); 
		\draw[-] (1,0) -- (0,-3); 
		\draw[-] (3,0) -- (0,-3);
	\end{scope}
	
	\begin{scope}[shift={(13,-6)}]
		\draw (0,-3.8) node {$z_0'$}; 
		\node at (-3,0)[circle,fill,inner sep=1pt]{}; \draw (-3,1) node {$\infty'$};
		\node at (-1,0)[circle,fill,inner sep=1pt]{}; \draw (-1,1) node {$1'$};
		\node at (1,0)[circle,fill,inner sep=1pt]{}; \draw (1.2,1) node {$t'$}; 
		\node at (3,0)[circle,fill,inner sep=1pt]{}; \draw (3.4,1) node {$0'$};
		\draw[-] (-3,0) -- (0,-3);
		\draw[-] (-1,0) -- (0,-3); 
		\draw[-] (1,0) -- (0,-3); 
		\draw[-] (3,0) -- (0,-3); 
	\end{scope}

	\begin{scope}[color=red, shift={(-13,-18)}]
		\draw (0,6) node {\rotatebox{-90}{$\simeq$}};
		\draw (0,3.6) node {$z_0$};
		\node at (-3,0)[circle,fill,inner sep=1pt]{}; \draw (-3.9,-1) node {$\infty'$};	
		\node at (-1,0)[circle,fill,inner sep=1pt]{}; \draw (-1,-1) node {$t'$};
		\node at (1,0)[circle,fill,inner sep=1pt]{}; \draw (1,-1) node {$1'$}; 
		\node at (3,0)[circle,fill,inner sep=1pt]{}; \draw (3.5,-1) node {$0'$};
		\draw[-] (-3,0) .. controls +(0,-4) and +(2,-3) .. (2,0) -- (0,3); 
		\draw[-] (-1,0) -- (0,3); 
		\draw[-] (1,0) -- (0,3); 
		\draw[-] (3,0) -- (0,3); 
	\end{scope}

	\begin{scope}[color=red, shift={(0,-18)}]
		\draw (0,6) node {\rotatebox{-90}{$\simeq$}};
		\draw (0,3.6) node {$z_0$}; 
		\node at (-3,0)[circle,fill,inner sep=1pt]{}; \draw (-3.5,-1) node {$t'$}; 
		\node at (-1,0)[circle,fill,inner sep=1pt]{}; \draw (-1.9,-1) node {$\infty'$}; 
		\node at (1,0)[circle,fill,inner sep=1pt]{}; \draw (1,-1) node {$1'$};
		\node at (3,0)[circle,fill,inner sep=1pt]{}; \draw (3.5,-1) node {$0'$};
		\draw[-] (-3,0) -- (0,3); 
		\draw[-] (-1,0) .. controls +(0,-4) and +(2,-3) .. (2,0) -- (0,3); 
		\draw[-] (1,0) -- (0,3); 
		\draw[-] (3,0) -- (0,3); 
	\end{scope}

	\begin{scope}[color=red, shift={(13,-18)}]
		\draw (0,6) node {\rotatebox{-90}{$\simeq$}};
		\draw (0,3.6) node {$z_0$}; 
		\node at (-3,0)[circle,fill,inner sep=1pt]{}; \draw (-3.9,-1) node {$\infty'$};
		\node at (-1,0)[circle,fill,inner sep=1pt]{}; \draw (-1.6,-1) node {$1'$}; 
		\node at (1,0)[circle,fill,inner sep=1pt]{}; \draw (1,-1) node {$t'$};
		\node at (3,0)[circle,fill,inner sep=1pt]{}; \draw (3.6,-1) node {$0'$};
		\draw[-] (-3,0) .. controls +(0,-4) and +(2,-3) .. (2.3,0) -- (0,3); 
		\draw[-] (-1,0) .. controls +(0,-3) and +(1,-2) .. (1.7,0) -- (0,3); 
		\draw[-] (1,0) -- (0,3); 
		\draw[-] (3,0) -- (0,3); 
	\end{scope}

\end{tikzpicture}
	\vskip-12pt
	\caption{The transform of the Figure~\ref{figure:wm-glueingx} by $x'=\epsilon-x$, $z'=z^{-1}$, $(0',t',1',\infty')=(\infty,t^{-1},1,0)$, 
		showing the paths around which are taken the loops $\gamma'_{0'},\gamma'_{t'},\gamma'_{1'},\gamma'_{\infty'}$ defining the confluent monodromy representation $\rho'$ and therefore the coordinate $X'$ on $x$. We are only interested in the lower half-plane part of this picture. The diagrams in the bottom row (red) are conjugated to those above (black) by moving the basepoint from $z'_0$ to $z_0$ in order  to compare with those in Figure~\ref{figure:wm-glueingx}.}
	\label{figure:wm-glueingxx}
\end{figure}

\section{The wild monodromy action on $\Cal S_V(\ptheta)$}\label{section:wm-2wild}

The only nonlinear monodromy actions (pure braid actions) on ${\Cal S}_{V}(\ptheta)$ that converge during the confluence are those generated by $\tilde g_{\infty0}^{\circ 2}$ 
(in Proposition~\ref{prop:wm-gt1}).
As was explained in Section~\ref{section:wm-1a2}, for the other actions we need instead to consider limits of their pullbacks by $\Phi_+$ \eqref{eq:wm-XXXUW} along the sequences 
\[(\epsilon_n)_{n\in\pm\N}, \quad \tfrac{1}{\epsilon_n}=\tfrac{1}{\epsilon_0}+n\] 
along which the divergent parameter $e_t^2=e^{\frac{2\pi i}{\epsilon}}$ stays constant.
This amounts to replacing $e_t^2$ by a new independent parameter $\kappa\in\C^*$, i.e. writing
\begin{equation}\label{eq:wm-kappa}
e_te_1=\tilde e_1,\qquad \frac{e_t}{e_1}=\frac{\kappa}{\tilde e_1}.
\end{equation}
The idea of taking the limit of the Riemann--Hilbert correspondence for the isomonodromic problem \eqref{eq:wm-tildeA} along such discrete sequences $(\epsilon_n)$ was already considered by Kitaev \cite{Kit} but with a different aim. 
The limit of a nonlinear monodromy operator accumulating towards a one-parameter family of wild monodromy operators acting on the wild character variety  have not been studied before.

\begin{proposition}\label{prop:wm-tildeg0t}
	For $\epsilon\neq 0$, let $\Phi_+$ be the transformation \eqref{eq:wm-XXXUW}, and let 
	\[\tilde g_{ij}^{\circ 2}(\,\cdot\,;\kappa):=\Phi_+^{-1}\circ g_{ij}^{\circ 2}\circ\Phi_+:\Cal S_{V}(\ptheta)\to \Cal S_{V}(\ptheta),\quad \{i,j\}=\{0,t\},\] 
	be the the pullback of the monodromy action $g_{ij}^{\circ 2}$,	composed with the substitution~\eqref{eq:wm-kappa}. Then
\begin{equation*}
%\mbox{\scriptsize$
\begin{aligned}
		\tilde g_{0t}^{\circ 2}(\,\cdot\,;\kappa):\ \ 
		a_0&\mapsto a_0 & \pXzero&\mapsto\frac{\tilde e_1}{\pXinf},\\
		\tilde e_1 &\mapsto \tilde e_1 & \pXone&\mapsto \pXone+\tfrac{1}{\tilde e_1}(\pXone\pF_{\!\pXone}-\pXinf\pF_{\!\pXinf})+\tfrac{\kappa}{\tilde e_1^2}\pXinf\pF_{\!\pXzero}-\pF_{\!\pXone}\big(\tfrac{1}{\kappa}+\tfrac{\kappa}{\tilde e_1^2}\big),\\
		a_\infty&\mapsto a_\infty & \pXinf&\mapsto \pXinf+\frac{\tilde e_1}{\kappa\pXinf}\pF_{\!\pXone},\\[8pt]
		%\end{aligned}
		%$}
		%\end{equation*}
		%
		%\begin{equation*}
		%\mbox{\scriptsize$
		%\begin{aligned}
		\tilde g_{t0}^{\circ 2}(\,\cdot\,;\kappa):\ \ 
		a_0&\mapsto a_0 &\pXzero&\mapsto \pXzero+\frac{\kappa}{\tilde e_1\pXzero}\pF_{\!\pXone},\\
		\tilde e_1&\mapsto \tilde e_1 &\pXone&\mapsto \pXone+\tfrac{1}{\tilde e_1}(\pXone\pF_{\!\pXone}-\pXzero\pF_{\!\pXzero})+\tfrac{1}{\kappa}\pXzero\pF_{\!\pXinf}-\pF_{\!\pXone}\big(\tfrac{1}{\kappa}+\tfrac{\kappa}{\tilde e_1^2}\big),\\
		a_\infty&\mapsto a_\infty &\pXinf&\mapsto\frac{\tilde e_1}{\pXzero}.
\end{aligned}
%$}
\end{equation*}
They preserve the symplectic form $\tilde{\omega}_{\Cal S_{V}}$, and change the Fricke relation by a factor:
\[\pF\circ \tilde g_{t0}^{\circ 2}(\pX;\kappa)=
\frac{(\pXone-\tfrac{\tilde e_1}{\kappa}-\tfrac{\kappa}{\tilde e_1})}
{\big(\pXone-\tfrac{\tilde e_1}{\kappa}-\tfrac{\kappa}{\tilde e_1}-(F_{\!X_0}\circ\Phi_+)\big)}\pF(\pX,\tilde{\theta}).\]
\end{proposition}

\begin{proof}
 The formulas are obtained by plugging \eqref{eq:wm-XXXUW}, \eqref{eq:wm-UWXXX} to $g_{0t}^{\circ 2}$ and $g_{t0}^{\circ 2}$ \eqref{eq:wm-gij2}, where the action of  $g_{0t}^{\circ 2}$ fixes $(e_t,e_1)$, and hence also $(\tilde e_1,\kappa)$.	
	
Alternatively, the action \eqref{eq:wm-beta} of $(\beta_{0t}^2)_*$  on the monodromy representation $\rho_+$ is
\begin{equation*}
\begin{aligned}
(\beta_{0t}^2)_*:\
M_{0+} &\mapsto (M_{t+}M_{0+})M_{0+}(M_{t+}M_{0+})^{-1},\qquad\\
M_{t+} &\mapsto (M_{t+}M_{0+})M_{t+}(M_{t+}M_{0+})^{-1},\\
M_{1+} &\mapsto M_{1+},\\
M_{\infty+} &\mapsto M_{\infty+},
\end{aligned}
\end{equation*}
where by \eqref{eq:wm-confluentM} $M_{t+}M_{0+}=N_tS_{2+}M_{0+}$.
Decompose $S_{2+}M_{0+}=LU$ with 
\[L=\begin{pmatrix} \alpha+\gamma s_{2+} & 0 \\ \gamma & \delta-\gamma\tfrac{\beta+\delta s_{2+}}{\alpha+\gamma s_{2+}} \end{pmatrix},\qquad
U=\begin{pmatrix}1 & \tfrac{\beta+\delta s_{2+}}{\alpha+\gamma s_{2+}} \\ 0 & 1 \end{pmatrix},\]
and conjugate the above monodromy representation by $N_tL$, 
so that the image of $M_{t+}=N_tS_{2+}$, resp. $M_{1+}=S_{1+}N_1$, 
is again upper, resp. lower, triangular:
\begin{equation*}
	\begin{aligned}
		M_{0+} &\mapsto UM_{0+}U^{-1},&&\\
		M_{t+} &\mapsto UM_{t+}U^{-1}, & 	S_{2+} &\mapsto N_t(\kappa)^{-1}UN_t(\kappa)S_{2+}U^{-1}\\
		M_{1+} &\mapsto L^{-1}N_t(\kappa)^{-1}M_{1+}N_t(\kappa)L,\quad & S_{1+} &\mapsto L^{-1}N_t(\kappa)^{-1}S_{1+}NLN^{-1}N_t(\kappa), \\
		M_{\infty+} &\mapsto L^{-1}N_t(\kappa)^{-1}M_{\infty+}N_t(\kappa)L, &&
	\end{aligned}
\end{equation*}
where
\[N_t(\kappa)=\begin{pmatrix}\kappa^{\frac12} & 0 \\ 0&\kappa^{-\frac12}  \end{pmatrix},\qquad 
N=\begin{pmatrix}\tilde e_1 & 0 \\ 0&\tilde e_1^{-1}  \end{pmatrix},\]
and in particular the images of $M_{t+}$ and $M_{1+}$ have the same diagonal parts as $M_{t+}$ and $M_{1+}$ have. 
From this one can express the action of $\tilde g_{0t}^{\circ 2}(\kappa)$ on $\tilde X$  \eqref{eq:wm-UW}
in terms of $\alpha,\beta,\gamma,\delta,s_{1+},s_{2+}$ and then re-express it in terms of $\pX,\tilde{\theta}$.
Similarly for the inverse $\tilde g_{t0}^{\circ 2}(\kappa)$.
\end{proof}

The ``initial conditions'' $c=(c_1,c_2)$ define through the map $(x,c)\mapsto \transp(q^{\uppie}(x;c),p^{\uppie}(x;c))$ \eqref{eq:wm-qp0} local coordinates on the space of leaves of the Painlev\'e foliation over the sector $x\in\sX^{\uppie}(0)$, $|c|<\delta_c$.
Therefore, $c$ defines local coordinates also on the Okamoto space of initial conditions $\Cal M_{V,\tilde t}$ \eqref{eq:wm-Okamotospace} of $P_V$ for each $x=\tilde t^{-1}\in \sX^{\uppie}(0)$
on a neighborhood of the point which corresponds to the sectoral center manifold $c=0$.
The Okamoto space space $\Cal M_{V,\tilde t}$ and the wild character variety $\Cal S_V(\ptheta)$ are isomorphic on some Zariski open set
\[ \Cal M_{V,\tilde t}(\tilde\vartheta) \dashrightarrow \Cal S_{V}(\ptheta).\]
It has been conjectured in \cite{PS} that in fact $\Cal M_{V,\tilde t}$ is a isomorphic to a minimal strict resolution of $\Cal M_{V,\tilde t}$ if singular (for $P_{VI}$ this is known to be true, see \cite{IIS} and the references there). 
Anyway, it means that the wild monodromy pseudogroup of Definition~\ref{def:wm-pseudogroup0} acts locally on $\Cal S_{V}(\ptheta)$ near the point \eqref{eq:wm-sectorialcm}.
It turns out this action is in fact global (Proposition~\ref{prop:wm-wildmonodromy} below). By Proposition~\ref{prop:wm-accumulation}, the wild monodromy pseudogroup
is obtained as a limit of the nonlinear monodromy group through the accumulation along the discrete sequences \eqref{eq:wm-epsilon_n} of the parameter $\epsilon$.

\begin{proposition}\label{prop:wm-wildmonodromy}
For $\epsilon=0$, the action of the nonlinear wild monodromy operators $\tilde{\mbf M}_{0+}^{\uppie}(c;\kappa)$ and $\tilde{\mbf M}_{\epsilon+}^{\uppie}(c;\kappa)$ \eqref{eq:wm-wildmonodromy}, which are defined locally near the point corresponding to $c=0$ on the Okamoto fiber $\Cal M_{V,\tilde t}$ over $x=\tilde t^{-1}\in\sX^{\uppie}(0)$, extend as bimeromorphic maps to the whole fiber.
In the coordinate $\pX$ \eqref{eq:wm-UW} on $\Cal S_V(\ptheta)$ over $\tilde t^{-1}\in\sX^{\uppie}(0)$ 
the nonlinear wild monodromy operator
\begin{itemize}
	\item[-]  $\tilde{\mbf M}_{0+}^{\uppie}(c;\kappa)$ corresponds to the action of $\tilde g_{0t}^{\circ 2}(\pX;\kappa)$ on $\Cal S_V(\ptheta)$,
	\item[-]  $\tilde{\mbf M}_{\epsilon+}^{\uppie}(c;\kappa)$ corresponds to the action of $\tilde g_{0\infty}^{\circ 2}\circ \tilde g_{t0}^{\circ 2}(\pX;\kappa)$  on $\Cal S_V(\ptheta)$.
\end{itemize}
\end{proposition}

\begin{proof}
In fact, for all $0\neq\epsilon\in\sE_+$ the corresponding action of the nonlinear monodromy of $P_{VI}(\vartheta)$  around $0$, resp. $\epsilon$, from some base-point $x_0\in\sX_+^{\uppie}(0)$ on the character variety  $\Cal S_V(\ptheta)$ is given by $\tilde g_{0t}^{\circ 2}(\pX;e_t^2)$, resp. 
$\tilde g_{0\infty}^{\circ 2}\circ \tilde g_{t0}^{\circ 2}(\pX;e_t^2)$, see Figure~\ref{figure:wm-glueingx} (compare also with the right side of Figure~\ref{figure:wm-monodromy}). In particular, by Propositions~\ref{prop:wm-gt1} and~\ref{prop:wm-tildeg0t} it is rational on all $\Cal S_V(\ptheta)$.
\end{proof}

\begin{lemma}\label{lemma:wm-inftorus}
The vector field $c_1\tfrac{\partial}{\partial c_1}-c_2\tfrac{\partial}{\partial c_2}$ which is Hamiltonian for $h(c)=c_1c_2$ with respect to $dc_1\wedge dc_2$
corresponds to the vector field
\begin{equation} \label{eq:wm-torusfield}
\tfrac{1}{\pXzero}\big(\pF_{\!\pXone}\tfrac{\partial}{\partial\pXinf}-\pF_{\!\pXinf}\tfrac{\partial}{\partial\pXone}\big),
\end{equation}
on the wild character variety $\Cal S_V(\ptheta)$. It is Hamiltonian with respect to $\tilde\omega_{\Cal S_V}$ \eqref{eq:wm-tildeomega} for the Hamiltonian function 
\begin{equation}\label{eq:wm-tildeH}
\tilde H_0(\pX)=\tfrac{1}{2\pi i}\log \pXzero+\tfrac{\vartheta_0}{2},
\end{equation}
which corresponds to $h$.
\end{lemma}

\begin{proof}
By \eqref{eq:wm-inftorus} and Proposition~\ref{prop:wm-wildmonodromy} the vector field $c_1\tfrac{\partial}{\partial c_1}-c_2\tfrac{\partial}{\partial c_2}$ corresponds to
\begin{equation*} 
	\dot\pX=-\big(\kappa\tfrac{\partial}{\partial\kappa} \tilde g_{0t}^{\circ 2}(\,\cdot\,,\kappa)\big)\circ \tilde g_{t0}^{\circ 2}(\pX,\kappa)
\end{equation*}
which can be calculated as \eqref{eq:wm-torusfield} using the formulas of Proposition~\ref{prop:wm-tildeg0t} (note that the $\tfrac{\partial}{\partial\pXzero}$ component is clearly null, and it is enough to calculate the $\tfrac{\partial}{\partial\pXinf}$ component only,
because then the $\tfrac{\partial}{\partial\pXone}$ component is uniquely determined by the relation of tangency $d\pF=0$).

Now calculating in the local coordinate $(\pXone,\pXinf)$ on ${\Cal S}_{V}(\ptheta)$ one has 
$d\pXzero=-\frac{\pF_{\!\pXone}}{\pF_{\!\pXzero}}d\pXone-\frac{\pF_{\!\pXinf}}{\pF_{\!\pXzero}}d\pXinf$, and therefore
$d\tilde H_0=\frac{1}{2\pi i\pF_{\!\pXzero}}\left(-\frac{\pF_{\!\pXone}}{\pXzero}d\pXone-\frac{\pF_{\!\pXinf}}{\pXzero}d\pXinf\right)$, 
so $\tilde H_0$ \eqref{eq:wm-tildeH}
is the Hamiltonian of \eqref{eq:wm-torusfield} with respect to $\tilde\omega_{\Cal S_V}=\frac{d\pXone\wedge d\pXinf}{2\pi i\pF_{\!\pXzero}}$ \eqref{eq:wm-tildeomega}. 
Such Hamiltonian is defined up to a constant, and we know that $h(c)=c_1c_2$ vanishes on the ``sectorial center manifold'' over $\sX_+^{\uppie}(0)$
(Corollaries~\ref{cor:wm-centermanifold0} and~\ref{cor:wm-centermanifold}), which is given by the initial condition $c=0$. 
So our corresponding Hamiltonian $\tilde H_0$ is uniquely determined by the condition that it vanishes on this ``sectorial center manifold'', which by Proposition~\ref{prop:wm-centermanifold} corresponds to the point $(\pXzero,\pXone,\pXinf)=(\tfrac{1}{e_0},\ \tilde e_1+ e_0a_\infty-e_0^2\tilde e_1,\ e_0\tilde e_1)$.
\end{proof}

\begin{remark}
The point $(c_1,c_2)=(0,0)$ is fixed by $c_1\tfrac{\partial}{\partial c_1}-c_2\tfrac{\partial}{\partial c_2}$ and correspondingly the point \eqref{eq:wm-sectorialcm} is fixed by \eqref{eq:wm-torusfield} -- in fact it is easy to verify that both  $\pF_{\!\pXone}$ and $\pF_{\!\pXinf}$ vanish at this point.
\end{remark}

Using the formulas \eqref{eq:wm-stokesops} we can now express also the Stokes operators.

\begin{theorem}[Wild monodromy action on $\Cal S_V(\ptheta)$]~\label{theorem:wm-main} 
	
\noindent	
i) The time-$\alpha$-flow map of the vector field \eqref{eq:wm-torusfield}, $\alpha\in\C$,
is given by
\begin{equation}\label{eq:wm-tealpha}
	\mbf t(\pX; e^{\alpha})=\tilde g_{0t}^{\circ 2}(\,\cdot\,;e^{-\alpha})\circ\tilde g_{t0}^{\circ 2}(\pX;1).
\end{equation}
It factors through the exponential $\alpha\mapsto e^{\alpha}\in\C^*=\C\sminus\{0\}$ as a multiplicative action of $\C^*$:
\begin{equation}\label{eq:wm-tildetorus}
	\begin{aligned}
		\mbf t(\,\cdot\,;e^\alpha):\ \ 
		\pXzero&\mapsto\pXzero,\\
		\pXone&\mapsto \pXone-(1-e^{-\alpha})\tfrac{\pF_{\!\pXinf}}{\pXzero}+
		(2-e^{\alpha}-e^{-\alpha})\tfrac{\pF_{\!\pXone}}{\pXzero^2},\\
		\pXinf&\mapsto \pXinf-(1-e^\alpha)\tfrac{\pF_{\!\pXone}}{\pXzero}.
	\end{aligned}
\end{equation}
The action of the nonlinear \emph{exponential torus} \eqref{eq:wm-Tc} of $P_V$ on $\Cal S_V(\ptheta)$ is given by the composition of $\mbf t(\pX;e^\alpha)$
with any analytic germ $\alpha(\tilde H_0(\pX))$ of $\tilde H_0(\pX)$ \eqref{eq:wm-tildeH}.

\noindent
ii) The action of the nonlinear \emph{formal monodromy} $\mbf N$ \eqref{eq:wm-nlN}  on $\Cal S_V(\ptheta)$ is given by 
\begin{equation}\label{eq:wm-ii}
	\mbf n(\pX):=\mbf t(\pX; \tfrac{\pXzero^4}{\tilde e_1^2}),
\end{equation}
where $\mbf t(\,\cdot\,;\,\cdot\,)$ is as above \eqref{eq:wm-tildetorus}.

\noindent
iii) The action of the nonlinear \emph{Stokes operator} ${\mbf N}^{\circ(-1)}\circ\tilde{\mbf S}_1\circ{\mbf N}$ \eqref{eq:wm-tildeS} on $\Cal S_V(\ptheta)$ is given by 
\begin{equation}\label{eq:wm-iii}
	\mbf n^{\circ-1}\circ\tilde{\mbf s}_1\circ\mbf n(\pX):=\tilde g_{0t}^{\circ 2}(\pX; \pXzero^2),
\end{equation}
and the action of the nonlinear \emph{Stokes operator} $\tilde{\mbf S}_2$  \eqref{eq:wm-tildeS} on $\Cal S_V(\ptheta)$ is given by 
\begin{equation}\label{eq:wm-iiii}
	\tilde{\mbf s}_2(\pX):=\tilde g_{0\infty}^{\circ 2}\circ \tilde g_{t0}^{\circ 2}(\pX; \tfrac{\tilde e_1^2}{\pXzero^{2}}),
\end{equation}
where $\tilde g_{0t}^{\circ 2}$, $\tilde g_{t0}^{\circ 2}$ are as in Proposition~\ref{prop:wm-tildeg0t}, and 
$\tilde g_{0\infty}^{\circ 2}$ given by \eqref{eq:wm-tildeg1t2} is as in Proposition~\ref{prop:wm-gt1}.

\noindent
iv) The  action of the \emph{total monodromy operator} $\tilde{\mbf M}^{\uppie}=\tilde{\mbf S}_2\circ\tilde{\mbf S}_1\circ\mbf N$  \eqref{eq:wm-nlM} on $\Cal S_V(\ptheta)$ is given by 
%$\tilde{\mbf s}_2\circ\tilde{\mbf s}_1\circ\mbf n(\pX)=\tilde g_{0\infty}^{\circ 2}(\pX)$
\begin{equation}\label{eq:wm-tildeg1t2}
\begin{aligned}
\tilde{\mbf s}_2\circ\tilde{\mbf s}_1\circ\mbf n=\tilde g_{0\infty}^{\circ 2}:\ \ 
\pXzero&\mapsto \pXzero-\pF_{\!\pXzero},\\[0pt]
\pXone&\mapsto \pXone,\\[0pt]
\pXinf&\mapsto \pXinf-\pF_{\!\pXinf}+\pXone\pF_{\!\pXzero}.	
\end{aligned}
\end{equation}

\noindent
v) The   action of the \emph{wild monodromy pseudogroup} of $P_V$ on $\Cal S_V(\ptheta)$ is generated by 
\[\big\langle \tilde g_{0\infty}^{\circ 2},\ \tilde g_{0t}^{\circ 2}(\,\cdot\,;e^{-\alpha(\tilde H_0)})\,|\,\alpha(\cdot)\ \text{analytic germ}\big\rangle,\] or equivalently by
\[\big\langle \tilde{\mbf s}_1,\ \tilde{\mbf s}_2,\ \mbf t(\,\cdot\,;e^{\alpha(\tilde H_0)})\,|\,\alpha(\cdot)\ \text{analytic germ}\big\rangle.\]
It fixes the singularities of $\Cal S_{V}(\ptheta)$, and
its restriction to the smooth locus of $\Cal S_{V}(\ptheta)$ represents faithfully the nonlinear action of the wild monodromy pseudogroup on the ``non-Riccati solutions'' of $P_{V}(\tilde\vartheta)$. 
\end{theorem}  

\begin{proof}
i) By direct calculation one verifies easily that \eqref{eq:wm-tildetorus} is indeed the time-$\alpha$-flow map of \eqref{eq:wm-torusfield}.

ii) From \eqref{eq:wm-formalmonodromy}, 
\[\mbf n(\pX)=\mbf t(\pX;e^{2\pi i(-2\vartheta_0-\tilde\vartheta_1+4\tilde H_0)}),\] 
where $H_0(\pX)$ is \eqref{eq:wm-tildeH} meaning that $e^{2\pi i(-2\vartheta_0-\tilde\vartheta_1+4\tilde H_0(\pX))}=\frac{\pXzero^4}{\tilde e_1^2}$.

iii) From \eqref{eq:wm-stokesops} and Proposition~\ref{prop:wm-wildmonodromy} and Lemma~\ref{lemma:wm-inftorus},
\[\mbf n^{\circ-1}\circ\tilde{\mbf s}_1\circ\mbf n(\pX)=\tilde g_{0t}^{\circ 2}(\pX; e^{2\pi i(-\vartheta_0+2\tilde H_0)}),\]
where $e^{2\pi i(-\vartheta_0+2\tilde H_0(\pX))}=\pXzero^2$,
and
\[\tilde{\mbf s}_2(\pX)=\tilde g_{0\infty}^{\circ 2}\circ \tilde g_{t0}^{\circ 2}(\pX; e^{2\pi i(\vartheta_0+\tilde\vartheta_1-2\tilde H_0)}),\] 
where
$e^{2\pi i(\vartheta_0+\tilde\vartheta_1-2\tilde H_0(\pX))}=\tfrac{\tilde e_1^2}{\pXzero^{2}}$.

iv) We know from Proposition~\ref{prop:wm-gt1} that $\tilde{\mbf M}^{\uppie}=\tilde{\mbf M}_\epsilon^{\uppie}\circ \tilde{\mbf M}_0^{\uppie}$
corresponds to the monodromy operator $\tilde g_{0\infty}^{\circ 2}$  around $x=\tilde t^{-1}=0$ 
(cf.~Proposition~\ref{prop:wm-wildmonodromy}) and is independent of $\kappa$.
 But let us calculate the composition $\tilde{\mbf s}_2\circ\tilde{\mbf s}_1\circ\mbf n$ for the sake of clarity.
From \eqref{eq:wm-tealpha}
\[ \tilde g_{t0}^{\circ 2}(\pX;e^{\alpha})=\tilde g_{t0}^{\circ 2}(\,\cdot\,;1)\circ \mbf t(\pX; e^{\alpha}),\qquad
	 \tilde g_{0t}^{\circ 2}(\pX;e^{\alpha})=\mbf t(\,\cdot\,; e^{-\alpha})\circ\tilde g_{t0}^{\circ 2}(\pX;1),\qquad e^\alpha\in\C^*,	\]
and plugging in \eqref{eq:wm-ii}, \eqref{eq:wm-iii} and \eqref{eq:wm-iiii} we have
\begin{align*}
\tilde{\mbf s}_2\circ\tilde{\mbf s}_1\circ\mbf n(\pX)&=	
\tilde g_{0\infty}^{\circ 2}(\pX)\circ \tilde g_{t0}^{\circ 2}(\pX; \tfrac{\tilde e_1^2}{\pXzero^{2}})\circ 
\mbf t(\pX; \tfrac{\pXzero^4}{\tilde e_1^2})\circ \tilde g_{0t}^{\circ 2}(\pX; \pXzero^2)\\
&=\tilde g_{0\infty}^{\circ 2}(\pX) \circ \tilde g_{t0}^{\circ 2}(\pX;1) \circ \mbf t(\pX;\tfrac{\tilde e_1^2}{\pXzero^{2}}) \circ 
\mbf t(\pX; \tfrac{\pXzero^4}{\tilde e_1^2}) \circ \mbf t(\pX; \tfrac{1}{\pXzero^2}) \circ \tilde g_{0t}^{\circ 2}(\pX;1)\\
&=\tilde g_{0\infty}^{\circ 2}(\pX),
\end{align*}
since $\pXzero$ is a first integral of $\mbf t$.
\end{proof}

%
%
%\begin{proposition}\label{prop:sectorial}
%	For $\epsilon=0$, the two sectorial center manifold solutions of Corollary~\ref{cor:wm-centermanifold0} coincide and are analytic on a full neighborhood of $x=0$ if and only if $a_\infty=e_0\tilde e_1+\tfrac{1}{e_0\tilde e_1}$,
%	in which case it is one of the Riccati solutions.
%\end{proposition}
%
%\begin{proof}
%The sectorial ``center manifold'' solution of the Hamiltonian system  of $P_V$ at the point $\tilde t=\infty$ is analytic
%(that is both $q(\tilde t)$ and $p(\tilde t)$ are analytic) if and only if it is
%fixed by the Stokes operators and by the total monodromy.
%Since both $\pF_{\!\pXone}$ and $\pF_{\!\pXinf}$ vanish at the corresponding point \eqref{eq:wm-sectorialcm}, to be fixed by $g_{0\infty}^{\circ 2}$ we need that also $\pF_{\!\pXzero}$ vanishes, i.e. that the point is singularity. By Proposition~\ref{prop:singularitiesV} the only possibility is if either $e_\infty=e_0\tilde e_1$
%or $e_\infty=\tfrac{1}{e_0\tilde e_1}$. 
%\end{proof}

\goodbreak

\section{Appendix: Painlevé equations as isomonodromic deformations of $3\!\times\! 3$ systems}

This section exposes first how to derive the Fricke formula \eqref{eq:wm-F} for the character variety $\Cal S_{VI}(\theta)$ of $P_{VI}(\vartheta)$
as the space of Stokes data corresponding to isomonodromic deformations of $3\!\times\!3$-systems in Okubo and Birkhoff canonical forms, and describes the braid group action on the Stokes data.
Most of this can be also found in a bit different form in the article of Boalch \cite[Sections 2 \& 3]{Bo05}.
This description is then used to study the confluence of eigenvalues in these systems in order to show how the Stokes data of the limit system
for $\epsilon=0$ are connected with those for $\epsilon\neq 0$
(Figure \ref{figure:wm-6}), providing thus another derivation of the wild character $\Cal S_V(\ptheta)$ variety of $P_V(\tilde\vartheta)$ \eqref{eq:wm-tildeF} and of the formulas of the birational change of variables $\Phi_+$ \eqref{eq:wm-UWXXX}.

\subsection{Systems in Okubo and Birkhoff forms}
Aside of the $2\!\times\!2$ systems \eqref{eq:wm-A}, the sixth Painlevé equation $P_{VI}$
governs also the isomonodromic deformations $3\!\times\! 3$ linear differential systems in an \emph{Okubo form}
\begin{equation}\label{eq:wm-okubosyst}
\Big(zI-\left(\begin{smallmatrix} 0 &&\\ &t&\\ &&1\end{smallmatrix}\right)\Big)\frac{d\psi}{d z}=\Big[B(t)+\lambda I\Big]\psi,  
\end{equation}
where the matrix $B(t)$ can be written as
\begin{equation*}
B(t)
=\begin{pmatrix}
\vartheta_0 & w_0u_tv_t-v_t & w_0u_1v_1-v_1 \\
 w_tu_0v_0-v_0 & \vartheta_t & w_tu_1v_1-v_1 \\
 w_1u_0v_0-v_0 & w_1u_tv_t-v_t & \vartheta_1 
\end{pmatrix}, \qquad w_i=\frac{v_i+\vartheta_i}{u_iv_i},
\end{equation*}
where $\vartheta_i$ are the parameters of $P_{VI}$, and
the matrix $B(t)$ has eigenvalues 
\begin{equation}\label{eq:wm-kappae}
	0,\quad -\kappa_1=\tfrac{1}{2}(\vartheta_0+\vartheta_t+\vartheta_1-\vartheta_\infty), \quad -\kappa_2=\tfrac{1}{2}(\vartheta_0+\vartheta_t+\vartheta_1+\vartheta_\infty).
\end{equation}

The isomonodromic deformation of such systems in relation to $P_{VI}$ was first considered in the papers of Harnad \cite{Har94}, Dubrovin \cite{Dub96,Dub99} and Mazzocco \cite{Maz}.
The system \eqref{eq:wm-okubosyst} can be obtained from the $2\!\times\!2$ system \eqref{eq:wm-A} by the addition
of $\tfrac{1}{2}(\frac{\vartheta_0}{z}+\frac{\vartheta_t}{z-t}+\frac{\vartheta_1}{z-1})I$ to $A(z,t)$ 
(this corresponds to a gauge transformation $\phi\mapsto z^{-\frac{\vartheta_0}{2}}(z-t)^{-\frac{\vartheta_t}{2}}(z-1)^{-\frac{\vartheta_1}{2}}\phi$),
 followed by the Katz's operation of middle convolution $mc_\lambda$ with a generic parameter $\lambda$ different from $0,\kappa_1,\kappa_2$ 
\cite{HF} (see also \cite{Maz,Bo05}).
%One may also replace $B(t)$ by the conjugated matrix
%\begin{equation*}
%\left(\begin{smallmatrix}v_0&&\\[3pt]&\!v_t\!&\\[3pt]
%&&v_1\end{smallmatrix}\right)B(t)
%\left(\begin{smallmatrix}v_0&&\\[3pt]&\!v_t\!&\\[3pt]
%&&v_1\end{smallmatrix}\right)^{\!\!-1}
%=\begin{pmatrix}
%\vartheta_0 & \!\!\tfrac{v_0+\vartheta_0}{u_0}u_t-v_0\! & \!\tfrac{v_0+\vartheta_0}{u_0}u_1-v_0 \\[3pt]
%\! \tfrac{v_t+\vartheta_t}{u_t}u_0-v_t\!\! & \vartheta_t & \!\!\tfrac{v_t+\vartheta_t}{u_t}u_0-v_t\!\\[3pt]
%\tfrac{v_1+\vartheta_1}{u_1}u_0-v_1\! & \!\tfrac{v_1+\vartheta_1}{u_1}u_t-v_1 \!\! & \vartheta_1 
%\end{pmatrix}.
%\end{equation*}

Equivalently, one may also consider the generalized  isomonodromic problem for systems in a \emph{Birkhoff canonical form}
\begin{equation}\label{eq:wm-bsyst}
\xi^2\frac{dy}{d\xi}=\Big[\left(\begin{smallmatrix} 0 &&\\ &t&\\ &&1\end{smallmatrix}\right)+ \xi B(t)\Big]y,  
\end{equation}
with a non-resonant irregular singularity at the origin.
These are dual to the Okubo systems \eqref{eq:wm-okubosyst} through the Laplace transform
\begin{equation}\label{eq:wm-laplace}
y(\xi)=\xi^{-1-\lambda}\int_0^\infty \psi(z)\,e^{-\frac{z}{\xi}}d z,\qquad |\arg(\xi) -\arg(z)|<\tfrac{\pi}{2}.
\end{equation}

All three kinds of systems \eqref{eq:wm-A}, \eqref{eq:wm-okubosyst}, \eqref{eq:wm-bsyst}, and their isomonodromy problems are essentially equivalent (at least on the Zariski open set of irreducible systems \eqref{eq:wm-A}). 
Under an additional assumption that no $\vartheta_i$ is an integer, 
the condition on (generalized) isomonodromicity of the each of the above linear systems is
equivalent to the Painlevé equation $P_{VI}(\vartheta)$ \cite{HF}.

\begin{notation}
The entries of $3\!\times\!3$ matrices will be indexed by $(0,t,1)$ rather than $(1,2,3)$, 
in a correspondence to the eigenvalues of the matrix
$\left(\begin{smallmatrix} 0 &&\\ &t&\\ &&1\end{smallmatrix}\right)$.
As before, the triple of indices $(i,j,k)$ will always denote a permutation of $(0,t,1)$, and $(i,j,k,l)$ will denote a permutation of $(0,t,1,\infty)$.
\end{notation}

\subsection{Stokes matrices of the Birkhoff system}
The Birkhoff system \eqref{eq:wm-bsyst} posses a canonical formal solution 
\begin{equation*}
\hat Y(\xi,t)=\hat T(\xi,t)\left(\begin{smallmatrix} \xi^{\vartheta_0} &&\\ & e^{-\frac{t}{\xi}}\xi^{\vartheta_t} &\\ && e^{-\frac{1}{\xi}}\xi^{\vartheta_1}\end{smallmatrix}\right),
\end{equation*}
with $\hat T(\xi,t)$ an  invertible formal series in $\xi$ (with coefficients locally analytic in $t\in\CP^1\sminus\{0,1,\infty\}$), which is unique up to right multiplication by an invertible diagonal matrix,
and unique if one demands that $ \hat T(0,t)=I$ \cite{Sch3}.
It is well known that this series is Borel summable in each non-singular direction $\alpha$ (we remind that a \emph{direction $\alpha\in\R$ is singular} for the system \eqref{eq:wm-bsyst} if $(i-j)\in e^{i\alpha}\R^+$ for some $i,j\in\{0,t,1\}$, $i\neq j$).
This means that for each non-singular direction $\alpha$, there is an associated canonical fundamental matrix solution 
\begin{equation}\label{eq:wm-Y}
Y_\alpha(\xi,t)=T_\alpha(\xi,t)\left(\begin{smallmatrix} \xi^{\vartheta_0} &&\\ & e^{-\frac{t}{\xi}}\xi^{\vartheta_t} &\\ && e^{-\frac{1}{\xi}}\xi^{\vartheta_1}\end{smallmatrix}\right), \qquad |\arg(\xi)-\alpha|<\pi,
\end{equation}
where $T_\alpha$ is the Borel sum in $\xi$ of $\hat T$ in the direction $\alpha$ given by the Laplace integral
\[T_{\alpha}(\xi,t)=\frac{1}{\xi} \int_0^{+\infty e^{i\alpha}} U(z,t)\,e^{-\frac{z}{\xi}}dz,\]
where $U(z,t)=\sum_{k=0}^{+\infty}\frac{T_k(t)}{k!}z^k$ is the formal Borel transform of $\xi\,\hat T(\xi,t)=\sum_{k=0}^{+\infty} T_k(t) \xi^{k+1}$.
This solution does not depend on $\alpha$ as long as $\alpha$ does not cross any singular direction \cite{Ba,IlYa,MR2}.

Let us restrict to $\alpha\in \,]-\!\pi,\pi[$, and suppose for a moment that ${0,t,1}$ are not collinear, i.e. that there are six distinct singular rays  $(i-j)\R^+$
as in Figure~\ref{figure:wm-2a}.
When $\alpha$ crosses some singular direction (in clockwise sense)
the corresponding sectorial basis $Y_\alpha$ changes in a way that corresponds to a multiplication by a constant (with respect to $\xi$) invertible matrix,
called \emph{Stokes matrix}, of the form
\begin{equation}\label{eq:wm-sij}
S_{ij}=I+s_{ij}E_{ij},
\end{equation}
 where $E_{ij}$ denotes the matrix with 1 at the position $(i,j)$ and zero elsewhere.
For the singular ray $(0-1)\R^+$, one needs to take in account also the jump in the argument of $\xi$ between $-\pi$ and $\pi$, therefore the change of basis is provided by a matrix $\bar NS_{01}$, where $\bar N$ is the \emph{formal monodromy} of $\hat Y$:
%$\left(\begin{smallmatrix} \xi^{\vartheta_0} &&\\ & e^{-\frac{t}{\xi}}\xi^{\vartheta_t} &\\ && e^{-\frac{1}{\xi}}\xi^{\vartheta_1}\end{smallmatrix}\right)$:
\begin{equation}\label{eq:wm-barN}
\bar N=\left(\begin{smallmatrix}  e_0^2 &&\\ & e_t^2 & \\ && e_1^2 \end{smallmatrix}\right), \quad \text{where}\quad e_j:=e^{\pi i \vartheta_j}.
\end{equation}
See Figure \ref{figure:wm-2a}.

\begin{figure}[t]
\centering
\begin{subfigure}[t]{0.4\textwidth}
\includegraphics [width=\textwidth]{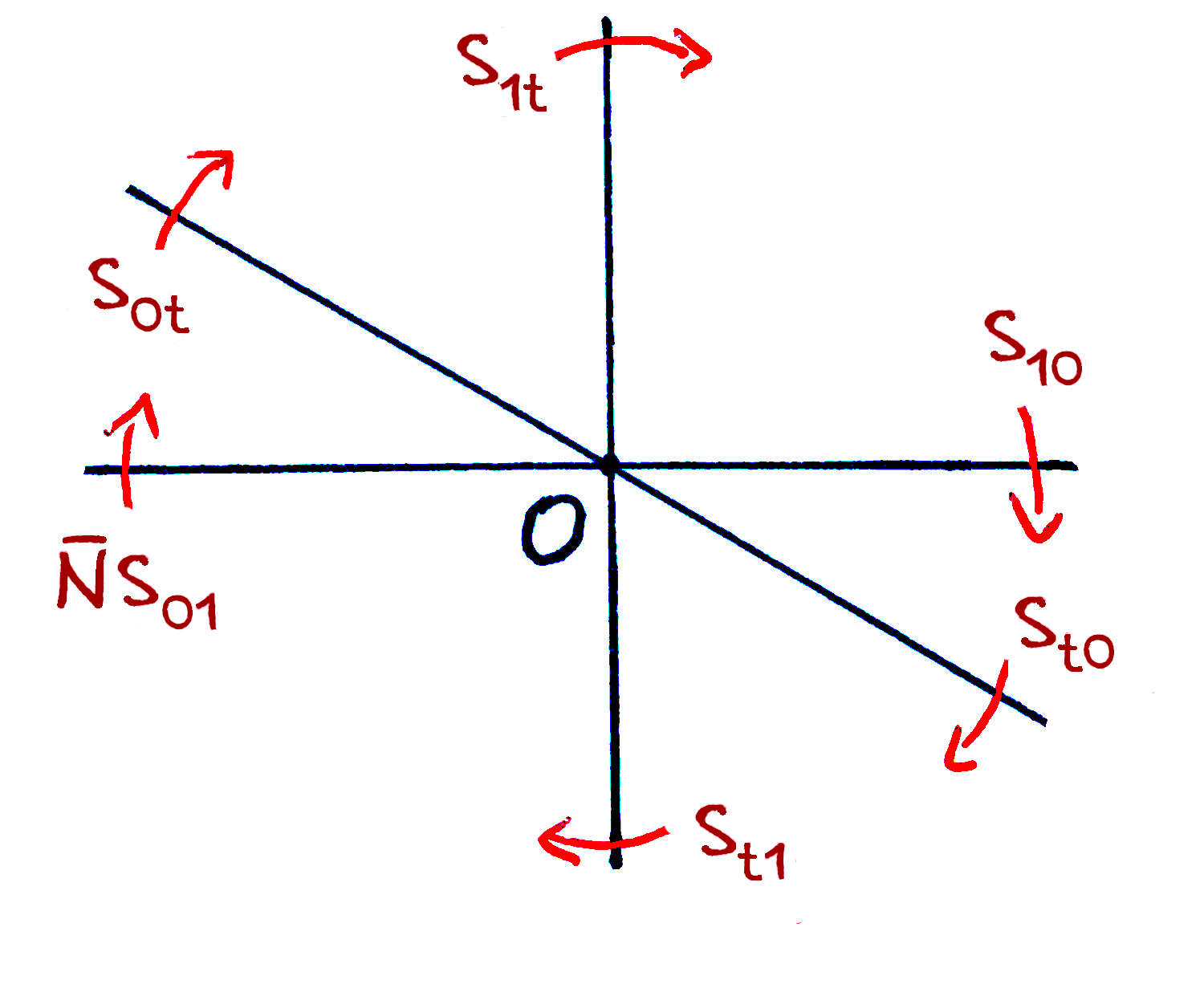}
\caption{Birkhoff system}	\label{figure:wm-2a} 
\end{subfigure}
\hskip0.05\textwidth
\begin{subfigure}[t]{0.4\textwidth}
\includegraphics [width=\textwidth]{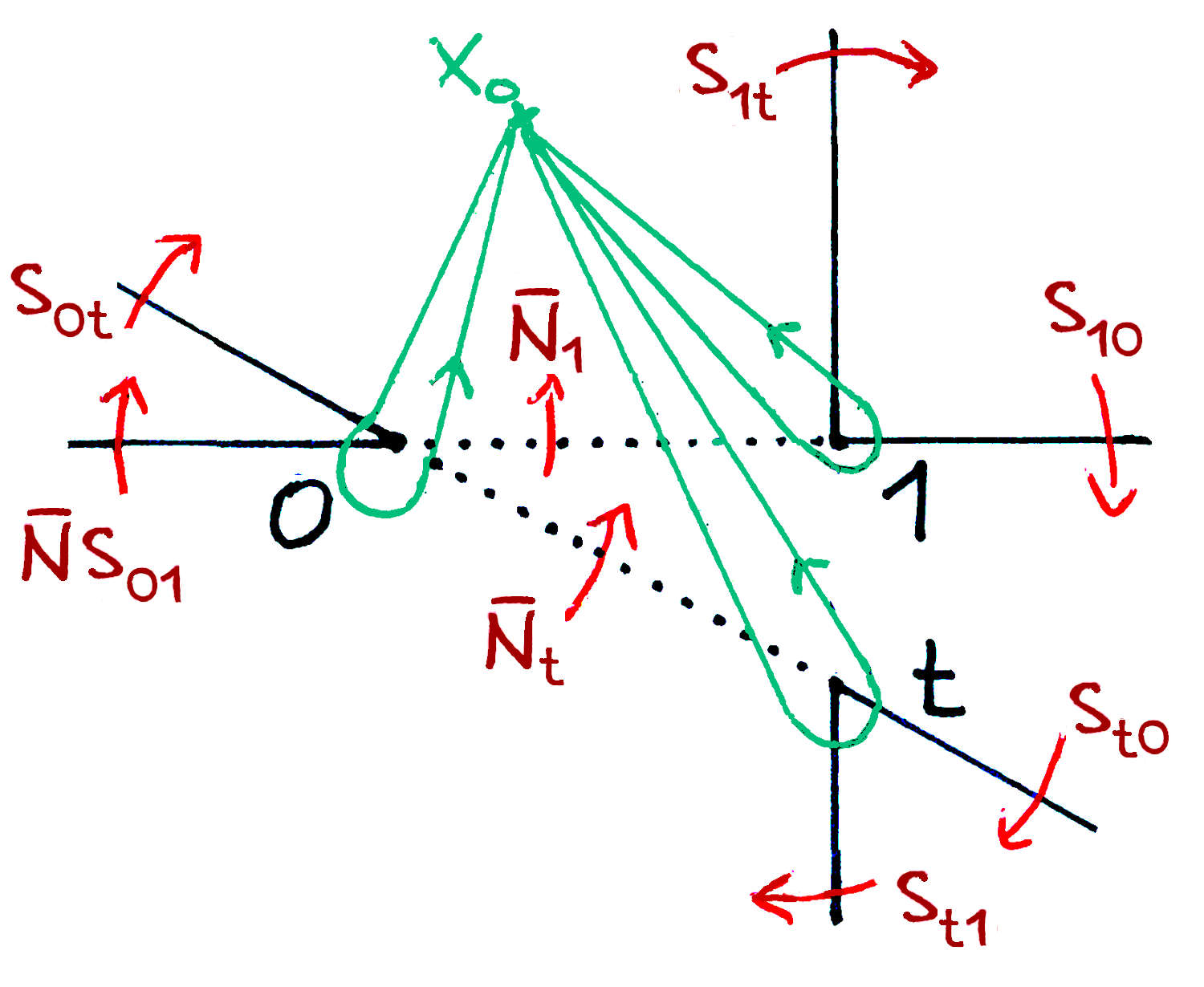}
\caption{Okubo system}	 \label{figure:wm-2b}
\end{subfigure}
\caption{Singular directions and Stokes matrices.}
\label{figure:wm-2}
\end{figure}

Since in general the formal transformation $\hat T(\xi,t)$, and therefore also the collection of the sectorial bases $Y_\alpha(\xi,t)$ \eqref{eq:wm-Y}, are unique only up to right
multiplication by invertible diagonal matrices, the collection of the Stokes matrices $S_{ij}$ is defined only up a simultaneous conjugation by diagonal matrices. 
The obvious invariants with respect to such conjugation are
\begin{equation}\label{eq:wm-ss}
s_{0t}s_{t0}, \quad s_{t1}s_{1t},\quad s_{10}s_{01},\quad s_{0t}s_{t1}s_{10},\quad s_{1t}s_{t0}s_{01},
\end{equation}
where $s_{ij}$ are as in \eqref{eq:wm-sij}, subject to the relation 
\begin{equation}\label{eq:wm-sss}
s_{0t}s_{t0}\cdot s_{t1}s_{1t}\cdot s_{10}s_{01}-s_{0t}s_{t1}s_{10}\cdot s_{1t}s_{t0}s_{01}=0.
\end{equation}
   
The \emph{isomonodromic condition} on the family \eqref{eq:wm-bsyst} demands that the collection of Stokes matrices is independent of $t$ up to conjugacy by diagonal matrices, i.e. that \eqref{eq:wm-ss} are constant.

\subsection{Monodromy of the Okubo system}
Now let us consider the Okubo system \eqref{eq:wm-okubosyst}, 
where we chose for simplicity 
\[\lambda=0.
%\footnote{For this choice of $\lambda=0$ the system  \eqref{eq:wm-okubosyst} is reducible but it does not matter here.}
\]
Corresponding to the canonical sectorial solutions bases $Y_\alpha=\big(Y_{\alpha,ij}\big)_{i,j}$ of \eqref{eq:wm-bsyst}, 
there are canonical sectorial solutions bases $\Psi_\alpha=\big(\Psi_{\alpha,ij}\big)_{i,j}$ of \eqref{eq:wm-okubosyst},
given by the convolution integral
\[\Psi_{\alpha,ij}(z,t)=\frac{1}{\Gamma(\vartheta_j+\lambda)} \int_j^z U_{ij}(\zeta-j,t)(z-\zeta)^{\vartheta_j+\lambda-1}d\zeta,\]
where $U(z,t)=\sum_{k=0}^{+\infty}\frac{T_k(t)}{k!}z^k$ is the formal Borel transform of $\xi\hat T(\xi,t)=\sum_{k=0}^{+\infty} T_k(t) \xi^k$ \cite{Sch1,Kl3}.
They are related to $Y_\alpha(\xi,t)$ by the Laplace transform \eqref{eq:wm-laplace}
\[Y_{\alpha,ij}(\xi,t)=\tfrac{1}{\xi}\int_j^{+\infty e^{i\alpha}} \Psi_{\alpha,ij}(z,t)\,e^{-\frac{z}{\xi}}d z,\qquad
i,j\in\{0,t,1\}.\]
The sectors on which they are defined (see Figure \ref{figure:wm-2b}) are the different components of the complement in $\C$ of
\[\bigcup_{i,j\in\{0,t,1\}} \big(i+(i\!-\!j)\R^+\big)\ \cup\ [0,1] \ \cup\ [0,t].\] %\ \cup\ \big(1+i\R^+\big)
When crossing one of the rays $i+(i-j)\R^+$ in clockwise sense the solution basis $\Psi_{\alpha}(z,t)$ changes by the same Stokes matrix $S_{ij}$ \eqref{eq:wm-sij} as before, except for the ray
$-\R^+$, where it again changes by $\bar NS_{01}$ (cf. \cite{Kl3}).
When crossing the segments $[0,t]$ the basis changes by $\bar N_t$, and on $[0,1]$ by $\bar N_1$,
where 
\[\bar N_i=I+(e_i^2-1)E_{ii},\qquad e_i=e^{\pi i \vartheta_i},\]
%with $e_i$ as in \eqref{eq:wm-e}, 
is the monodromy matrix of 
$\left(\begin{smallmatrix} z^{\vartheta_0} &&\\ & \!(z-t)^{\vartheta_t}\!\! &\\ && \!\!(z-1)^{\vartheta_1}\!\end{smallmatrix}\right)$
around the point $i\in\{0,t,1\}$, and $\bar N=\bar N_1\bar N_t\bar N_0$ \eqref{eq:wm-barN}.
See Figure \ref{figure:wm-2b}.

Fixing a base-point $z_0$ and three simple loops $\gamma_0,\gamma_t,\gamma_1$ in positive direction around the points $0,t,1$ respectively, such that their composition $\gamma_0\gamma_t\gamma_1=\gamma_\infty^{-1}$ gives a simple loop around the infinity as in Figure \ref{figure:wm-2b}, 
let $\bar M_l$, $l=0,t,1,\infty$, be the associated monodromy matrices along $\gamma_l$:
\begin{equation}\label{eq:wm-barM}
\bar M_0=\bar N_0S_{01}S_{0t},\quad \bar M_t=\bar N_1^{-1}S_{t0}S_{t1}\bar N_t\bar N_1,\quad \bar M_1=S_{1t}S_{10}\bar N_1,
\end{equation}
determined up to a simultaneous conjugation in $\GL_3(\C)$.
We have
\[\tr(\bar M_i)=e_i^2+2,\quad i\in\{0,t,1\}.\]
Denoting
\[X_i=\tfrac{\tr(\bar M_j\bar M_k)-1}{e_je_k},
%\qquad \text{for any permutation}\ (i,j,k) \ \text{of}\ (0,t,1),
\]
and $s_{ij}$ as in \eqref{eq:wm-sij}, we have
\begin{equation}\label{eq:wm-XXX}
X_0=\tfrac{e_t^2+e_1^2+e_1^2s_{t1}s_{1t}}{e_te_1},\qquad 
X_t=\tfrac{e_0^2+e_1^2+e_0^2s_{10}s_{01}}{e_0e_1},\qquad 
 =\tfrac{e_0^2+e_t^2+e_0^2s_{0t}s_{t0}}{e_0e_t}.
\end{equation}
The monodromy around all the three points equals 
\begin{align*}
\bar M_\infty^{-1}&=\bar M_1\bar M_t\bar M_0=S_{1t}S_{10}S_{t0}S_{t1}\bar NS_{01}S_{0t}\\[3pt]
&=\mbox{\scriptsize$\begin{pmatrix}
e_0^2 & e_0^2s_{0t} & e_0^2s_{01} \\[3pt]
e_0^2s_{t0} & e_t^2 + e_0^2s_{t0}s_{0t} & e_1^2s_{t1}+e_0^2s_{t0}s_{01} \\[3pt]
e_0^2s_{10}+e_0^2s_{1t}s_{t0} & e_t^2s_{1t}+e_0^2s_{10}s_{0t}+e_0^2s_{1t}s_{t0}s_{0t} & e_1^2+ e_1^2s_{1t}s_{t1}+e_0^2s_{10}s_{01}+e_0^2s_{1t}s_{t0}s_{01}
\end{pmatrix}$}
\end{align*}
From \eqref{eq:wm-kappae}, we know that its eigenvalues are $1$, $e^{-2\pi i\kappa_1}=\frac{e_0e_te_1}{e_\infty}$ and $e^{-2\pi i\kappa_2}={e_0e_te_1e_\infty}$. %, where $e_\infty:=e^{\pi i \vartheta_\infty}$.
Expressing the coefficients of the linear term $E$ and the quadratic term $E'$ of the characteristic polynomial of $\bar M_\infty^{-1}$ leads to
\begin{equation}\label{eq:wm-s1t0}
e_te_1X_0+e_0e_1X_t+e_0e_tX_1 +e_0^2s_{1t}s_{t0}s_{01}-e_0^2-e_t^2-e_1^2= 1+\tfrac{e_0e_te_1}{e_\infty}+e_0e_te_1e_\infty:=E, 
\end{equation}
\begin{equation}\label{eq:wm-s0t1}
\begin{split} %\mbox{\scriptsize$
e_0^2e_te_1X_0+e_t^2e_0e_1X_t+e_1^2e_0e_tX_1 -e_0^2e_1^2&s_{0t}s_{t1}s_{10}-e_0^2e_t^2-e_t^2e_1^2-e_1^2e_0^2=\\ &=\tfrac{e_0e_te_1}{e_\infty}+e_0e_te_1e_\infty+e_0^2e_t^2e_1^2:=E'.
%$}
\end{split}
\end{equation}
Inserting the expression for $s_{1t}s_{t0}s_{01}$ \eqref{eq:wm-s1t0} and for $s_{0t}s_{t1}s_{10}$ \eqref{eq:wm-s0t1} into the relation \eqref{eq:wm-sss}
gives the \emph{Fricke relation} \eqref{eq:wm-F}
\begin{equation*}
 X_0X_tX_1+X_0^2+X_t^2+X_1^2-\theta_0X_0-\theta_tX_t-\theta_1X_1+\theta_\infty=0,
\end{equation*}
with
\begin{align*}
\theta_i&=\tfrac{e_j^2+e_k^2+E}{e_je_k}+\tfrac{E'}{e_i^2e_je_k}=a_ia_\infty+a_ja_k,
%\qquad\text{for any permutation $(i,j,k)$ of $(0,t,1)$,} %\quad \{i,j,k\}=\{0,t,1\},
\\
\theta_\infty&=
1+\tfrac{E}{e_0}+\tfrac{E}{e_t}+\tfrac{E}{e_1}+\tfrac{E'}{e_0e_t}+\tfrac{E'}{e_te_1}+\tfrac{E'}{e_1e_0}=a_0a_ta_1a_\infty+a_0^2+a_t^2+a_1^2+a_\infty^2-4.
%\tfrac{e_j^2+e_k^2+1}{e_je_k} +\frac{e_i}{e_\infty}+e_ie_\infty-\tfrac{1}{e_ie_\infty}-\tfrac{e_\infty}{e_i}-e_je_k \qquad 
\end{align*}
%with $a_l=e_l+\tfrac{1}{e_l}$.
%\begin{align*}
%\theta_\infty&=1+\big(1+\tfrac{e_0e_te_1}{e_\infty}+e_0e_te_1e_\infty\big)\big(\tfrac{1}{e_0}+\tfrac{1}{e_t}+\tfrac{1}{e_1}\big)+
%\big(\tfrac{e_0e_te_1}{e_\infty}+e_0e_te_1e_\infty+e_0^2e_t^2e_1^2\big)\big(\tfrac{1}{e_0e_t}+\tfrac{1}{e_te_1}+\tfrac{1}{e_1e_0}\big),\\
%&=a_0a_ta_1a_\infty+a_0^2+a_t^2+a_1^2+a_\infty^2-4.
%\end{align*}

Let us remark that for $(i,j,k)$ a cyclic permutation of $(0,t,1)$, the line 
\[\{X_k=\tfrac{e_i}{e_j}+\tfrac{e_j}{e_i},\ \ e_iX_i+e_jX_j=a_\infty+e_ie_ja_k\}\]
of Proposition \ref{prop:wm-lines} corresponds to $s_{ij}=0$ in \eqref{eq:wm-s0t1}, while the line
\[\{X_k=\tfrac{e_i}{e_j}+\tfrac{e_j}{e_i},\ \ e_iX_j+e_jX_i=a_k+e_ie_ja_\infty\}\] 
corresponds to $s_{ji}=0$ in \eqref{eq:wm-s1t0}.\label{page:wm-line}

\medskip

We will now derive the induced action of the braids  $\beta_{0t}$ and $\beta_{t1}$ (Figure~\ref{figure:wm-braidaction}) on the Stokes matrices $S_{ij}$, providing an 
alternative proof of Proposition~\ref{prop:wm-modular}.
The induced action of $\beta_{0t}$, resp. $\beta_{t1}$, on the Stokes matrices is obtained by:
\begin{itemize}
\item[1)] Tracing the connection matrices of the Okubo system \eqref{eq:wm-okubosyst} as the two corresponding points turn around each other according to the braid $\beta_{0t}$, resp. $\beta_{t1}$, and see how they change when the three points $0,t,1$ align. 
See Figure \ref{figure:wm-braidaction}. 
We use the fact that $S_{ij}S_{kl}=S_{kl}S_{ij}$ if $j\neq k$ and $l\neq i$.
\item[2)] Swapping the names of the points $0\leftrightarrow t$, resp. $t\leftrightarrow 1$. This permutes also the corresponding positions of all the matrices,
i.e. acts on them by conjugation by $P_{0t}=\left(\begin{smallmatrix} 0 & 1 & 0 \\ 1 & 0 & 0 \\ 0 & 0& 1 \end{smallmatrix}\right)$, resp. 
$P_{t1}=\left(\begin{smallmatrix} 1 & 0 & 0 \\ 0 & 0 & 1 \\ 0 & 1& 0 \end{smallmatrix}\right)$.
\end{itemize}
\begin{figure}[t]
\centering
\begin{subfigure}[t]{0.45\textwidth} 
\includegraphics [width=\textwidth]{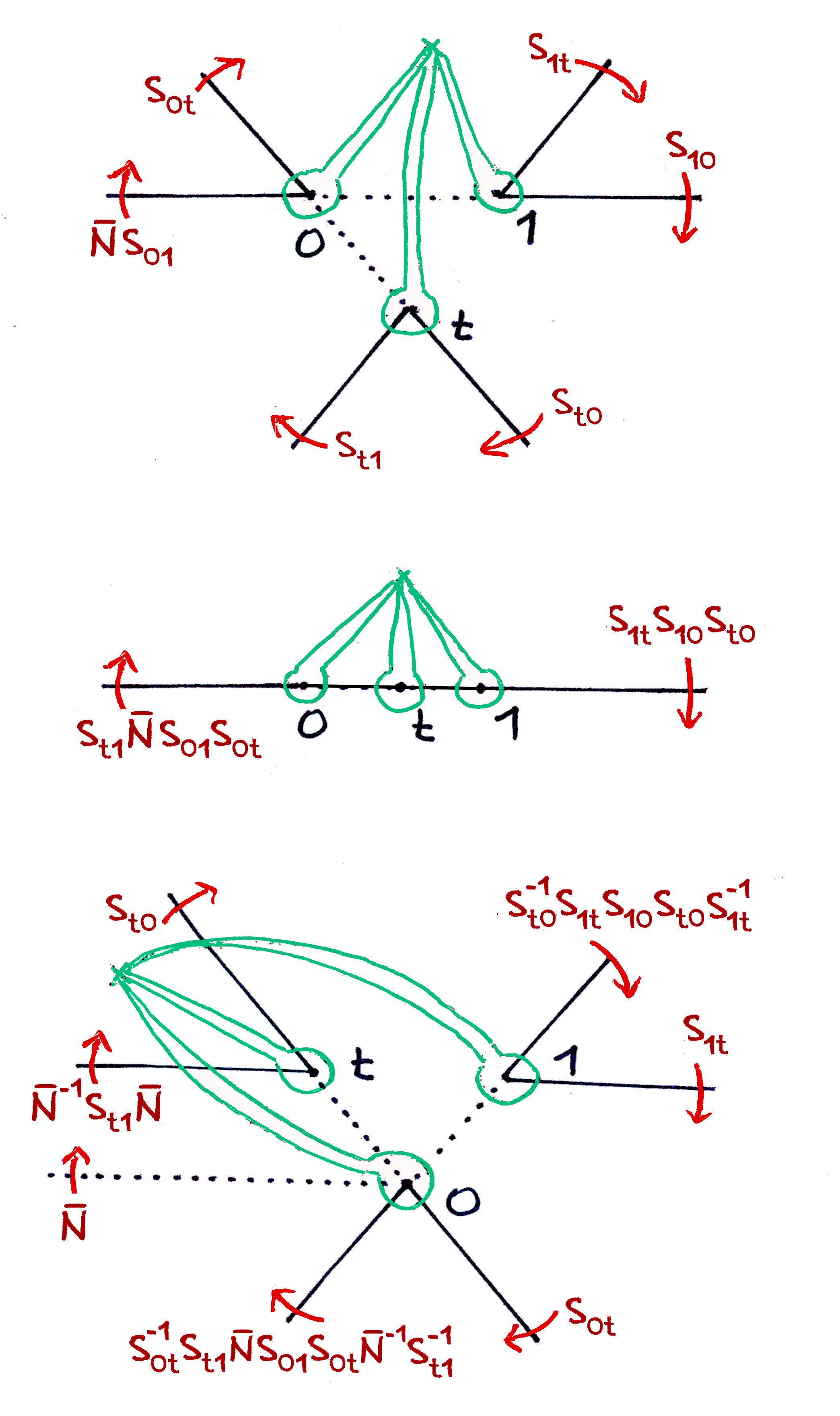}
\caption{Action of $\beta_{0t}$.}	
\end{subfigure}
\hskip0.05\textwidth
\begin{subfigure}[t]{0.45\textwidth} 
\includegraphics [width=\textwidth]{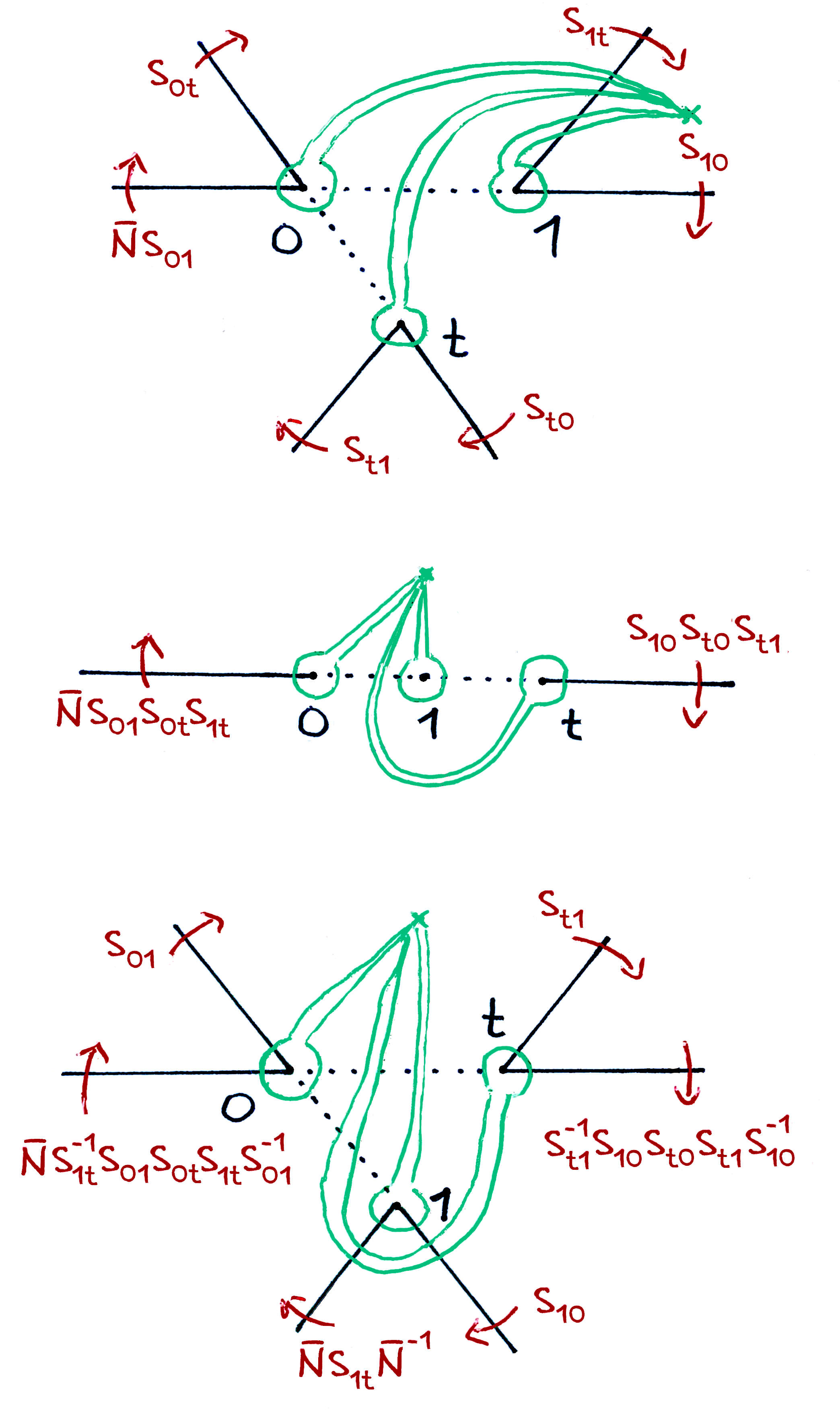}
\caption{Action of $\beta_{t1}$.}	
\end{subfigure}
\caption{Braid actions on the Stokes matrices. The monodromy along each loop stays the same as the points $0,t,1$ move.}
\label{figure:wm-braidaction}
\end{figure}
It follows from Figure~\ref{figure:wm-braidaction} that the action on the Stokes matrices is given (up to a simultaneous conjugation by diagonal matrices) by
\begin{equation*}
\begin{aligned}
(\beta_{0t})_*:\
\bar N &\mapsto P_{0t}\bar NP_{0t}, \\
S_{0t} &\mapsto P_{0t}S_{t0}P_{0t},\\
S_{1t} &\mapsto P_{0t}S_{t0}^{-1}S_{1t}S_{10}S_{t0}S_{1t}^{-1}P_{0t},\\
S_{10} &\mapsto P_{0t}S_{1t}P_{0t},\\
S_{t0} &\mapsto P_{0t}S_{0t}P_{0t},\\
S_{t1} &\mapsto P_{0t}S_{0t}^{-1}S_{t1}\bar NS_{01}S_{0t}\bar N^{-1}S_{t1}^{-1}P_{0t},\qquad\\
S_{01} &\mapsto P_{0t}\bar N^{-1}S_{t1}\bar NP_{0t},\\
\end{aligned}
\begin{aligned}
(\beta_{t1})_*:\ 
\bar N &\mapsto P_{t1}\bar NP_{t1}, \\
S_{0t} &\mapsto P_{t1}S_{01}P_{t1},\\
S_{1t} &\mapsto P_{t1}S_{t1}P_{t1},\\
S_{10} &\mapsto P_{t1}S_{t1}^{-1}S_{10}S_{t0}S_{t1}S_{10}^{-1}P_{t1},\\
S_{t0} &\mapsto P_{t1}S_{10}P_{t1},\\
S_{t1} &\mapsto P_{t1}\bar NS_{1t}\bar N^{-1}P_{t1},\\
S_{01} &\mapsto P_{t1}S_{1t}^{-1}S_{01}S_{0t}S_{1t}S_{01}^{-1}P_{t1}.\\
\end{aligned}
\end{equation*}
From this the corresponding action of $g_{0t}$, resp. $g_{t1}$, on the invariant elements \eqref{eq:wm-ss} can be easily expressed, and subsequently re-expressed in terms of the coordinates $X_0,X_t,X_1$ \eqref{eq:wm-XXX}.
Or equivalently, one can express the induced action on the monodromy matrices \eqref{eq:wm-barM}:
\begin{equation*}
\begin{aligned}
(\beta_{0t})_*:
\bar M_0 &\mapsto P_{0t}S_{t0}^{-1}\bar M_t S_{t0}P_{0t}, \\
\bar M_t &\mapsto P_{0t}S_{t0}^{-1}\bar M_t\bar M_0\bar M_t^{-1}S_{t0}P_{0t},\\
\bar M_1 &\mapsto P_{0t}S_{t0}^{-1}\bar M_1 S_{t0}P_{0t}, \\
\end{aligned}
\begin{aligned}
(\beta_{t1})_*:
\bar M_0 &\mapsto P_{t1}S_{1t}^{-1}\bar M_0 S_{1t}P_{t1}, \\
\bar M_t &\mapsto P_{t1}S_{1t}^{-1}\bar M_1 S_{1t}P_{t1}, \\
\bar M_1 &\mapsto P_{t1}S_{1t}^{-1}\bar M_1\bar M_t\bar M_1^{-1} S_{1t}P_{t1}. \\
\end{aligned}
\end{equation*}

\subsection{Confluence of the Birkhoff systems and their character varieties}

The substitution \eqref{eq:wm-tt}, \eqref{eq:wm-tildev1}
in the Birkhoff system \eqref{eq:wm-bsyst} and a conjugation by 
$Q=\left(\begin{smallmatrix} \epsilon\tilde t & & \\ & 1 & -1 \\ & & \epsilon\tilde t \end{smallmatrix}\right)$, corresponding to the change of variable $\tilde y= Q y$,
gives a parametric family of isomonodromic systems
\begin{equation}\label{eq:wm-bsysttilde}
\xi^2\frac{d\tilde y}{d\xi}=\Big[\left(\begin{smallmatrix} 0 &&\\ &1+\epsilon\tilde t& 1\\ && 1\end{smallmatrix}\right)+ \xi\tilde B(\tilde t,\epsilon)\Big]\tilde y,  
\end{equation}
with
\begin{align*}
\tilde B=QBQ^{-1}&=
\mbox{\scriptsize$\begin{pmatrix}
\vartheta_0 & \epsilon\tilde t(w_0u_tv_t-v_t) & -w_0u_0v_0-v_1-v_t \\[3pt]
\tfrac{1}{\epsilon\tilde t} (w_t-w_1)u_0v_0  & \vartheta_t+v_t-w_1u_tv_t  &  -\tfrac{1}{\epsilon\tilde t}(w_t-w_1)u_0v_0 \\[3pt]
 w_1u_0v_0-v_0 & \epsilon\tilde t(w_1u_tv_t-v_t) & \vartheta_1-v_t +w_1u_tv_t
\end{pmatrix}$}\\[6pt] 
&=\mbox{\scriptsize$\begin{pmatrix}
\vartheta_0 & \tilde u_1\tilde v_1 w_0-\tilde v_1 & -\kappa_2-\vartheta_0 \\[4pt]
-\tilde b_{t1} & \tilde\vartheta_1+\kappa_2+\tilde b_{10} & \tilde b_{t1} \\[4pt]
 \tilde b_{10} &\tilde t-\epsilon\tilde t[\tilde\vartheta_1+\kappa_2+\tilde b_{10}] & -\kappa_2-\tilde b_{10}
\end{pmatrix}$},
\end{align*} 
where $\kappa_2=-\tfrac{\vartheta_0+\tilde\vartheta_1+\vartheta_\infty}{2}$ and $\tilde u_1,\tilde v_1$ are as in \eqref{eq:wm-tildev1},
and
\begin{equation*}
\begin{scriptstyle}
\tilde b_{10}=u_0v_0\tfrac{\tilde v_1+\tilde t+\epsilon\tilde t(v_0-\kappa_2-\tilde\vartheta_1)}{\tilde u_1\tilde v_1+\epsilon\tilde tu_0v_0}-v_0,\qquad
\tilde b_{t1}=u_0v_0 \tfrac{v_0-\tilde\vartheta_1-\kappa_2-\tfrac{u_0v_0}{\tilde u_1\tilde v_1}(\tilde t+\tilde\vartheta_1)}
	{\tilde u_1\tilde v_1+\epsilon\tilde t u_0v_0}. 
\end{scriptstyle}
\end{equation*}

When $\epsilon\neq 0$ the irregular singular point at the origin  is non-resonant and the local description of the Stokes phenomenon is the same as in the precedent section with the six Stokes matrices $S_{ij}$ \eqref{eq:wm-sij}.
But for  $\epsilon=0$ the singularity becomes resonant and the description changes.

For $|\epsilon t|$ small, there is a formal transformation 
\begin{equation}\label{eq:wm-hatT}
\tilde y=\hat T(\xi,\epsilon)\left(\begin{smallmatrix}\tilde y' \\ \tilde y''\end{smallmatrix}\right),\quad (\tilde y',\tilde y'')\in\C\!\times\!\C^2,
\end{equation}
written as a formal power series in $\xi$ with coefficients analytic in $\epsilon$,   
that splits the system in two diagonal blocks, one corresponding to the eigenvalue $0$, other corresponding to the other eigenvalues 
$\{1+\epsilon\tilde t,1\}$ (cf. \cite{Ba}):
\begin{align}\label{eq:wm-bsystV0}
\xi^2\frac{d\tilde y'}{d\xi}&=\xi\vartheta_0 \,\tilde y',\\
\xi^2\frac{d\tilde y''}{d\xi}&=\left[ 
\left(\begin{smallmatrix}  1+\epsilon\tilde t& 1\\[3pt] 0& 1\end{smallmatrix}\right) +\xi\tilde B''(\epsilon) +O(\xi^2)
\right]\tilde y'',
\label{eq:wm-bsystV1}
\end{align}
where $\tilde B''=\left(\begin{smallmatrix} \tilde b_{tt} & \tilde b_{t1}\\[0pt] \tilde b_{1t} & \tilde b_{11} \end{smallmatrix}\right)$ 
is the submatrix of $\tilde B=\big(b_{ij}\big)$.
This formal transformation $\hat T(\xi,\epsilon)$ is Borel summable in $\xi$ all directions except of the coalescing singular directions $\pm\R^+$ and $\pm(1\!+\!\epsilon\tilde t)\R^+$.
Therefore it possesses Borel sums $T_\alpha(\xi,\epsilon)$ on four sectors:
two large sectors which persist to the limit $\epsilon\to 0$, and two small ones that disappear whenever $\epsilon\tilde t\in\R$.
Only the Borel sums on the large sectors will be considered. 

\medskip

\paragraph{Confluence of eigenvalues in the subsystem \eqref{eq:wm-bsystV1}.}
The phenomenon of confluence of eigenvalues in $2\!\times\!2$ parametric systems at an irregular singular point of Poincaré rank 1 was studied previously by the author \cite{Kl1}.
This paragraph applies some of the results to the system \eqref{eq:wm-bsystV1}.

The matrix of the right side of the system has its eigenvalues equal to
\[\lambda^{(0)}+\xi\lambda^{(1)}\pm \sqrt{\alpha^{(0)}+\xi\alpha^{(1)}}\quad (\text{mod }\xi^2),\] 
where
\begin{align*}
\lambda^{(0)}&=1+\tfrac{\epsilon\tilde t}{2},\qquad &
\lambda^{(1)}&=\tfrac{\tilde b_{tt}+\tilde b_{11}}{2}=\tfrac{\tilde\vartheta_1}{2},\\
\alpha^{(0)}&=\left(\tfrac{\epsilon\tilde t}{2}\right)^2,\qquad &
\alpha^{(1)}&=\tilde b_{1t}+\tfrac{\epsilon\tilde t(\tilde b_{tt}-\tilde b_{11})}{2}=\epsilon\tilde t\tfrac{\vartheta_t-\vartheta_1}{2}=
\tilde t-\epsilon\tilde t\tfrac{\tilde\vartheta_1}{2},
\end{align*}
constitute the formal invariants of the system.
In \cite{Kl1}, it has been shown that \eqref{eq:wm-bsystV1} possess a fundamental matrix solutions of the form 
\begin{equation*}
\tilde Y_\bullet''= R_\bullet''(\xi,t,\epsilon)\cdot e^{-\tfrac{\lambda^{(0)}}{\xi}}\xi^{\lambda^{(1)}}
 \mbox{\scriptsize$
 \begin{pmatrix} \big({\alpha^{(0)}+\xi\alpha^{(1)}}\big)^{-\tfrac{1}{4}}\hskip-12pt & \\[0pt] &  \big({\alpha^{(0)}+\xi\alpha^{(1)}}\big)^{\tfrac{1}{4}}  \end{pmatrix}
 $}
 \begin{pmatrix} 1 & 1 \\ 1 & \!\!\!\!-1 \end{pmatrix}
 \begin{pmatrix} e^\Theta\!\!\!\! & \\ & \!e^{-\Theta}\! \end{pmatrix},
\end{equation*} 
where 
%$\,s=\sqrt{\mu+\tfrac{\xi}{\alpha^{(1)}}},\ $ with 
\[\Theta(\xi,\epsilon,\tilde t)=\!\int_\infty^\xi\!\! \tfrac{\sqrt{\alpha^{(0)}+\zeta\alpha^{(1)}}}{\zeta^2}d\zeta =
%\int \tfrac{2s^2}{(s^2-\mu)^2}ds=
\left\{  \hskip-6pt 
\begin{array}{l l}
     %-\frac{s}{s^2-\mu}-\frac{1}{2\!\sqrt\mu}\log\frac{s+\!\sqrt\mu}{s-\!\sqrt\mu}
-\frac{\sqrt{\alpha^{(0)}+\xi\alpha^{(1)}}}{\xi}-
\frac{\alpha^{(1)}}{2\sqrt{\alpha^{(0)}}}\log\frac{\sqrt{\alpha^{(0)}+\xi\alpha^{(1)}}+\sqrt{\alpha^{(0)}}}{\sqrt{\alpha^{(0)}+\xi\alpha^{(1)}}-\sqrt{\alpha^{(0)}}},   &  \epsilon\neq 0,\\[6pt]
-\frac{2\alpha^{(0)}}{\xi}, & \epsilon=0,\\
\end{array}\right.
\]
and $R_\bullet''$, $\bullet=\sI_\pm,\sO$, are invertible analytic transformations defined on  certain domains in the $\xi$-space. 
These domains are delimited by the so called Stokes curves (in the sense of exact WKB analysis \cite{KawTak}): 
the separatrix curves of the foliation by real-time trajectories of the vector field%\,\footnote{This vector field is defined only on the Riemann surface of $\sqrt{\alpha^{(0)}+\xi\alpha^{(1)}}$ but the foliation by its real trajectories projects well onto the $\xi$-space.} 
\[e^{i\omega}\tfrac{\xi^2}{2\sqrt{\alpha^{(0)}+\xi\alpha^{(1)}}}\tfrac{\partial}{\partial\xi},\qquad\text{with some}\quad  \omega\in]\!-\!\tfrac{\pi}{2}+\eta,\tfrac{\pi}{2}-\eta[,\ \eta>0,\]
emanating  either from $\infty$ or from the ``turning point'' at $\xi=-\frac{\alpha^{(0)}}{\alpha^{(1)}}$, if $\epsilon\neq 0$. 
\begin{figure}[t]
	\centering
	\begin{tikzpicture} [scale=0.6] 
		\filldraw[fill=black!10] (0,0) circle (5);
		\filldraw[fill=white] (0,0) arc(180:360:1.5) arc(0:120:2) arc(60:180:2) arc(180:360:1.5);
		\draw[->] (0,{sqrt(3)}) -- (0,-5) node[below] {$\tfrac{\alpha^{(0)}}{\alpha^{(1)}}\mathbb R^+$};
		\draw[dashed, ->] (0,0) -- (-5,0) node[left] {$-\mathbb R^+$};
		\draw (0.5,4) node {$\sO$};
		\draw (2,.5) node {$\sI_+$};
		\filldraw[color=white] (-2.6,0) circle (0.35);
		\draw (-2.5,-.1) node {$\sI_-$};
		\filldraw (0,0) circle (3pt) node[right] {$0$};
		\draw (0,{sqrt(3)}) node[above] {$-\tfrac{\alpha^{(0)}}{\alpha^{(1)}}$};
		\draw[color=red, very thick, ->] (-4,-.5) node[below] {$e_te_1I$} -- (-4,.5); 
		\draw[color=red, very thick, ->] (-1.8,-.3) node[below] {$N''$} (-1.8,-.4) -- (-1.8,.6); 
		\draw[color=red, very thick, ->] (-.5,.7) node[left] {$S_{1t}''\!\!$} -- (.5,.7); 
		\draw[color=red, very thick, ->] (.4,-.5) node[right] {$\!\! S_{t1}''$} -- (-.6,-.5); 
		\draw[color=red, very thick, ->] (.5,-4) node[right] {$\!\!\tilde S_{t}''$} -- (-.5,-4); 
		\draw[color=red, very thick, ->] (-130:2.5) node[below] {$C_-''$} -- +(70:1); 
		\draw[color=red, very thick, ->] (130:3) node[above] {$\tfrac{1}{e_te_1}C_-''N''$} -- +(-60:1); 
		\draw[color=red, very thick, ->] (50:3) node[above] {$C_+''$} -- +(-120:1); 
		\draw[color=red, very thick, ->] (-40:2.8) node[below] {$C_+''$} -- +(110:1); 
	\end{tikzpicture}
	\caption{The inner domains $\sI_\pm$ (white) and the outer domain $\sO$ (grey), and the connection matrices of the system \eqref{eq:wm-bsystV1}.
	Note that for $\epsilon\tilde t\neq$ the point $-\frac{\alpha^{(0)}}{\alpha^{(1)}}\neq 0$ is not a singularity of \eqref{eq:wm-bsystV1} thus solutions are regular there.}
	\label{figure:wm-5}
\end{figure}
There are two kinds of such sectorial domains (see Figure~\ref{figure:wm-5}), whose shape in the coordinate $\frac{\xi}{\alpha^{(1)}}$ depends only on a parameter
$\ \mu=\frac{\alpha^{(0)}}{(\alpha^{(1)})^2}=\big(\tfrac{\epsilon}{2-\epsilon\tilde\vartheta_1}\big)^2$:
\begin{itemize}\leftskip-12pt
\item[-] A pair of \emph{inner domains} $\sI_\pm$ for $\epsilon\neq 0$: these are sectors at 0 of radius proportionate to $\mu\sim\frac{\epsilon^2}{4}$, 
separated one from another by the singular directions $\pm\epsilon t\R^+$. They disappear at the limit. 
The connection matrices between $\tilde Y''_{\sI_+}$ and $\tilde Y''_{\sI_-}$ are given by
$S_{1t}''=\left(\begin{smallmatrix} 1 & 0 \\[3pt]  s_{1t} & 1 \end{smallmatrix}\right)$,
$S_{t1}''=\left(\begin{smallmatrix} 1 &  s_{t1} \\[3pt] 0 & 1 \end{smallmatrix}\right)$,
the submatrices of the Stokes matrices $S_{1t}$, $S_{t1}$ \eqref{eq:wm-sij},
and needs to take into account also the formal monodromy $N''=\left(\begin{smallmatrix} e_t^2 &  \\[3pt]  & e_1^2 \end{smallmatrix}\right)$,
the submatrix of $\bar{N}$ \eqref{eq:wm-barN}. See Figure \ref{figure:wm-5}.

\item[-] An \emph{outer domain} $\sO$ covering a complement of $\sI_+\cup \sI_-$ in a disc of a fixed radius with a cut in the direction $\frac{\alpha{(0)}}{\alpha^{(1)}}\R^+\sim \epsilon^2t\R^+$. We are mainly interested in the limit when $\epsilon\to 0$ along the sequences  
$\frac{1}{\epsilon}\in\frac{1}{\epsilon_0}\pm 2\N$, hence the cut can be assumed to be in the direction $t\R^+$. 
The connection matrix on this cut is 
$\tilde S_{t}''=\left(\begin{smallmatrix} X_0 & -i \\[3pt] -i & 0 \end{smallmatrix}\right)$
with $X_0$ \eqref{eq:wm-XXX} the trace of monodromy of the subsystem \eqref{eq:wm-bsystV1} around $0$. See Figure \ref{figure:wm-5}.

\item[-] The connection matrices between the outer and the inner solution bases can be expressed as
\begin{equation*}
	C_+''=\left(\begin{smallmatrix}  1 & \frac{1}{s_{1t}} \\[1pt]  0 & -i\frac{e_t}{e_1s_{1t}} \end{smallmatrix}\right),\qquad
	C_-''=\left(\begin{smallmatrix}  0 & \frac{e_t}{e_1s_{1t}} \\[1pt] i & -i\frac{e_t^2}{e_1^2s_{1t}} \end{smallmatrix}\right).
\end{equation*}
\end{itemize}

\medskip\goodbreak

Returning now to the full system \eqref{eq:wm-bsysttilde}, one must intersect the domains $\sI_\pm, \sO$ with the sectors of the Borel summability
of the transformation $\hat T$ \eqref{eq:wm-hatT}. The full picture is therefore that of Figure  \ref{figure:wm-6}.
There are four inner Stokes matrices $S_{1t}$, $S_{t1}$ and $S_{10}S_{t0}$, $S_{01}S_{0t}$ \eqref{eq:wm-sij} between the canonical solutions on the inner domains,
and three outer Stokes matrices $\tilde S_0$, $\tilde S_t$, $\tilde S_1$ between the canonical solutions on the outer domains of the form:
\begin{equation}\label{eq:wm-tildeStokes}
\tilde S_0=\left(\begin{smallmatrix} 1 & \tilde s_{0t} & \tilde s_{01} \\[3pt]  & 1 &  \\[3pt] &  & 1 \end{smallmatrix}\right),\quad
  \tilde S_t=\left(\begin{smallmatrix} 1 &  &  \\[3pt]  & X_0 & -i \\[3pt]  & -i & 0 \end{smallmatrix}\right),\qquad
  \tilde S_1=\left(\begin{smallmatrix} 1 &  &  \\[3pt]  \tilde s_{t0} & 1 &  \\[3pt] \tilde s_{10} &  & 1 \end{smallmatrix}\right),\quad
  \tilde N=\left(\begin{smallmatrix} e_0^2 &  &  \\[3pt]  & e_te_1 &  \\[3pt]  &  & e_te_1 \end{smallmatrix}\right).
\end{equation}
The connection matrices between the canonical bases on the inner and outer domains are provided by:
\begin{equation}\label{eq:wm-C}
C_+=\left(\begin{smallmatrix} 1 &  &  \\[3pt]  & 1 & \frac{1}{s_{1t}} \\[1pt]  & 0 & -i\frac{e_t}{e_1s_{1t}} \end{smallmatrix}\right),\quad
 C_-=\left(\begin{smallmatrix} 1 &  &  \\[3pt]  & 0 & \frac{e_t}{e_1s_{1t}} \\[1pt]  & i & -i\frac{e_t^2}{e_1^2s_{1t}} \end{smallmatrix}\right).
\end{equation}

\begin{figure}[t]
\centering
\contourlength{2pt}
\begin{tikzpicture} [scale=0.6] 
	\draw (0,0) arc(180:360:1.5) arc(0:120:2) arc(60:180:2) arc(180:360:1.5);
	\draw[->] (0,{sqrt(3)}) -- (0,-4) node[left] {$\tfrac{\alpha^{(0)}}{\alpha^{(1)}}\mathbb R^+$};
	\draw[->] (0:5) -- (180:5); \draw (182:5) node[left] {$-\mathbb R^+$};
	\draw[->] (-5:5) -- (175:5); \draw (173:5) node[left] {$-(1+\epsilon\tilde t)\mathbb R^+$};
	\filldraw (0,0) circle (2pt) (0,.3) node[right] {$0$};
	\draw (0,{sqrt(3)}) node[above] {$-\tfrac{\alpha^{(0)}}{\alpha^{(1)}}$};
	\draw[color=red, very thick, ->] (-4.2,-.2) node[below] {$\tilde N\tilde S_0$} (-4.2,-.3) -- (-4.2,.8); 
	\draw[color=red, very thick, ->] (4,.3) node[above] {$\tilde S_1$} -- (4,-.8); 
	\draw (-2.2,-.0) node[below] {\contour{white}{\color{red} $\bar NS_{01}S_{0t}$}}; \draw[color=red, very thick, ->]  (-2.2,-.3) -- (-2.2,.6); 
	\draw[color=red, very thick, ->] (1.8,.15) node[above] {$S_{10}S_{t0}$} (1.8,.3) -- (1.8,-.6); 
	\draw[color=red, very thick, ->] (-.5,1) node[left] {$S_{1t}\!\!$} -- (.5,1); 
	\draw[color=red, very thick, ->] (.4,-.5) node[right] {$\!\! S_{t1}$} -- (-.6,-.5); 
	\draw[color=red, very thick, ->] (.5,-3) node[right] {$\!\!\tilde S_{t}$} -- (-.5,-3); 
	\draw[color=red, very thick, ->] (-130:2.5) node[below] {$C_-$} -- +(70:1); 
	\draw[color=red, very thick, ->] (140:3) node[above] {$C_-\bar N\tilde N^{-1}$} (135:3) -- +(-60:1); 
	\draw[color=red, very thick, ->] (40:3) node[above] {$C_+$} (45:3) -- +(-120:1); 
	\draw[color=red, very thick, ->] (-40:2.8) node[below] {$\ C_+$} -- +(110:1); 
\end{tikzpicture}	
\caption{The connection matrices between different canonical solution bases of the confluent system \eqref{eq:wm-bsysttilde} with $\bar N$ \eqref{eq:wm-barN}, $S_{ij}$ \eqref{eq:wm-sij}, $\tilde S_i$, $\tilde N$ \eqref{eq:wm-tildeStokes} and $C_\pm$ \eqref{eq:wm-C}.}
\label{figure:wm-6}
\end{figure}

\begin{small}
\begin{lemma} The coefficients of the outer Stokes matrices are equal to
\begin{equation}\label{eq:wm-tildes}
\begin{aligned}
\tilde s_{0t}&=s_{0t}+ s_{01}s_{1t}, & \tilde s_{01}&=-i\tfrac{e_1}{e_t} s_{0t},\\
\tilde s_{t0}&=s_{t0}+\tfrac{s_{10}}{s_{1t}}, & \tilde s_{10}&=-i\tfrac{e_t}{e_1}\tfrac{s_{10}}{s_{1t}}.
\end{aligned}
\end{equation}
\end{lemma}
\begin{proof}
We have $\tilde S_0=\tilde N^{-1}C_-\bar NS_{01}S_{0t}\bar N^{-1}\tilde N(C_-)^{-1}$, $\tilde S_1=C_+S_{10}S_{t0}(C_+)^{-1}$, see Figure \ref{figure:wm-6}.
\end{proof}
\end{small}
\goodbreak

\begin{remark}\label{remark:wm-1}
The inner Stokes matrices $S_{ij}$ are determined only up to conjugation by diagonal matrices
$\left(\begin{smallmatrix} d_0 &  &  \\[0pt]  & d_t & \\[0pt] & & d_1 \end{smallmatrix}\right)$. 
This corresponds through \eqref{eq:wm-tildes} to conjugation of the outer Stokes matrices $\tilde S_{i}$ by  
$\left(\begin{smallmatrix} d_0 &  &  \\[0pt]  & d_t & \\[0pt] & & d_t \end{smallmatrix}\right)$.
\end{remark}

The monodromy matrix of an outer solution around the origin is given by
\begin{equation}\label{eq:wm-tildeMinfty}
	\tilde M_\infty^{-1}=\tilde S_1\tilde S_t\tilde N\tilde S_0=
	\mbox{\scriptsize$\begin{pmatrix}
			e_0^2 & e_0^2\tilde s_{0t} & e_0^2\tilde s_{01} \\[3pt]
			e_0^2\tilde s_{t0} & e_te_1 X_0+e_0^2\tilde s_{t0}\tilde s_{0t} & -ie_te_1+e_0^2\tilde s_{t0}\tilde s_{01} \\[3pt]
			e_0^2\tilde s_{10} & -ie_te_1+e_0^2\tilde s_{10}\tilde s_{0t} & e_0^2\tilde s_{10}\tilde s_{01} 
		\end{pmatrix}$},
\end{equation}
and we know that its eigenvalues are $1$, $e^{-2\pi i\kappa_1}=\frac{e_0e_te_1}{e_\infty}$ and $e^{-2\pi i\kappa_2}={e_0e_te_1e_\infty}$. 

%\begin{definition}
The new variables $\pX$ \eqref{eq:wm-UW} are defined by
\begin{equation}\label{eq:wm-XUW}
\pXone=X_0,\qquad 
\pXzero=e_0\tilde s_{10}\tilde s_{01}+e_0,\qquad 
\pXinf=ie_0\tilde s_{t0}\tilde s_{01}+\tfrac{e_te_1}{e_0}.
\end{equation}
They are invariant with respect to the conjugation of Remark \ref{remark:wm-1}.
%\end{definition}

Expressing the coefficients of the linear and the quadratic term of the characteristic polynomial of $\tilde M_\infty^{-1}$ \eqref{eq:wm-tildeMinfty}
gives
\begin{equation*}%\label{eq:wm-s1t0}
e_te_1\pXone+e_0\pXzero+ e_0^2\tilde s_{0t}\tilde s_{t0} = 1+\tfrac{e_0e_te_1}{e_\infty}+e_0e_te_1 e_\infty:=E, 
\end{equation*}
\begin{equation*}%\label{eq:wm-s0t1}
%\begin{split} %\mbox{\scriptsize$
e_0e_te_1\pXone\pXzero+e_0e_te_1\pXinf +ie_0^2e_te_1\tilde s_{0t}\tilde s_{10}=\tfrac{e_0e_te_1}{e_\infty}+e_0e_te_1e_\infty+e_0^2e_t^2e_1^2:=E'.
%$}
%\end{split}
\end{equation*}
Inserting these two identities into the identity
\begin{align*}
0&=e_0^2\cdot\tilde s_{0t}\tilde s_{10}\cdot\tilde s_{t0}\tilde s_{01}- 
e_0^2\cdot\tilde s_{0t}\tilde s_{t0}\cdot\tilde s_{10}\tilde s_{01}\\
&=-i\tilde s_{0t}\tilde s_{10}(e_0\pXinf+e_te_1)-\tilde s_{0t}\tilde s_{t0}(e_0\pXzero-e_0^2)=0,
\end{align*}
gives the Fricke relation \eqref{eq:wm-tildeF}
\begin{equation*}
0=\pXone\pXzero\pXinf+\pXzero^2+\pXinf^2-\pthetaone\pXone-\pthetazero\pXzero-\pthetainf\pXinf+\pthetat,
\end{equation*}
with
\begin{equation*}
\begin{aligned}
\pthetaone&=e_te_1,\qquad 
&\pthetazero&=\tfrac{e_te_1}{e_0}+\tfrac{E'}{e_0e_te_1}=a_\infty+e_te_1a_0,\qquad \\
\pthetat&=E+\tfrac{E'}{e_0^2}=1+e_te_1a_0 a_\infty+e_t^2e_1^2, \qquad
&\pthetainf&=e_0+\tfrac{E}{e_0}=a_0+e_te_1a_\infty.
\end{aligned}
\end{equation*}

The formulas of change of variables \eqref{eq:wm-UWXXX}, \eqref{eq:wm-XXXUW} of Theorem~\ref{theorem:wm-Phi} between $(X,\theta)$ and $(\pX,\ptheta)$ are obtained from \eqref{eq:wm-XUW}, \eqref{eq:wm-tildes} and \eqref{eq:wm-XXX}.
As remarked on page~\pageref{page:wm-line}, the singular line $L_0$ \eqref{eq:wm-lineL0} corresponds to $s_{1t}=0$, i.e. to the triviality of the Stokes matrix $S_{1t}$ \eqref{eq:wm-sij}.

\small

\goodbreak

\end{document}